\RequirePackage{fix-cm}
\documentclass[smallextended]{svjour3}       % onecolumn (second format)
\smartqed  % flush right qed marks, e.g. at end of proof

\usepackage[a4paper, total={7.0in, 10in}]{geometry}

\usepackage{graphicx}
\usepackage{amssymb}
\usepackage{isomath}% Pour utiliser, mathbold
\usepackage{mathtools}
\usepackage{tikz}
\usepackage{pgfplots}
\usetikzlibrary{arrows.meta}
  \pgfplotsset{compat=newest}
  \usetikzlibrary{plotmarks}
  \usepackage{grffile}

\pgfplotsset{every axis/.append style={
        scaled ticks = false, 
        tick label style={/pgf/number format/fixed}
    }
}

\usetikzlibrary{external}
\newlength\figureheight 
\newlength\figurewidth 
\usetikzlibrary{patterns}
\usetikzlibrary{plotmarks}
\usetikzlibrary{patterns}
\usepackage{tikz-3dplot}
\usepackage{multirow}
\usepackage{subfigure}
\usepackage{esint} % Pour intégrale double
\usepackage{multirow}
\usepackage{booktabs}
\usepackage{ulem}
\usepackage{cancel}
\usepackage[autostyle]{csquotes} 
\usepackage{stackengine} 
\usepackage{float}
\usepackage{setspace}
\usepackage{placeins}
\usepackage{lipsum}
\usepackage{amsfonts}
\usepackage{stmaryrd} % Pour la notation du saut
\usepackage{adjustbox}
\usepackage{hyperref}

\usepackage{algorithm}
\usepackage{algpseudocode}

\begin{document}

\title{A $P$-Adaptive Hermite Method for Nonlinear Dispersive Maxwell's Equations%\thanks{Grants or other notes
}

\titlerunning{A $P$-Adaptive Hermite Method for Nonlinear Dispersive Maxwell's Equations}        % if too long for running head

\author{Yann-Meing Law \and Zhichao Peng \and Daniel Appel\"{o} \and Thomas Hagstrom
}


\institute{Y.-M. Law \at
              Department of Mathematics and Industrial Engineering, Polytechnique Montr\'eal,  Montr\'eal, QC, Canada, 
              \email{yann-meing.law@polymtl.ca} \\
              Z. Peng \at 
              Department of Mathematics, Hong Kong University of Science and Technology, Hong Kong, China,
              \email{pengzhic@ust.hk} \\
              D. Appel\"{o} \at 
              Department of Mathematics, Virginia Tech, Blacksburg VA, USA,
              \email{appelo@vt.edu} \\
              T. Hagstrom \at 
              Department of Mathematics, Southern Methodist University, Dallas TX, USA,
              \email{thagstrom@smu.edu} 
}

\date{Received: date / Accepted: date}

\maketitle

\begin{abstract}

In this work, 
we introduce a novel Hermite method to handle Maxwell's equations for nonlinear dispersive media.
The proposed method achieves high-order accuracy and 
    is free of any nonlinear algebraic solver,
    requiring solving instead small local linear systems for which the dimension is independent of the order. 
The implementation of order adaptive algorithms is straightforward in this setting,
    making the resulting $p$-adaptive Hermite method appealing for the simulations of soliton-like wave propagation.

%Insert your abstract here. Include keywords, PACS and mathematical
%subject classification numbers as needed.
  % 35Q61 : Maxwell equations, 65M06 : Finite difference, 78M20 : FD for electromag, 78A45 : Diffraction, scattering, 65M70: Spectral, collocation and related methods, 78M22:Spectral, collocation and related methods applied to problems in optics and electromagnetic theory
\end{abstract}

\section{Introduction}

When light propagates through nonlinear optical media, 
    such as a Kerr-type medium, 
    the response of the medium depends on the intensity of the light in a nonlinear manner,
    giving rise to various intriguing optical phenomena such as third-order harmonic generation,
    four-wave mixing,
    and soliton propagation \cite{Agrawal2013}. 
In this situation,
    the linear and nonlinear electromagnetic responses of the medium can be modeled by Maxwell's equations with a nonlinear constitutive law relating the electric flux density to the electric field. 

Various numerical methods have been developed to simulate light propagation in optical 
    media with a nonlinear Kerr-type response.  
Finite-difference time-domain (FDTD) \cite{hile1996numerical,sorensen2005kink,greene2006general,jia2019new}, 
    finite-element time-domain (FETD) \cite{zhu2015novel,abraham2018convolution} 
    and energy stable mixed FE \cite{anees2020energy,huang2021time} methods 
    have been developed for Kerr media with or without a linear Lorentz response.
It is worth noting that FDTD \cite{Gilles2000}, 
    FETD \cite{abraham2019perfectly,li2023energy}
    and energy stable discontinuous Galerkin  \cite{bokil2017energy,Lyu2021,Lyu2022} methods are also able to treat nonlinear Raman scattering. 
Finally, 
    numerical methods for Kerr-Debye media %, which consider a finite relaxation time in the third-order nonlinear response, 
    have been considered in \cite{ziolkowski1993full,fisher2007efficient,huang2017second,peng2020asymptotic}.
In this work, we explore another avenue based on the Hermite method.
    
The Hermite method was introduced in \cite{Goodrich2005} for 
    linear hyperbolic problems with periodic boundary conditions.
The method is composed of a Hermite interpolation procedure in space and 
    a Taylor method in time to evolve the polynomial coefficients.
It achieves a $2m+1$ convergence rate when evolving, 
    in time, 
    the numerical solution 
    and its spatial derivatives through order $m$.
Remarkably, 
    the stability condition of the Hermite method is determined only
	by the largest wave speed of the problem and is independent of the order.
Moreover,  
    the implementation of order adaptive algorithms in space, 
	referred to as $p$-adaptivity, 
    is straightforward within the Hermite method framework \cite{Chen2012}. 
Note that dissipative and conservative Hermite methods for Maxwell's equations in linear dispersive media described by generalized Lorentz models were designed in \cite{Hagstrom2019,Appelo2024}.

Here we propose a $p$-adaptive Hermite method for the 
    simulation of light propagation 
    in nonlinear optical media governed by Maxwell's equations with nonlinear Kerr effect, 
    nonlinear Raman scattering and the linear Lorentz response \cite{Agrawal2013}.
One approach to model these dispersion effects is to utilize a set of ordinary differential equations, 
    called auxiliary differential equations (ADEs) \cite{Joseph1991}.
Remarkably, for this type of nonlinear optical media, 
    the proposed method is devoid of nonlinear solvers. 

The paper is organized as follows.
We first describe the governing equations for which we seek a numerical solution in Section~\ref{sec:problem_definition}.
The Hermite method for nonlinear dispersive media in one space dimension is described in detail in Section~\ref{sec:Hermite_method_1D}.
We extend the proposed method in higher space dimensions in Section~\ref{sec:Hermite_method_multi_dimensions} and describe a $p$-adaptive algorithm,
	introduced in \cite{Chen2012}, 
	particularly adapted to the Hermite method framework, in Section~\ref{sec:p_adaptive_algo}.
The properties of the Hermite method for nonlinear dispersive media
	are discussed in Section~\ref{sec:analysis_method}.
Finally, 
	in Section~\ref{sec:numerical_examples}, 
	we demonstrate the performance of the proposed method through a series of numerical examples in 1-D and 2-D. %to verify the properties of the Hermite method.
	
\section{Problem Definition} \label{sec:problem_definition}

We seek numerical solutions to Maxwell's equations 
\begin{equation} \label{eq:Maxwell_3D}
	\begin{aligned}
		\mu \partial_t \hat{\mathbold{H}} + \nabla\times\hat{\mathbold{E}} =&\,\, 0, \\
		\partial_t \hat{\mathbold{D}} - \nabla\times \hat{\mathbold{H}} = &\,\, 0, 
	\end{aligned}
\end{equation}
	in a domain $\Omega \subset \mathbb{R}^d$ $(d=1,2,3)$ and a time interval $I=[t_0,t_f]$.
Here $\hat{\mathbold{H}}$ is the magnetic field, 
	$\hat{\mathbold{E}}$ is the electric field,
	$\hat{\mathbold{D}}$ is the electric flux density, 
    and $\mu$ is the magnetic permeability of free space.
	
For Kerr-Lorentz-Raman type nonlinear optical media,
	the constitutive law that relates the electric flux density and the electric field is given by 
\begin{equation} \label{eq:constitutive_law_DEPQ_3D}
	\hat{\mathbold{D}}  = \epsilon(\epsilon_\infty \hat{\mathbold{E}}+\hat{\mathbold{P}}+a(1-\theta)\|\hat{\mathbold{E}}\|^2\hat{\mathbold{E}}+a\theta \hat{Q}\hat{\mathbold{E}}).
\end{equation}
Here $\| \cdot\|$ denotes the Euclidean norm,
	$\epsilon$ is the electric permittivity of free space, 
    and $\epsilon_\infty$ is the relative electric permittivity when the frequency goes to infinity.
	
The constitutive law \eqref{eq:constitutive_law_DEPQ_3D} can model linear and nonlinear responses of the media to the electromagnetic fields.
The term $\epsilon\epsilon_\infty\hat{\mathbold{E}}$ describes its linear instantaneous response  
	while $\epsilon\hat{\mathbold{P}}$ describes its delayed linear response modeled by the Lorentz effect:
\begin{equation} \label{eq:second_order_ODE_P_3D}
	\partial_t^2 \hat{\mathbold{P}} + \gamma \partial_t \hat{\mathbold{P}} + \omega_0^2 \hat{\mathbold{P}} = \omega_p^2 \hat{\mathbold{E}}.
\end{equation}
Here $\omega_0$ is the resonance frequency of the medium, 
	$\omega_p = \sqrt{\epsilon_s-\epsilon_\infty}\omega_0$ is the plasma frequency,
	$\epsilon_s$ is the zero frequency relative permittivity, 
    and $\gamma$ is a damping constant.   
	
We consider a two-component nonlinear response of the media to the electromagnetic fields. 
The instantaneous response, 
	called Kerr effect, 
	is described by $\epsilon a(1-\theta)\|\hat{\mathbold{E}}\|^2\hat{\mathbold{E}}$ 
	and the delayed Raman vibrational molecular response is described by $\epsilon a\theta \hat{Q}\hat{\mathbold{E}}$.
Here $a$ is a third-order coupling constant and $\theta$ is the relative strength between the Kerr and Raman effects.
The natural molecular vibration within the material is modeled by 
\begin{equation} \label{eq:second_order_ODE_Q_3D}
	\partial_t^2 Q + \gamma_v\partial_t Q+\omega_v^2Q = \omega_v^2 \|\hat{\mathbold{E}}\|^2,
\end{equation}
	where $\omega_v$ is the resonance vibration frequency and $\gamma_v$ is a damping constant. 

By introducing the auxiliary variables $\hat{\mathbold{J}}$ and 
    $\hat{\mathbold{S}}$, 
    we rewrite the equations \eqref{eq:second_order_ODE_P_3D} and \eqref{eq:second_order_ODE_Q_3D} as a first-order system of 
    ordinary differential equations (ODEs)
\begin{equation} \label{eq:first_order_ODEs_3D}
	\begin{aligned}
		\partial_t \hat{\mathbold{P}}  =&\,\, \hat{\mathbold{J}}, \\
		\partial_t \hat{\mathbold{J}} + \gamma\hat{\mathbold{J}} + \omega_0^2 \hat{\mathbold{P}}  =&\,\,  \omega_p^2 \hat{\mathbold{E}}, \\
		\partial_t \hat{Q} =&\,\, \hat{S}, \\
		\partial_t \hat{S} + \gamma_v\hat{S}+\omega_v^2\hat{Q} =&\,\, \omega_v^2\|\hat{\mathbold{E}}\|^2.
	\end{aligned}
\end{equation}
The initial conditions are given by $\hat{\mathbold{H}}(\mathbold{x},t_0) = \hat{\mathbold{H}}_0(\mathbold{x})$,  
	$\hat{\mathbold{E}}(\mathbold{x},t_0) = \hat{\mathbold{E}}_0(\mathbold{x})$,
	$\hat{\mathbold{P}}(\mathbold{x},t_0) = \hat{\mathbold{P}}_0(\mathbold{x})$,
	$\hat{\mathbold{J}}(\mathbold{x},t_0) = \hat{\mathbold{J}}_0(\mathbold{x})$,
	$\hat{Q}(\mathbold{x},t_0) = \hat{Q}_0(\mathbold{x})$ and 
	$\hat{S}(\mathbold{x},t_0) = \hat{S}_0(\mathbold{x})$ in $\Omega$.

As in \cite{Lyu2021,Lyu2022}, 
	the energy of Maxwell's equations for nonlinear dispersive media is defined as 
\begin{equation}
	\begin{aligned}
	\mathcal{E} = \int_{\Omega} \left( \frac{\mu}{2} \|\hat{\mathbold{H}}\|^2 + \frac{\epsilon\epsilon_\infty}{2}\|\hat{\mathbold{E}}\|^2 + \frac{\epsilon}{2\omega_p^2}\|\hat{\mathbold{J}}\|^2 + \frac{\epsilon\omega_0^2}{2\omega_p^2}\|\hat{\mathbold{P}}\|^2 + \frac{a\epsilon\theta}{4\omega_p^2}\hat{\sigma}^2  \right. & &\\ \left. 
	+ \frac{a\epsilon\theta}{2}\hat{Q}\|\hat{\mathbold{E}}\|^2 + \frac{3a\epsilon(1-\theta)}{4}\|\hat{\mathbold{E}}\|^4 + \frac{a\epsilon\theta}{4}\hat{Q}^2 \right) d\Omega. & &
	\end{aligned}
\end{equation}
Note that the energy decreases with time 
    when periodic boundary conditions are considered 
    and is positive for $\theta \in [0,3/4]$.
Furthermore, 
	the energy is conserved when the damping terms in equations \eqref{eq:second_order_ODE_P_3D} and \eqref{eq:second_order_ODE_Q_3D} are neglected,
	that is $\gamma=\gamma_v=0$. 
	
\section{Hermite Method in One Space Dimension} \label{sec:Hermite_method_1D}

In this section, 
	we focus on a one-dimensional simplification of the problem defined in Section~\ref{sec:problem_definition},
	given by 
\begin{subequations} \label{eq:1D_equations}
	\begin{align}
		\mu \partial_t \hat{H} - \partial_x \hat{E} =&\,\, 0, \label{eq:Faraday_1D}\\
		\partial_t \hat{D} - \partial_x \hat{H} = &\,\, 0, \label{eq:Ampere_Maxwell_1D}\\
		\partial_t \hat{P} - \hat{J} =&\,\, 0, \label{eq:P_ODE_1D}\\
		\partial_t \hat{J} + \gamma\hat{J} + \omega_0^2 \hat{P} - \omega_p^2 \hat{E} =&\,\, 0, \label{eq:J_ODE_1D}\\
		\partial_t \hat{Q} - \hat{S} =&\,\, 0, \label{eq:Q_ODE_1D}\\
		\partial_t \hat{S} + \gamma_v\hat{S}+\omega_v^2(\hat{Q}-\hat{E}^2) =&\,\, 0, \label{eq:S_ODE_1D}\\
		\hat{D} -\epsilon(\epsilon_\infty \hat{E}+\hat{P}+\theta_K\hat{E}^3+\theta_R \hat{Q}\hat{E}) =&\,\, 0, \label{eq:D_nonlinear_1D}
	\end{align}
\end{subequations}	
	in a domain $\Omega = [x_\ell, x_r]$ and a time interval $I=[t_0,t_f]$.
Here $\theta_K = a(1-\theta)$ and $\theta_R= a\theta$.
We consider periodic boundary conditions.

The Hermite method requires a mesh staggered in space and time. 
The primal mesh is defined as
\begin{equation}
	x_i = x_\ell + i\Delta x, \quad i = 0,\dots,N_x, \quad \Delta x = \frac{x_r-x_\ell}{N_x}.
\end{equation}
The numerical solution on the primal mesh is centered at times
\begin{equation}
	t_n = t_0 + n\Delta t, \quad n=0,\dots,N_t, \quad \Delta t = \frac{t_f-t_0}{N_t}.
\end{equation}
Here $N_x$ is the number of cells and $N_t$ is the number of time steps to reach the final time $t_f$.
The dual mesh nodes are defined at the cell centers
\begin{equation}
	x_{i+1/2} = x_\ell + (i+1/2)\Delta x, \quad i=0,\dots,N_x-1,
\end{equation} 
	and times
\begin{equation}
	t_{n+1/2} = t_0 + (n+1/2)\Delta t, \quad n=0,\dots,N_t-1.
\end{equation}

The Hermite method is composed of three parts:
\begin{itemize}
	\item[1.] \underline{Hermite interpolation:}
	
	Let us assume sufficiently accurate approximations of the unknowns variables and their derivatives through order $m$ on the 
		primal mesh are known at time $t_n$.
	For each variable and each cell, 
		we compute the unique degree $2m+1$ Hermite interpolant matching the data for each field.
		For example we use $\partial_x^k \hat{H}(x_i,t_n)$, $\partial_x^k \hat{H}(x_{i+1},t_n)$
		$k=0,\dots,m$, to construct the Hermite interpolant approximating the magnetic field in  
		the cell $[x_i,x_{i+1}]$.
		 
	\item[2.] \underline{System of ordinary differential equations for the polynomial coefficients:}
	
	A system of ordinary differential equations for the polynomial coefficients is used to evolve all the unknown variables
		and their spatial derivatives through order $m$ on the dual mesh to time $t_{n+1/2}$.
	Considering a cell and a given Hermite interpolant of each variable on this cell,
		we identify the derivatives of the variable as scaled coefficients of the Hermite polynomial at the cell center.
	Noticing that the scaled coefficients are actually time-dependent functions and substituting the variables by their Hermite interpolant 
		in system \eqref{eq:1D_equations}, 
		we obtain a local system of ODEs for the coefficients on each cell.
	
	\item[3.] \underline{Time evolution:}
	
	For each cell, 
		we evolve locally the polynomial coefficients (and therefore the function value of the unknown variables and their spatial derivatives through order $m$) 
		to time $t_{n+1/2}$ by solving numerically the system of ODEs.
	Note that local sub-steps in time might be required to match the accuracy of the Hermite method in space. 
	
\end{itemize}

Figure~\ref{fig:hermite_method_1d} illustrates the different steps of the Hermite method for executing a complete time step, 
    that is evolving the data from $(x_{i+1},t_{n})$ to $(x_{i+1},t_{n+1})$.
We now provide more details on each part of the proposed Hermite method.
To do so, 
	we consider the evolution of the data for the first half time step,  
	from the initial time $t_0$ to $t_{1/2}$,
	as an example.
The subsequent evolutions in time of the numerical solution follows analogously. 
%***************************************
\begin{figure}   
	\centering
		\includegraphics[width=3.8in,trim={0.0cm 0cm 0cm 0cm},clip]{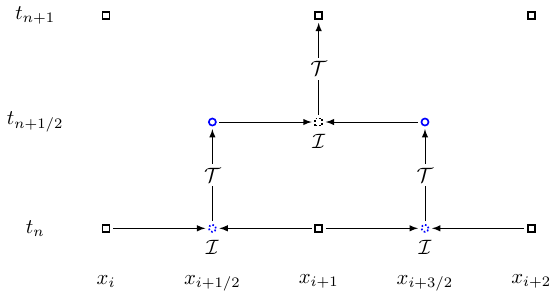} 
       \caption{Illustration of the Hermite method to evolve the data at the primal grid node $x_{i+1}$ from $t_{n}$ to $t_{n+1}$. Here the Hermite interpolation and time-evolution procedures are denoted respectively by $\mathcal{I}$ and $\mathcal{T}$.}
       \label{fig:hermite_method_1d}
\end{figure}
%***************************************
\begin{remark}
	In this work, 
		we focus on the Hermite method with periodic boundary conditions and therefore the enforcement of other types of boundary conditions will not be addressed. 
	As it is usually the case for Hermite methods,
		the treatment of boundary conditions is challenging because the variables and their spatial derivatives through order $m$ are also required on the boundary.
	That being said, 
		the boundary conditions could be handled by a discontinuous Galerkin (DG) method, 
		leading to a hybrid DG-Hermite method, 
		similar to what it is done in \cite{Chen2010} for linear Maxwell's equations.  
	For the Hermite method presented here, 
		DG methods proposed in \cite{Bokil2017,Lyu2021,Lyu2022} could be potential candidates and will be explored in future work. 
\end{remark}

Let us assume that the spatial derivatives through order $m$ of all variables are known at the initial 
	time $t_0$ on the primal mesh.
Note that, 
	if the spatial derivatives are too expensive to be computed explicitly,
	the initial solution can be projected on a polynomial space of degree at least $2m+1$ to 
	maintain accuracy. 
The derivatives sought at the initial time are then computed from the polynomials resulting from the projection.

\subsection{Hermite Interpolation}

For each variable and each cell $[x_i,x_{i+1}]$ on the primal mesh, 	
	we seek the $2m+1$ degree Hermite interpolant $f_i^{\mathcal{H}}(x)$ satisfying 
\begin{equation}
	\frac{d^{k} f_{i+1/2}^{\mathcal{H}}(x_i)}{d x^{k}} = \frac{d^{k} \hat{f}(x_i,t_0)}{d x^{k}}, \quad
	\frac{d^{k} f_{i+1/2}^{\mathcal{H}}(x_{i+1})}{d x^{k}} = \frac{d^{k} \hat{f}(x_{i+1},t_0)}{d x^{k}}, \quad
	 k = 0,\dots,m.
\end{equation}
Here $\hat{f}$ is a variable that we seek to approximate. 
We then obtain a polynomial approximating each variable on the cell $[x_i, x_{i+1}]$,
\begin{equation}
	\hat{f}(x,t)|_{t=t_0} \approx f_{i+1/2}^{\mathcal{H}}(x) = \sum_{k=0}^{2m+1} f_{k,i}(t)|_{t=t_0} \xi^k.
\end{equation}
Here $\xi=\frac{x-x_{i+1/2}}{\Delta x}$.
Note that the Hermite interpolants are centered at the cell center $x_{i+1/2}$ and the coefficients $f_{k,i}$ are time-dependent functions.

Identifying the spatial derivatives at the cell center as scaled coefficients of $f_{i+1/2}^{\mathcal{H}}(x)$, 
	we therefore seek a system of ODEs for the coefficients of the cell polynomial to evolve the data 
	on the dual mesh at time $t_{1/2}$.
For simplicity, 
	we omit the subscript $i$ associated with the cell in the remaining parts of this section.

\subsection{System of Ordinary Differential Equations for the Polynomial Coefficients} \label{sec:syst_ODEs_coeff}
 
We use a one-step time-stepping method instead of a Taylor method,
    as in the original Hermite method \cite{Goodrich2005}, 
    to evolve the cell polynomial coefficients.
This approach only requires a single time derivative 
	per polynomial coefficient and is therefore more efficient for nonlinear problems \cite{Hagstrom2007,Hagstrom2015}.
Note that the nonlinear problems considered in previous works involved 
    spatial derivative terms and were 
	of the form $\partial_t u + \partial_x f(u) = 0$,
    where $f(u)$ is a nonlinear function. 
The problem of interest here has the form 
    $\partial_t f(u) + \partial_x u = 0$,
	requiring therefore a novel formulation in the Hermite setting.
    
\subsubsection{System of ODEs for the Magnetic Field Coefficients}

We first substitute the electromagnetic fields by their Hermite interpolants $f^{\mathcal{H}}(x)$ in Faraday's law \eqref{eq:Faraday_1D} 
\begin{equation}
		\partial_t H^{\mathcal{H}} = \partial_x E^{\mathcal{H}}.
\end{equation}
Expanding the cell polynomials 
	we obtain
\begin{equation}
		\mu\sum_{k=0}^{2m+1} H'_k(t) \xi^k = \sum_{k=0}^{2m} \frac{(k+1)E_{k+1}(t)}{\Delta x} \xi^k. \\
\end{equation}
Here and in the following, 
    we use the Lagrange's notation for the time derivative instead of $\tfrac{d}{dt}\cdot$ for legibility.

By matching the coefficients of the polynomials in the right and left hand sides, 
	we obtain a system of ODEs to evolve the polynomial coefficients 
\begin{equation} \label{eq:ode_coeff_magnetic_field_part_1}
		 H'_k = \frac{(k+1)}{\mu\Delta x}E_{k+1}, \quad k = 0,\dots,2m.
\end{equation}
	with 
\begin{equation} \label{eq:ode_coeff_magnetic_field_part_2}
	H'_{2m+1} = 0.
\end{equation}

\subsubsection{Systems of ODEs for the Auxiliary Variables Coefficients}

Applying the same procedure as for the magnetic field to the ODEs for $\hat{P}$, $\hat{J}$ and $\hat{Q}$ in \eqref{eq:1D_equations} leads to 
\begin{equation} \label{eq:ode_coeff_auxiliary_P_J_Q}
	 P'_k = J_{k}, \quad  J'_k = -\gamma J_k - \omega_0^2 P_k + \omega_p^2 E_k, \quad Q'_k = S_{k},  
\end{equation}
for $k = 0,\dots,2m+1$.
The ODE \eqref{eq:S_ODE_1D} for the auxiliary variable $\hat{S}$ contains the nonlinear term $\hat{E}^2$. 
In this situation, 
	we compute the square of the electric field polynomial and truncate it to obtain 
	a polynomial of degree $2m+1$ approximating $\hat{E}^2$.
We then have  
\begin{equation}
	\hat{E}^2 \approx (E^{\mathcal{H}})^2 = \sum_{k=0}^{4m+2} \phi_k \xi^k = \sum_{k=0}^{2m+1}\phi_k \xi^k  + \mathcal{O}(\Delta x^{2m+2}), \quad \phi_k = \sum_{p=0}^{k} E_p E_{k-p}.
\end{equation}
The resulting ODEs for the coefficients of the Hermite interpolant $S^{\mathcal{H}}$ is 
\begin{equation} \label{eq:ode_coeff_auxiliary_S}
	  S'_k = -\gamma_v S_k-\omega_v^2Q_k+\omega_v^2\phi_k, \quad k = 0,\dots,2m+1.
\end{equation} 

\subsubsection{A Recursive System of ODEs for the Electric Field Coefficients}
Using equations \eqref{eq:P_ODE_1D} and \eqref{eq:D_nonlinear_1D},
    we rewrite Amp\`{e}re-Maxwell's law \eqref{eq:Ampere_Maxwell_1D} as
\begin{equation} \label{eq:AmpereMaxwell_1}
    \partial_t (\epsilon_\infty \hat{E}+\theta_K\hat{E}^3+\theta_R \hat{Q}\hat{E}) +  \hat{J} - \frac{\partial_x\hat{H}}{\epsilon} = 0.
\end{equation}
We first compute a degree $2m+1$ polynomial approximating the term $\hat{E}^3$ by truncating $(E^{\mathcal{H}})^3$ and obtain
\begin{equation} \label{eq:approx_term_E3}
	\hat{E}^3 \approx (E^{\mathcal{H}})^3 = \sum_{k=0}^{6m+3} \psi_k \xi^k = \sum_{k=0}^{2m+1}\psi_k \xi^k  + \mathcal{O}(\Delta x^{2m+2}), \quad \psi_k = \sum_{p=0}^{k}E_p\phi_{k-p}. 
\end{equation}
Substituting the variables by their interpolants and matching their coefficients lead to 
\begin{equation} \label{eq:coeff_AmpereMaxwell_1}
    \bigg(\epsilon_\infty E_k+\theta_K\psi_k+\theta_R \sum_{p=0}^k E_p Q_{k-p} \bigg)' +  J_k - \frac{(k+1)}{\epsilon\Delta x}H_{k+1} = 0, 
\end{equation}
    for $k = 0, \dots, 2m$, 
    and 
\begin{equation} \label{eq:coeff_AmpereMaxwell_2}
     \bigg(\epsilon_\infty E_{2m+1}+\theta_K\psi_{2m+1}+\theta_R \sum_{p=0}^{2m+1} E_p Q_{2m+1-p}\bigg)' + J_{2m+1} = 0.
\end{equation}
Since the coefficients are time-dependent functions, 
    we have 
\begin{equation} 
    \epsilon_\infty E'_k+\theta_K \psi'_k+\theta_R  \bigg(\sum_{p=0}^k E_p Q_{k-p}\bigg)' +  J_k - \frac{(k+1)}{\epsilon\Delta x}H_{k+1} = 0.
\end{equation}
We now seek to isolate $E_k'$.
Computing the time derivative of $\psi_k$ leads to
\begin{equation}
    \begin{aligned}
\psi'_k =&\,\,  \bigg(\sum_{p=0}^k E_p \phi_{k-p}\bigg)'\\
    =&\,\,  \sum_{p=0}^k (E'_p \phi_{k-p} + E_p \phi'_{k-p}) \\
    =&\,\, E'_k \phi_0 + \sum_{p=0}^{k-1} E'_p \phi_{k-p} +  \sum_{p=0}^{k} \phi'_{p} E_{k-p} \\
    =&\,\, E'_k \phi_0 + \phi'_k E_0 + \sum_{p=0}^{k-1}(E'_p \phi_{k-p}+\phi'_p E_{k-p}).
    \end{aligned}
\end{equation}
Knowing that 
\begin{equation}
     \phi'_k = \bigg(\sum_{p=0}^k E_pE_{k-p}\bigg)' 
        = 2\sum_{p=0}^k E'_p E_{k-p} 
        = 2 E'_k E_0 + 2 \sum_{p=0}^{k-1} E'_p E_{k-p},
\end{equation}
    and using $\phi_0 = E_0^2$,
    we obtain 
\begin{equation} \label{eq:dt_psi_1D}
    \begin{aligned}
    \psi'_k =&\,\, E'_k \phi_0 +  2 E'_k E_0^2 + \sum_{p=0}^{k-1}(E'_p \phi_{k-p}+ (\phi'_p +2 E'_pE_0) E_{k-p} ) \\
    =&\,\, 3  E'_k \phi_0 + \sum_{p=0}^{k-1}(E'_p \phi_{k-p}+ (\phi'_p +2E'_pE_0) E_{k-p} ).
    \end{aligned}
\end{equation}
Note that the terms in the summation only involve $E'_p$, $p=0,1,\dots,k-1$.

Computing the time derivative involving $E_pQ_{k-p}$ and $Q'_k = S_k$ leads to 
\begin{equation}  \label{eq:dt_sum_EpQkmp_1D}
    \bigg(\sum_{p=0}^k E_pQ_{k-p}\bigg)' = \sum_{p=0}^k ( E'_p Q_{k-p} + E_pQ'_{k-p}) 
    = E'_k Q_0 + \sum_{p=0}^{k-1} E'_p Q_{k-p} + \sigma_k. %\sum_{p=0}^k E_p S_{k-p}.
\end{equation}
Here 
\begin{equation}
	E^{\mathcal{H}}S^{\mathcal{H}} = \sum_{k=0}^{2m+1} \sigma_k \xi^k + \mathcal{O}(\Delta x^{2m+2}), \quad \sigma_k = \sum_{p=0}^k E_pS_{k-p}.
\end{equation}

Substituting \eqref{eq:dt_psi_1D} and \eqref{eq:dt_sum_EpQkmp_1D} into \eqref{eq:coeff_AmpereMaxwell_1} and \eqref{eq:coeff_AmpereMaxwell_2} gives 
\begin{equation} \label{eq:ode_coeff_E_1}
    M E'_k  =  \frac{(k+1)}{\epsilon\Delta x}H_{k+1}-  J_k -\theta_R\sigma_k- \sum_{p=0}^{k-1}\big(\theta_K(E'_p \phi_{k-p}  
    + (\phi'_p +2 E'_pE_0) E_{k-p}) + \theta_R E'_p Q_{k-p}\big),
\end{equation}
    for $k=0,\dots,2m$, 
    and 
\begin{equation} \label{eq:ode_coeff_E_2}
    M E'_k  =  -J_k -\theta_R\sigma_k  - \sum_{p=0}^{k-1}\big(\theta_K(E'_p \phi_{k-p} 
    + (\partial_t\phi_p +2 E'_pE_0) E_{k-p}) + \theta_R E'_p Q_{k-p}\big), 
\end{equation}
    when $k=2m+1$.
Here 
\begin{equation} \label{eq:matrix_M_1D}
	M = \epsilon_\infty+3\theta_K\phi_0 + \theta_R Q_0.
\end{equation}

\subsubsection{The Complete Set of ODEs for the Polynomial Coefficients}

In summary, 
    given the Hermite polynomial approximations of the unknown variables,
    the complete system of ODEs for the polynomial coefficients is given by 
\begin{equation} \label{eq:complete_syst_ODE_1d}
    \left\{ 
          \begin{aligned}
              &\,\, H'_k = \frac{(k+1)}{\mu\Delta x}E_{k+1}, \quad k = 0,\dots,2m, \\
              &\,\, H'_{2m+1} = 0, \\
              &\,\, P'_k = J_{k}, \quad k = 0,\dots,2m+1, \\
              &\,\, J'_k = -\gamma J_k - \omega_0^2 P_k + \omega_p^2 E_k, \quad k = 0,\dots,2m+1, \\
              &\,\, Q'_k = S_{k},\quad k = 0,\dots,2m+1, \\ 
              &\,\, S'_k = -\gamma_v S_k-\omega_v^2Q_k+\omega_v^2\phi_k, \quad k = 0,\dots,2m+1, \\
              &\,\, M E'_k = \frac{(k+1)}{\epsilon\Delta x}H_{k+1}-  J_k -\theta_R\sigma_k - \sum_{p=0}^{k-1}\big(\theta_K(E'_p \phi_{k-p} \\ 
    &\,\, \qquad + (\phi'_p +2E'_pE_0) E_{k-p}) + \theta_R E'_p Q_{k-p}\big), \quad k = 0,\dots,2m, \\
    &\,\, ME'_{2m+1} = -J_{2m+1} -\theta_R\sigma_{2m+1}  - \sum_{p=0}^{2m}\big(\theta_K(E'_p \phi_{2m+1-p}  \\ 
    &\,\, \qquad+ (\phi'_p +2E'_pE_0) E_{2m+1-p}) + \theta_R E'_p Q_{2m+1-p}\big),
          \end{aligned} \right. 
\end{equation}
with 
\begin{equation}
    \phi_k =\sum_{p=0}^{k} E_p E_{k-p}, \quad 
    \phi'_k = 2\sum_{p=0}^k E'_p E_{k-p} \quad \mbox{and} \quad M = \epsilon_\infty+3\theta_K\phi_0 + \theta_R Q_0.
\end{equation}
Note that we can solve for the time derivatives recursively in the system of ODEs \eqref{eq:complete_syst_ODE_1d}, 
    thereby avoiding the need for nonlinear solvers. 

%%%%%%%%%%%%%%%%%%
\subsection{Time Evolution} \label{sec:time_evolution_1D}

We now seek a numerical solution of the unknown variables and their derivatives through order $m$ on the dual mesh at time $t_{1/2}$.
To do so, 
	we solve numerically the system of ODEs for the polynomial coefficients using explicit Runge-Kutta (RK) methods.
Here we choose a fifth-order Dormand-Prince RK method with a local fixed number of time sub-steps \cite{Hairer1993}, but any other single step method could be employed.
In the following,
	we briefly describe this method in the Hermite setting.

We first rewrite the system of ODEs for polynomial coefficients as
\begin{equation} \label{eq:ODE_system}
	\mathbold{c}' = \mathbold{F}(t,\mathbold{c}).
\end{equation}
Here the vector $\mathbold{c}$ contains all the polynomial coefficients. 
From the Hermite interpolation procedure, 
we know the initial coefficients $\mathbold{c}$ at time $t_0$. 

To evolve the polynomial coefficients from $t_0$ to $t_{1/2}$, 
	we use local time sub-steps on each cell. 
For a fixed number of sub-steps $N_s$ to reach $t_{1/2}$, 
	we define the times
\begin{equation}
	\tau_s = t_0 + s\Delta \tau, \quad s=0,\dots,N_\tau,  \quad	\Delta \tau = \frac{t_{1/2}-t_0}{N_{\tau}}=\frac{\Delta t}{2N_{\tau}}.
\end{equation}
The approximations of the function value of the unknown variables and their derivatives through order $m$ are identified as scaled coefficients on the dual mesh.

To complete the time step, 
	we repeat the same procedure but this time from the dual mesh at time $t_{1/2}$ to the primal mesh at $t_1$.

\begin{remark}
The Hermite interpolation procedure and the time evolution of the coefficients on each cell can be done locally, 
making it suitable for GPU implementation \cite{Vargas2017,Appelo2018,Vargas2019}.
\end{remark}

\section{Extension to Multi-Dimensions} \label{sec:Hermite_method_multi_dimensions} 

In this section, 
	we extend the proposed Hermite method to higher dimensions.
The Hermite interpolation procedure is easily generalized in multi-dimensions 
	using tensor products of one space dimensional Hermite interpolants and we refer the reader to \cite{Goodrich2005} for more detail. 
The time evolution of the polynomial coefficients is performed using the same approach described in subsection \ref{sec:time_evolution_1D} 
	for the one dimensional case.
Here we focus on the system of ODEs to evolve in time the cell polynomial coefficients. 
The two dimensional case is described in detail.
We also provide the recursive system of ODEs for the electric field coefficients in the three dimensional case, 
	the derivation of the ODEs for the remaining coefficients should be straightforward. 
	
\subsection{Two Dimensional Case}

We now seek a numerical solution of the simplification of Maxwell's equations in 2-D  
\begin{equation} \label{eq:PDE_system_2D}
	\begin{aligned}
		\partial_t \hat{D}_x - \partial_y \hat{H}_z = &\,\, 0, \\
		\partial_t \hat{D}_y + \partial_x \hat{H}_z = &\,\, 0, \\
		\mu \partial_t \hat{H}_z - \partial_y \hat{E}_x + \partial_x \hat{E}_y =&\,\, 0, 
	\end{aligned}
\end{equation} 
	in a domain $\Omega = [x_\ell, x_r]\times[y_b,y_t]$ and a time interval $I=[t_0,t_f]$.
The system of ODEs \eqref{eq:first_order_ODEs_3D} is simplified to
\begin{equation} \label{eq:ODE_system_2D}
	\begin{aligned}
		&\partial_t \hat{P}_x - \hat{J}_x =  0, \quad \partial_t \hat{J}_x + \gamma \hat{J}_x + \omega_0^2 \hat{P}_x - \omega_p^2 \hat{E}_x = 0,\\
		&\partial_t \hat{P}_y - \hat{J}_y =  0, \quad \partial_t \hat{J}_y + \gamma \hat{J}_y + \omega_0^2 \hat{P}_y - \omega_p^2 \hat{E}_y = 0,\\
		&\partial_t \hat{Q} - \hat{S} = 0, \quad \partial_t \hat{S} + \gamma_v\hat{S}+\omega_v^2(\hat{Q}-\hat{E}_x^2-\hat{E}_y^2) = 0, 
	\end{aligned}
\end{equation}	
	and the nonlinear constitutive law \eqref{eq:constitutive_law_DEPQ_3D} becomes  
\begin{equation} \label{eq:nonlinear_constitutive_law_2D}
	\begin{aligned}
		\hat{D}_x -\epsilon(\epsilon_\infty \hat{E}_x+\hat{P}_x+\theta_K\hat{E}_x(\hat{E}_x^2+\hat{E}_y^2)+\theta_R \hat{Q}\hat{E}_x) =&\,\, 0, \\
		\hat{D}_y -\epsilon(\epsilon_\infty \hat{E}_y+\hat{P}_y+\theta_K\hat{E}_y(\hat{E}_x^2+\hat{E}_y^2)+\theta_R \hat{Q}\hat{E}_y) =&\,\, 0.
	\end{aligned}
\end{equation}

As in the one dimensional case, 
	the mesh is staggered in space and time. 
We focus on the spatial component of the mesh since the time discretization remains the same in higher dimensions.
The primal mesh is defined as
\begin{equation}
	(x_i,y_j) = (x_\ell+i\Delta x,y_b+j\Delta y), \quad i=0,\dots, N_x, \quad j=0,\dots,N_y.
\end{equation}
Here $N_x$ and $N_y$ are the number of cells in the $x$ and $y$ directions, 
	$\Delta x = \frac{x_r-x_\ell}{N_x}$ and $\Delta y = \frac{y_t-y_b}{N_y}$.
The dual mesh is given by 
\begin{equation}
	(x_{i+1/2},y_{j+1/2}) = (x_\ell+(i+1/2)\Delta x,y_b+(j+1/2)\Delta y), 
\end{equation}
	for $i = 0,\dots,N_x-1$, 
	$j=0,\dots,N_y-1$,
	which corresponds to the cell centers of the primal mesh. 

\subsubsection{System of ODEs for the Polynomial Coefficients in Two Dimensions} \label{sec:syst_ODEs_coeff_2D}

We write the Hermite interpolants of all variables in the form 
\begin{equation} \label{eq:Hermite_interpolant_2D}
	f^{\mathcal{H}}_{i+1/2,j+1/2}(x,y) = \sum_{k=0}^{2m+1}\sum_{\ell=0}^{2m+1} f_{k,\ell} \xi^k \eta^\ell.
\end{equation}
Here $\xi = \frac{x-x_{i+1/2}}{\Delta x}$ and $\eta = \frac{y-y_{j+1/2}}{\Delta y}$.
For simplicity, 
	we omit the subscripts $i$ and $j$ associated with the cell in the remaining part of this section.
	
The system of ODEs for the polynomial coefficients is obtained by substituting the variables by polynomials 
	in systems \eqref{eq:PDE_system_2D} and \eqref{eq:ODE_system_2D}, 
	and matching the coefficients, 
	as in subsection \ref{sec:syst_ODEs_coeff} for the one dimensional case.
We then have
\begin{equation}
    \begin{aligned}
    % Faraday law
    &\left\{ 
          \begin{aligned}
              &\,\, H'_{z_{k,\ell}} = \frac{(\ell+1)}{\mu \Delta y} E_{x_{k,\ell+1}} -\frac{(k+1)}{\mu \Delta x} E_{y_{k+1,\ell}},  \quad k,\ell = 0, \dots, 2m, \\[1pt] %  \quad \ell = 0,\dots 2m,
               &\,\, H'_{z_{2m+1,\ell}} = \frac{(\ell+1)}{\mu\Delta y} E_{x_{2m+1,\ell+1}}, \quad \ell = 0,\dots, 2m,\\[1pt]
               &\,\, H'_{z_{k,2m+1}} = -\frac{(k+1)}{\mu\Delta x} E_{y_{k+1,2m+1}}, \quad k = 0,\dots, 2m,\\[1pt]
               &\,\, H'_{z_{2m+1,2m+1}} = 0,\\
          \end{aligned} \right. 
    \end{aligned}
\end{equation}
	for the linear PDE part, 
	and 
\begin{equation}
    \begin{aligned}
    & P'_{x_{k,\ell}} = J_{x_{k,\ell}}, \quad P'_{y_{k,\ell}} = J_{y_{k,\ell}}, \quad J'_{x_{k,\ell}} = \omega_p^2 E_{x_{k,l}} - \gamma J_{x_{k,\ell}}- \omega_0^2P_{x_{k,\ell}}, \\
    & J'_{y_{k,\ell}} = \omega_p^2 E_{y_{k,l}} - \gamma J_{y_{k,\ell}}- \omega_0^2P_{y_{k,\ell}}, \quad Q'_{k,\ell} = S_{k,\ell}, 
    \end{aligned}
\end{equation}
	for $k,\ell=0, \dots, 2m+1$,
	for the linear ODE part.
The nonlinear ODEs to evolve $\hat{S}$ involves the computation of $(E_x^{\mathcal{H}})^2$ and $(E_y^{\mathcal{H}})^2$.
As in 1-D, 
	we truncate the resulting polynomial to 
\begin{equation}
(E_{\beta}^{\mathcal{H}})^2 \approx \sum_{k=0}^{2m+1}\sum_{\ell=0}^{2m+1} \phi_{\beta_{k,\ell}} \xi^k\eta^\ell, \quad \phi_{\beta_{k,\ell}} = \sum_{i=0}^k\sum_{j=0}^{\ell} E_{\beta_{i,j}}E_{\beta_{k-i,\ell-j}}.
\end{equation}
Here $\beta$ is either $x$ or $y$.
We then have 
\begin{equation}
	S'_{k,\ell} = \omega_v^2 (\phi_{x_{k,\ell}} + \phi_{y_{k,\ell}} - Q_{k,\ell}) - \gamma_v S_{k,\ell}, \quad k,\ell = 0,\dots, 2m+1. 
\end{equation}

For the electric field coefficients, 
	we use the constitutive law equations \eqref{eq:nonlinear_constitutive_law_2D} and 
	the ODEs for $\hat{P}_x$ and $\hat{P}_y$ to rewrite Amp\`{e}re-Maxwell's law as 
\begin{equation} \label{eq:AmpereMaxwell_2D_1}
	\begin{aligned}
	\partial_t (\epsilon_\infty  \hat{E}_x + \theta_K \hat{E}_x (\hat{E}_x^2 + \hat{E}_y^2)+ \theta_R \hat{Q}\hat{E}_x) + \hat{J}_x - \frac{\partial_y \hat{H}_z}{\epsilon} = 0, && \\
	\partial_t (\epsilon_\infty  \hat{E}_y+ \theta_K \hat{E}_y  (\hat{E}_x^2 + \hat{E}_y^2) + \theta_R \hat{Q}\hat{E}_y) + \hat{J}_y + \frac{\partial_x \hat{H}_z}{\epsilon} = 0. &&
	\end{aligned}
\end{equation}
We then substitute the variables by their interpolants and truncate polynomials powers  
\begin{equation}
	E_{\alpha}^{\mathcal{H}}(E_{\beta}^{\mathcal{H}})^2 \approx \sum_{k=0}^{2m+1}\sum_{\ell=0}^{2m+1} \psi_{\alpha\beta_{k,\ell}} \xi^k\eta^\ell, \quad \psi_{\alpha\beta_{k,\ell}} = \sum_{i=0}^k\sum_{j=0}^\ell E_{\alpha_{i,j}}\phi_{\beta_{k-i,\ell-j}}.
\end{equation}
Focusing on the first equation of \eqref{eq:AmpereMaxwell_2D_1} and matching the remaining coefficients give
\begin{equation} \label{eq:dt_Ex_2D_1}
		\epsilon_\infty  E'_{x_{k,\ell}} + \theta_K ( \psi'_{xx_{k,\ell}} +  \psi'_{xy_{k,\ell}}) + \theta_R \bigg( \sum_{i=0}^k\sum_{j=0}^\ell E_{x_{i,j}}Q_{k-i,\ell-j}\bigg)' 
		+ J_{x_{k,\ell}} - \frac{(\ell+1)}{\epsilon\Delta y}H_{z_{k,\ell+1}} = 0, 
\end{equation}
for $k = 0,\dots,2m+1$ and  $\ell = 0, \dots, 2m$, 
	and  
\begin{equation}  \label{eq:dt_Ex_2D_2}
		\epsilon_\infty E'_{x_{k,\ell}} + \theta_K ( \psi'_{xx_{k,\ell}} + \psi'_{xy_{k,\ell}}) + \theta_R \bigg( \sum_{i=0}^k\sum_{j=0}^\ell E_{x_{i,j}}Q_{k-i,\ell-j}\bigg)' + J_{x_{k,\ell}} = 0, 
\end{equation}
	for $k = 0, \dots, 2m+1$ and $\ell=2m+1$. 

Extracting $E'_{x_{k,\ell}}$ and $E'_{y_{k,\ell}}$ from the terms $\psi'_{xx_{k,\ell}}$ and $\psi'_{xy_{k,\ell}}$ leads to the general expression
\begin{equation} \label{eq:dt_psi_2D}
	\begin{aligned}
	\psi'_{\alpha\beta{k,\ell}} = E'_{\alpha_{k,\ell}} \phi_{\beta_{0,0}} + 2 E'_{\beta_{k,\ell}} E_{\alpha_{0,0}} E_{\beta_{0,0}} 
	+ \sum_{i=0}^{k-1} \big( E'_{\alpha_{i,\ell}}\phi_{\beta_{k-i,0}} + 2E_{\alpha_{0,0}} E'_{\beta_{i,\ell}}E_{\beta_{k-i,0}} +  \phi'_{\beta_{i,\ell}}E_{\alpha_{k-i,0}}\big) & & \\
		+ \sum_{i=0}^k\sum_{j=0}^{\ell-1} \big( E'_{\alpha_{i,j}}\phi_{\beta_{k-i,\ell-j}} + 2E_{\alpha_{0,0}} E'_{\beta_{i,j}}E_{\beta_{k-i,\ell-j}} +  \phi'_{\beta_{i,j}}E_{\alpha_{k-i,\ell-j}}\big). & & 
	\end{aligned}
\end{equation}
Focusing now on the term of the form $\left( Q^{\mathcal{H}}E_\alpha^{\mathcal{H}} \right)'$ and using the ODEs for the coefficients of $Q^{\mathcal{H}}$, 
	we have 
\begin{equation} \label{eq:dt_QE_2D}
		 \bigg( \sum_{i=0}^k\sum_{j=0}^\ell E_{\alpha_{i,j}}Q_{k-i,\ell-j}\bigg)' = E'_{\alpha_{k,\ell}} Q_{0,0} + \sum_{i=0}^{k-1} E'_{\alpha_{i,\ell}}Q_{k-i,0} 
		+ \sum_{i=0}^k\sum_{j=0}^{\ell-1} E'_{\alpha_{i,j}}Q_{k-i,\ell-j} + \sigma_{\alpha_{k,\ell}}. % \sum_{i=0}^k\sum_{j=0}^\ell E_{\alpha_{i,j}} S_{k-i,\ell-j}. & & 
\end{equation}
Here $\alpha$ and $\beta$ are either $x$ or $y$, 
	and 
\begin{equation}
	E_{\alpha}^{\mathcal{H}} S^{\mathcal{H}} \approx \sum_{k=0}^{2m+1} \sum_{\ell=0}^{2m+1} \sigma_{\alpha_{k,\ell}} \xi^k \eta^\ell, \quad \sigma_{\alpha_{k,\ell}} = \sum_{i=0}^k\sum_{j=0}^\ell E_{\alpha_{i,j}} S_{k-i,\ell-j}.
\end{equation} 

Using \eqref{eq:dt_psi_2D} and \eqref{eq:dt_QE_2D} in equations    \eqref{eq:dt_Ex_2D_1} and \eqref{eq:dt_Ex_2D_2}, 
    and repeating the same procedure for the second equation of system \eqref{eq:AmpereMaxwell_2D_1} 
    allow us to isolate the terms $E_{x_{k,\ell}}'$ and $E_{y_{k,\ell}}'$.
This leads to a $(2m+2)^2$ linear system of ODEs of the form
\begin{equation}
	M \mathbold{E}'_{k,\ell} = \mathbold{b}_{k,\ell}.
\end{equation}
Here $\mathbold{E}_{k,\ell} = [E_{x_{k,\ell}}, E_{y_{k,\ell}}]^T$, 
	the right-hand side $\mathbold{b}_{k,\ell}$ depends among other things on the previous computed $\mathbold{E}'_{i,j}$ but not on $\mathbold{E}'_{k,\ell}$,
	and the $2\times2$ matrix $M$ is 
\begin{equation} \label{eq:matrix_M_2D}
		\begin{bmatrix}
		\epsilon_\infty + \theta_K(3\phi_{x_{0,0}} + \phi_{y_{0,0}}) + \theta_R Q_{0,0} & 2 \theta_K E_{x_{0,0}}E_{y_{0,0}} \\
		2 \theta_K E_{x_{0,0}}E_{y_{0,0}}  & \epsilon_\infty + \theta_K(\phi_{x_{0,0}} + 3\phi_{y_{0,0}}) + \theta_R Q_{0,0}
		\end{bmatrix}.
\end{equation}
We therefore have a recursive system of ODEs for the electric field coefficients.
Note that the matrix $M$ is independent of the subscripts $(k,\ell)$, 
	and the order of the method. 
	
\subsection{A Note on the Three Dimensional Case}

In this subsection, 
	we introduce the proposed Hermite method for Maxwell's equations \eqref{eq:Maxwell_3D}, 
	providing only the key components since its extension from two to three space dimensions is straightforward. 
For legibility, 
	we omit the subscripts that indicate the spatial coordinates. 
	
Combining Amp\`{e}re-Maxwell's law and the nonlinear constitutive law \eqref{eq:constitutive_law_DEPQ_3D} leads to 
\begin{equation} \label{eq:AmpereMaxwell_3D_1}
	\partial_t (\epsilon_\infty \hat{\mathbold{E}} + \theta_K \|\hat{\mathbold{E}}\|^2\hat{\mathbold{E}} + \theta_R \hat{Q} \hat{\mathbold{E}}) + \hat{\mathbold{J}} -\frac{\nabla\times\hat{\mathbold{H}}}{\epsilon} = 0.
\end{equation}
Employ polynomial approximations of all variables of the form 
\begin{equation} \label{eq:Hermite_interpolant_3D}
	f^{\mathcal{H}}(x,y,z) = \sum_{k=0}^{2m+1}\sum_{\ell=0}^{2m+1}\sum_{r=0}^{2m+1} f_{k,\ell,r} \xi^k \eta^\ell \zeta^r.
\end{equation}
Following the same procedure as in the two space dimensional case, 
	the nonlinear terms $\hat{E}_x(\hat{E}_x^2+\hat{E}_y^2+\hat{E}_z^2)$ would lead to 
\begin{equation} \label{eq:nonlinear_Hermite_interpolant_3D}
	E_{\alpha}^{\mathcal{H}} (E_\beta^{\mathcal{H}})^2 \approx 
	\sum_{k=0}^{2m+1}\sum_{\ell=0}^{2m+1}\sum_{p=0}^{2m+1} \psi_{\alpha\beta_{k,\ell,p}}\xi^k\eta^\ell \zeta^p.
\end{equation}
Here $\alpha$ and $\beta$ are either $x$,
	$y$ and $z$, 
	and 
\begin{equation}
	\psi_{\alpha\beta_{k,\ell,p}} = \sum_{i=0}^k\sum_{j=0}^\ell\sum_{r=0}^p E_{\alpha_{i,j,r}} \phi_{\beta_{k-i,\ell-j,p-r}}.
\end{equation}
Substituting the polynomial approximations for each variable in \eqref{eq:AmpereMaxwell_3D_1} requires to extract $E'_{x_{k,\ell,p}}$,
	$E'_{y_{k,\ell,p}}$ and $E'_{z_{k,\ell,p}}$ from the nonlinear terms \eqref{eq:nonlinear_Hermite_interpolant_3D}.
This leads to 
\begin{equation}
	\begin{aligned}
		\psi'_{\alpha\beta_{k,\ell,p}} = E'_{\alpha_{k,\ell,p}} \phi_{\beta_{0,0,0}} + 2 E'_{\beta_{k,\ell,p}} E_{\beta_{0,0,0}} E_{\alpha_{0,0,0}}  
		+ \sum_{i=0}^{k-1} \big(E'_{\alpha_{i,\ell,p}}\phi_{\beta_{k-i,0,0}} + 
			2 E'_{\beta_{i,\ell,p}}E_{\beta_{k-i,0,0}}E_{\alpha_{0,0,0}} + \phi'_{\beta_{i,\ell,p}}E_{\alpha_{k-i,0,0}}\big)   & & \\ 
			 + \sum_{i=0}^k\sum_{j=0}^{\ell-1} \big( E'_{\alpha_{i,j,p}}\phi_{\beta_{k-i,\ell-j,0}}   
			+ 2 E'_{\beta_{i,j,p}}E_{\beta_{k-i,\ell-j,0}}E_{\alpha_{0,0,0}}  
		+ \phi'_{\beta_{i,j,p}}E_{\alpha_{k-i,\ell-j,0}}\big)  & & \\
		+ \sum_{i=0}^k\sum_{j=0}^{\ell}\sum_{r=0}^{p-1} ( E'_{\alpha_{i,j,r}}\phi_{\beta_{k-i,\ell-j,p-r}} + 2 E'_{\beta_{i,j,r}}E_{\beta_{k-i,\ell-j,p-r}}E_{\alpha_{0,0,0}}  
		+ \phi'_{\beta_{i,j,r}}E_{\alpha_{k-i,\ell-j,p-r}}). & & 
	\end{aligned}
\end{equation}
Isolating the time derivative of the coefficients $E_{x_{k,\ell,p}}$,
	$E_{y_{k,\ell,p}}$ and $E_{z_{k,\ell,p}}$ in the time evolution equations for the electric field coefficients leads 
	to $(2m+2)^3$ linear systems of ODEs of the form 
\begin{equation}
	M \mathbold{E}'_{k,\ell,p} = \mathbold{b}_{k,\ell,p}.
\end{equation}
Here $\mathbold{E}_{k,\ell,p} = [E_{x_{k,\ell,p}}, E_{y_{k,\ell,p}}, E_{z_{k,\ell,p}}]^T$ and the $3\times3$ matrix $M$ 
	defined by the components 
\begin{equation} \label{eq:matrix_M_3D}
	\begin{aligned}
		m_{1,1} = \epsilon_\infty + \theta_K(3\phi_{x_{0,0,0}} + \phi_{y_{0,0,0}} + \phi_{z_{0,0,0}}  ) + \theta_R Q_{0,0,0}, & &\\
		m_{2,2} = \epsilon_\infty + \theta_K(\phi_{x_{0,0,0}} + 3\phi_{y_{0,0,0}} + \phi_{z_{0,0,0}}  ) + \theta_R Q_{0,0,0}, & &\\
		m_{3,3} = \epsilon_\infty + \theta_K(\phi_{x_{0,0,0}} + \phi_{y_{0,0,0}} + 3\phi_{z_{0,0,0}}  ) + \theta_R Q_{0,0,0}, & &\\
		m_{1,2} = m_{2,1} = 2 \theta_R E_{x_{0,0,0}}E_{y_{0,0,0}}, & & \\
		m_{1,3} = m_{3,1} = 2 \theta_R E_{x_{0,0,0}}E_{z_{0,0,0}}, & & \\
		m_{2,3} = m_{3,2} = 2 \theta_R E_{y_{0,0,0}}E_{z_{0,0,0}}. & & 
	\end{aligned}
\end{equation}
Here again we obtain a recursive system of ODEs for the electric field coefficients and the matrix $M$ is independent of the subscripts $(k,\ell,p)$,
	and the order of the method.
	
\section{$P$-Adaptive Algorithm} \label{sec:p_adaptive_algo}

Since the Hermite method evolves in time the data independently on each cell, 
	it is straightforward to adapt the degree of the local Hermite interpolant and hence the polynomial approximants on a cell.
This is generally referred to as $p$-adaptivity.
Here we use the $p$-adaptive method presented in \cite{Chen2012} 
    to efficiently
    conduct long time simulations of wave packets with local structures, 
	such as solitons propagating through a Kerr-type medium.

To do so, 
	for each node of the primal and dual meshes in three space dimensions, 
	we track the data by defining the variable $m_{\alpha,[i,j,r]}$ representing the maximal order of the spatial derivatives with respect 
	to $\alpha$ available for all variables at the node $(x_i,y_j,z_r)$ of the mesh. 
Here, 
	for simplicity, 
	we assume $m_{[i,j,r]} = m_{x,[i,j,r]} = m_{y,[i,j,r]} = m_{z,[i,j,r]}$ is the same for all variables and in every coordinate direction. We note, however, that using spatially anisotropic approximations as well as differing degrees for each variable is a straightforward extension, albeit one which adds some implementational complexity.  
Note that we set $m_{[i,j,r]}$ to the maximum derivative order that is allowed at the initial time $t_0$.

Consider a half time step from the primal mesh at $t_n$ to the dual mesh at $t_{n+1/2}$. 
For each cell and each variable,
	we compute the $(2\bar{m}+1)^3$ degree Hermite polynomial using the function value and the derivatives through order 
	$\bar{m}$ at each vertex of the cell.
Here $\bar{m}$ is the minimal value of $m_{[i,j,r]}$ associated with the vertices forming the primal cell $[x_i, x_{i+1}]\times[y_j,y_{j+1}]\times[z_r,z_{r+1}]$, 
	centered at $(x_{i+1/2},y_{j+1/2},z_{r+1/2})$.
The evolution of the data at the cell center $(x_{i+1/2},y_{j+1/2},z_{r+1/2})$ from $t_n$ to $t_{n+1/2}$ is the same as described 
	in subsection~\ref{sec:time_evolution_1D}.

The $p$-adaptivity occurs at the end of the time evolution step. 
Having all the $(\bar{m}+1)^3$ polynomial coefficients at $(x_{i+1/2},y_{j+1/2},z_{r+1/2},t_{n+1/2})$ approximating 
	scaled derivatives of the variables, 
	we seek the smallest $m_{[i+1/2,j+1/2,r+1/2]}$ that satisfies
\begin{equation} \label{eq:cutoff_criteria}
	\begin{aligned}
		\max_{k>m_{[i+1/2,j+1/2,r+1/2]}} \|\mathbold{U}_{k,\ell,p}\|_\infty \leq \epsilon_{p_{tol}}, && \\
		\max_{\ell>m_{[i+1/2,j+1/2,r+1/2]}} \|\mathbold{U}_{k,\ell,p}\|_\infty \leq 	\epsilon_{p_{tol}}, && \\
		\max_{p>m_{[i+1/2,j+1/2,r+1/2]}} \|\mathbold{U}_{k,\ell,p}\|_\infty \leq 	\epsilon_{p_{tol}}, & & 
	\end{aligned}
\end{equation}
    and therefore determine the data available at $(x_{i+1/2},y_{j+1/2},z_{r+1/2},t_{n+1/2})$ on the dual mesh.
Here $\epsilon_{p_{tol}}$ is the given tolerance and
\begin{equation}
    \mathbold{U}_{k,\ell,p} = (\mathbold{H}_{k,\ell,p}, \mathbold{E}_{k,\ell,p},\mathbold{P}_{k,\ell,p},\mathbold{J}_{k,\ell,p},Q_{k,\ell,p},S_{k,\ell,p})
\end{equation} 
    is a vector containing all the polynomial 
	coefficients associated with the subscripts $(k,\ell,p)$ at $(x_{i+1/2},y_{j+1/2},z_{r+1/2},t_{n+1/2})$.

In summary, 
    given the data on the primal mesh at time $t_n$, 
    the algorithm of the $p$-adaptive Hermite method to execute a complete time step is: 
\begin{itemize}
    \item[(i)] for each primal cell at time $t_n$, 
        find the value of $\bar{m}$, 
        compute the unique degree $(2\bar{m}+1)^3$ Hermite interpolant centered at the cell's center for each variable, 
        and extract their coefficients;
    \item[(ii)] for each primal cell, 
        solve numerically the system of ODEs \eqref{eq:ODE_system} to evolve all the coefficients from $t_n$ to $t_{n+1/2}$. 
        We now have all the polynomial coefficients at each dual mesh node;
    \item[(iii)] for each dual mesh node, 
        determine $m$ using the cutoff criterion \eqref{eq:cutoff_criteria};
    \item[(iv)] for each dual cell at time $t_{n+1/2}$, 
        find the value of $\bar{m}$, 
        compute the unique degree $(2\bar{m}+1)^3$ Hermite interpolant centered at the cell's center for each variable, 
        and extract their coefficients;
    \item[(v)] for each dual cell, 
        solve numerically the system of ODEs \eqref{eq:ODE_system} to evolve all the coefficients from $t_{n+1/2}$ to $t_{n+1}$. 
        We now have all the polynomial coefficients at each primal mesh node;
    \item[(vi)] for each primal mesh node, 
        determine $m$ using the cutoff criterion \eqref{eq:cutoff_criteria}.
\end{itemize}

%%%%%%%%%%%%%%%

\section{Properties of the Method} \label{sec:analysis_method}

\subsection{Solving the Linear Systems}

The proposed Hermite method requires solving small linear systems to evolve the electric field polynomial coefficients in time. 
In the situation when an explicit RK method is used, 
	the symmetric matrix $M$ is constant at each stage,
	independent of the order, 
	and depends only on the coefficients $\mathbold{E}_{0,0,0}$ and $Q_{0,0,0}$ computed at the 
	previous stage. 
Since the system involves at most three linear equations, 
	we explicitly implement the solution.

The solvability of the linear system is related to the conditions which guarantee the local hyperbolicity of the system with frozen coefficients. Specifically, the principal part of the frozen coefficient system takes the form:
\begin{eqnarray*}
\mu \partial_t \tilde{\mathbold{H}} + \nabla \times \tilde{\mathbold{E}} & = & 0, \\
M \partial_t \tilde{\mathbold{E}} - \nabla \times \tilde{\mathbold{H}} & = & 0, \\
\partial_t \tilde{\mathbold{P}} = \partial_t \tilde{\mathbold{J}} = \partial_t \tilde{\mathbold{Q}} = \partial_t \tilde{\mathbold{S}} & = & 0,
\end{eqnarray*}
where the matrix $M$ is evaluated at the local values of $\hat{\mathbold{E}}$, $\hat{\mathbold{Q}}$. This system is hyperbolic if and only if $M$ is positive definite. Under the same assumptions for the point values of the approximate solution we can solve the system for the time derivatives of the electric field. 

\begin{proposition}
If the matrix $M$ evaluated at $\mathbold{\phi}_{0,0}$, $\mathbold{Q}_{0,0}$ is positive definite then the system of equations for the time derivatives is uniquely solvable. This condition always holds when only the Kerr term is considered, $(\theta = 0)$ and $a\geq0$. 
\end{proposition}
\begin{proof}
If $M$ is positive definite the result is automatic based on the recursions described in the previous sections. Therefore we only consider the Kerr case. 
In one space dimension, 
	$M = \epsilon_\infty + 3a\phi_0 > 0$ since $\phi_0 = E_0^2 \geq 0$. 

In two space dimensions, 
	the determinant of the $2\times2$ matrix $M$, 
	defined by \eqref{eq:matrix_M_2D}, 
	is 
\begin{equation}
	\det{M} = \epsilon_\infty + 4a\epsilon_\infty(\phi_{x_{0,0}} + \phi_{y_{0,0}}) + 3a^2(\phi_{x_{0,0}}+\phi_{y_{0,0}})^2 > 0,
\end{equation}
	since $\phi_{\alpha_{0,0}} = E_{\alpha_{0,0}}^2 \geq 0$. As the diagonal entries are also positive the result holds. 

In three space dimensions, 
	the determinant of the $3\times3$ matrix $M$ with the components $\eqref{eq:matrix_M_3D}$ is 
\begin{equation}
	\det{M} = (\epsilon_\infty + a(\phi_{x_{0,0,0}} + \phi_{y_{0,0,0}} + \phi_{z_{0,0,0}}))^2(\epsilon_\infty + 3a(\phi_{x_{0,0,0}}+\phi_{y_{0,0,0}} + \phi_{z_{0,0,0}})) > 0,
\end{equation}
	since $\phi_{\alpha_{0,0,0}} = E_{\alpha_{0,0,0}}^2 \geq 0$. We can similarly check the subdeterminants and conclude by Sylvester's criterion that $M$ is positive definite. 
\end{proof}

\subsection{Computational Cost}

The computational cost of the Hermite method is dominated by the Hermite interpolation procedures 
	and the computation of the right-hand side of the 
	system of ODEs for the polynomial coefficients. 
As shown in \cite{Chen2012}, 
	the cost of the Hermite interpolation is $\mathcal{O}(m^{d+1})$ in $\mathbb{R}^d$. 
In the following, 
	the cost of the computation of the right-hand side (RHS) of the system of ODEs \eqref{eq:ODE_system} for the coefficients is performed. 

Consider the one space dimensional case.
At the beginning of each stage, 
	we compute the vectors $\mathbold{\phi}$ and $\mathbold{\sigma}$ containing respectively the coefficients $\phi_k$ and $\sigma_k$.
Using the fast Fourier transform (FFT), 
	the computation of these vectors requires $\mathcal{O}((2m+1)\log(2m+1))$ operations \cite{Cormen2022}.

The computation of the RHS of the linear PDE part \eqref{eq:ode_coeff_magnetic_field_part_1} and \eqref{eq:ode_coeff_magnetic_field_part_2}, 
	requires $\mathcal{O}(m)$ operations. 
As for the linear ODE part \eqref{eq:ode_coeff_auxiliary_P_J_Q}, 
	only the terms $J_k'$ require some computations for a total of $\mathcal{O}(m)$ operations. 
Since we already compute the vector $\mathbold{\phi}$, 
	the cost for the nonlinear ODE part \eqref{eq:ode_coeff_auxiliary_S} is also $\mathcal{O}(m)$ operations. 

As for the computation of the RHS of the nonlinear PDE part \eqref{eq:ode_coeff_E_1} and \eqref{eq:ode_coeff_E_2}, 
	we assume that we construct the vector $\mathbold{\phi}'$ containing the terms $\phi_k'$ 
	after each computation of the RHS of $E_k'$.
Note that we cannot take advantage of the FFT to compute this vector 
	since we have to isolate the terms $E_k'$ in order to avoid using a nonlinear solver.
The cost of $\mathbold{\phi}'$ is therefore equivalent of a direct implementation of two polynomial products, 
	that is $\mathcal{O}(m^2)$.
For each term $E_k'$,	
	the RHS of this term involves $\sigma_p$, 
	$\phi_p$, 
	$E_p'$ and $\phi_p'$ for $p=0,\dots,k-1$, 
	that were already computed and stored.
The number of operations is therefore dominated by the terms of the form $\sum_{p=0}^k A_pB_{k-p}$, 
	which require $\mathcal{O}(k)$ operations with $k \leq 2m+1$. 
	
In 1-D, 
	the total computational cost of the RHS of the nonlinear PDE part is $\mathcal{O}(m^2)$ operations.
We therefore conclude that the cost of the RHS of the system \eqref{eq:ODE_system} is $\mathcal{O}(m^2)$ operations, 
	which has the same order as the cost of the Hermite interpolation procedure.
	
Repeating the same analysis for the two and three dimensional cases and considering $m = m_x = m_y = m_z$, 
	the total cost is dominated by the implementation of multi-dimension polynomial products that require 
	$(m+1)^d(2m+3)^d$ multiplications and $(m+1)^d(2m+3)^{d-1}(2m+1)$ additions in $\mathbb{R}^d$, 
	and the summations for the RHS of
	the time derivatives of the electric field coefficients, 
	leading to $\mathcal{O}(m^4)$ and $\mathcal{O}(m^6)$ operations in 2-D and 3-D.
This exceeds the computational cost of the Hermite interpolation procedure. 

Consider the total number of evaluations of the RHS,
	denoted $s$, 
	for an explicit time-stepping method, 
	the total cost for the evolution of the polynomial coefficients over half a time step is $\mathcal{O}(sN_{\tau}m^{2d})$ in $\mathbb{R}^d$.
The computational cost of the proposed Hermite method is reasonable for small values of $m$,
	leading to a $2m+1$ order method.  
Moreover, 
	for soliton-like wave propagation, 
	the cost could be further reduced using the $p$-adaptive algorithm described in Section~\ref{sec:p_adaptive_algo}.
In this work, 
	we use $m\leq 4$ without $p$-adaptivity and $m\leq6$ with it, restricting the spatial order to $9$ and $13$ respectively.

\subsection{Local Truncation Error Analysis in Space}

In this subsection, 
	we investigate the impact of the truncation of the polynomials $(E_\beta^\mathcal{H})^2$ and $E_\alpha^{\mathcal{H}}(E_\beta^\mathcal{H})^2$.
We limit the analysis to one space dimension.

Consider a cell $[x_i,x_{i+1}]$ and the exact data as a function of time $t$ at all primal nodes for each variable.
For a given variable $\hat{f}$ at $(x_i,t)$,
	we then have the exact data
	$\partial^{k} \hat{f}(x_i,t)$, 
	$k=0,\dots,m$.  
For a sufficiently smooth $\hat{f}$, 
	the error $e_f(x,t) = \hat{f}(x,t) -   f^{\mathcal{H}}(x,t)$ of the Hermite interpolant of degree $2m+1$ satisfies for $x \in (x_i,x_{i+1})$ and $0 \leq k \leq m+1$ (\cite[Thm. 9]{CiarletSchultzVarga})
\begin{equation}
\left\arrowvert \partial_x^k e_f(x,t) \right\arrowvert \leq 
	\frac{1}{(2m+2-2k)! k!}((x-x_{i})(x_{i+1}-x))^{m+1-k} (\Delta x)^k \| \frac{\partial^{2m+2}\hat{f}}{\partial x^{2m+2}} \|_{C^0} . 
\end{equation} 

Substituting the magnetic field and the electric field in Faraday's law \eqref{eq:Faraday_1D} by their associated Hermite interpolants and errors leads to 
	\begin{eqnarray}
		\frac{\partial \hat{H}(x,t)}{\partial t} - \frac{\partial \hat{E}(x,t)}{\partial x} & = &
        \frac{\partial H^{\mathcal{H}}(x,t)}{\partial t} - \frac{\partial E^{\mathcal{H}}(x,t)}{\partial x} + \tau_H  \label{eq:local_trunc_err_Faraday_law} \\ 
        \left\arrowvert \tau_H \right\arrowvert & \leq & \frac{1}{(2m+2)!} ((x-x_{i})(x_{i+1}-x))^{m+1} \| \frac{\partial^{2m+3}\hat{H}}{\partial t\partial x^{2m+2}} \|_{C^0} \nonumber \\ & &  + \frac{\Delta x}{(2m)!} ((x-x_{i})(x_{i+1}-x))^m \| \frac{\partial^{2m+2}\hat{E}}{\partial x^{2m+2}}\|_{C^0}.  \nonumber
        \end{eqnarray}
The local truncation error in space is then 
	$\mathcal{O}(\Delta x^{2m+1})$. 
For equations \eqref{eq:P_ODE_1D},
	\eqref{eq:J_ODE_1D} and \eqref{eq:Q_ODE_1D} we have a truncation error in space of $\mathcal{O}(\Delta x^{2m+2})$ as no spatial derivatives appear.
	
Let us now consider equation \eqref{eq:S_ODE_1D}. 
In this case, 
	we truncate the square of the Hermite interpolant of the electric field, 
	introducing an extra error. 
Focusing on this term,
	we have 
\begin{equation} \label{eq:E_square_trunc_error}
		E^2 = (E^\mathcal{H} + e_E )^2 = (E^\mathcal{H})^2 + 2E^{\mathcal{H}} e_E + e_E^2 = \Phi + \mathcal{O}(\Delta x^{2m+2}).
\end{equation}
Here 
$$\Phi(x,t) =  \displaystyle\sum_{k=0}^{2m+1} \phi_k(t) \xi^k(x).$$
At the cell center, 
	note that $(E^{\mathcal{H}}(x_{i+1/2},t))^2 = \phi_0(t)$, 
	eliminating the extra error term coming from the truncation of the polynomial $(E^\mathcal{H})^2$. 
Substituting the remaining variables by their associated Hermite interpolants and errors in equation \eqref{eq:S_ODE_1D} and 
	considering equation \eqref{eq:E_square_trunc_error} lead to 
	a local truncation error of $\mathcal{O}(\Delta x^{2m+2})$. 

Finally, 
	we investigate equation \eqref{eq:AmpereMaxwell_1}. 
Here also we introduce an additional error term due to the truncation of $(E^{\mathcal{H}})^3$ and
	$E^{\mathcal{H}}Q^{\mathcal{H}}$.
We have 
\begin{equation}
	E^3 = (E^\mathcal{H} + e_E )^3 = (E^\mathcal{H})^3 + 3(E^{\mathcal{H}})^2 e_E +3 E^{\mathcal{H}} e_E^2  + e_E^3 = \Psi + \mathcal{O}(\Delta x^{2m+2}).
\end{equation}
Here 
$$\Psi(x,t) = \displaystyle\sum_{k=0}^{2m+1} \psi_k(t)\xi^k(x).$$
Note that $\Psi(x_{i+1/2},t) = \psi_0(t)$.
Following the same procedure as $(E^\mathcal{H})^2$ for the product $E^{\mathcal{H}}Q^{\mathcal{H}}$ leads to $\mathcal{O}(\Delta x^{2m+2})$.
Using the fact that the time derivative acts only on the polynomial coefficients and following a similar procedure as before, 
	we obtain a local truncation error in space of $\mathcal{O}(\Delta x^{2m+1})$. 

In summary, 
	the truncation error in space introduced in the approximation of the polynomial products  
	has higher order than the derivative approximations.
Therefore we do not expect that the truncation procedure required to approximate the polynomial products impact on the overall order of the method.

\subsection{Stability}	
	
For linear hyperbolic problems, 
    the  Hermite-Taylor method is stable under the condition $\Delta t\leq h/c$ with $c$ being the largest wave speed, 
    regardless of its order of accuracy \cite{Goodrich2005}. 
In other words, 
    the stability condition depends only on the domain of dependance of the PDE. 
When using the Hermite-Runge-Kutta method, 
	numerical experiments in \cite{Hagstrom2007} show that the method is stable with a CFL number slightly smaller than one for a sufficiently large number of sub-steps in time.
	
Based on the insights of the linear case, 
	in our current setting, 
	we adapt the number of time sub-steps locally to match the accuracy of the Hermite method in space, 
	leading to 
\begin{equation} \label{eq:nb_substeps_equation}
	N_\tau = N_\tau^* = \bigg\lceil\frac{\Delta t}{2 h^{\frac{2\bar{m}+1}{k}}}\bigg\rceil,
\end{equation} 
	for a $k^{th}$-order RK method.
Here $h=\min\{\Delta x, \Delta y, \Delta z\}$ and $k=5$. 
That being said, 
	note that the maximal value of $N_\tau$ could also be fixed for all considered $m$, 
    but one should expect a reduction of the time step size $\Delta t$ in this situation as shown in \cite{Hagstrom2007} for the linear case.
The wave speeds for this problem could, in principle, become unbounded if the eigenvalues of the matrix $M$ became small, in which case $\Delta t$ would have to be decreased. We did not encounter this in our numerical experiments and we found that our choice of sub-steps provided a stable method. This is
   demonstrated in the numerical examples of Section~\ref{sec:numerical_examples}.

%%%%%%%%%%%%%%%%%%

\section{Numerical Examples} \label{sec:numerical_examples}

In this section,
	we perform convergence studies to verify the proposed Hermite method in one and two dimensions.
The performance of the $p$-adaptive method is also investigated.

\subsection{Examples in One Space Dimension}

We first consider the manufactured solution
\begin{equation}
	\begin{aligned}
		\hat{H} = \sin(wx)\sin(wt), \quad \hat{E} = -\cos(wx)\cos(wt), \quad \hat{P} = -a\hat{E}^3,  & & \\
		\hat{J} = -3aw\cos(wx)\sin(wt)\hat{E}^2, \quad \hat{Q} = \hat{E}^2, \quad \hat{S} = 2w\cos(wx)\sin(wt)\hat{E}.  & &
	\end{aligned}
\end{equation}
Here $w = 10\pi$.
All the physical parameters are one, 
	except for $a = 1/3$, 
	$\theta = 1/2$,
	$\gamma = 1/20$ and $\gamma_v=1/20$. 
Note that source terms are required for the ODEs \eqref{eq:J_ODE_1D} and \eqref{eq:S_ODE_1D} associated with $\hat{J}$ and $\hat{S}$. 

The domain is $\Omega = [0,1]$ and the time interval is $I=[0,10]$.
We consider $m=1-4$ and set the time step size to $\Delta t = \Delta x/2$ and the number of time sub-steps to $N_\tau = N_\tau^*$ for all considered values of $m$.
The left plot in Figure~\ref{fig:conv_1d} illustrates the relative error in maximum norm as a function of the mesh size,
	confirming the expected $2m+1$ rate of convergence. 
%***************************************
\begin{figure}   
	\centering
	\begin{adjustbox}{max width=1.0\textwidth,center}
		\includegraphics[width=2.5in,trim={0.0cm 0cm 1.75cm 0cm},clip]{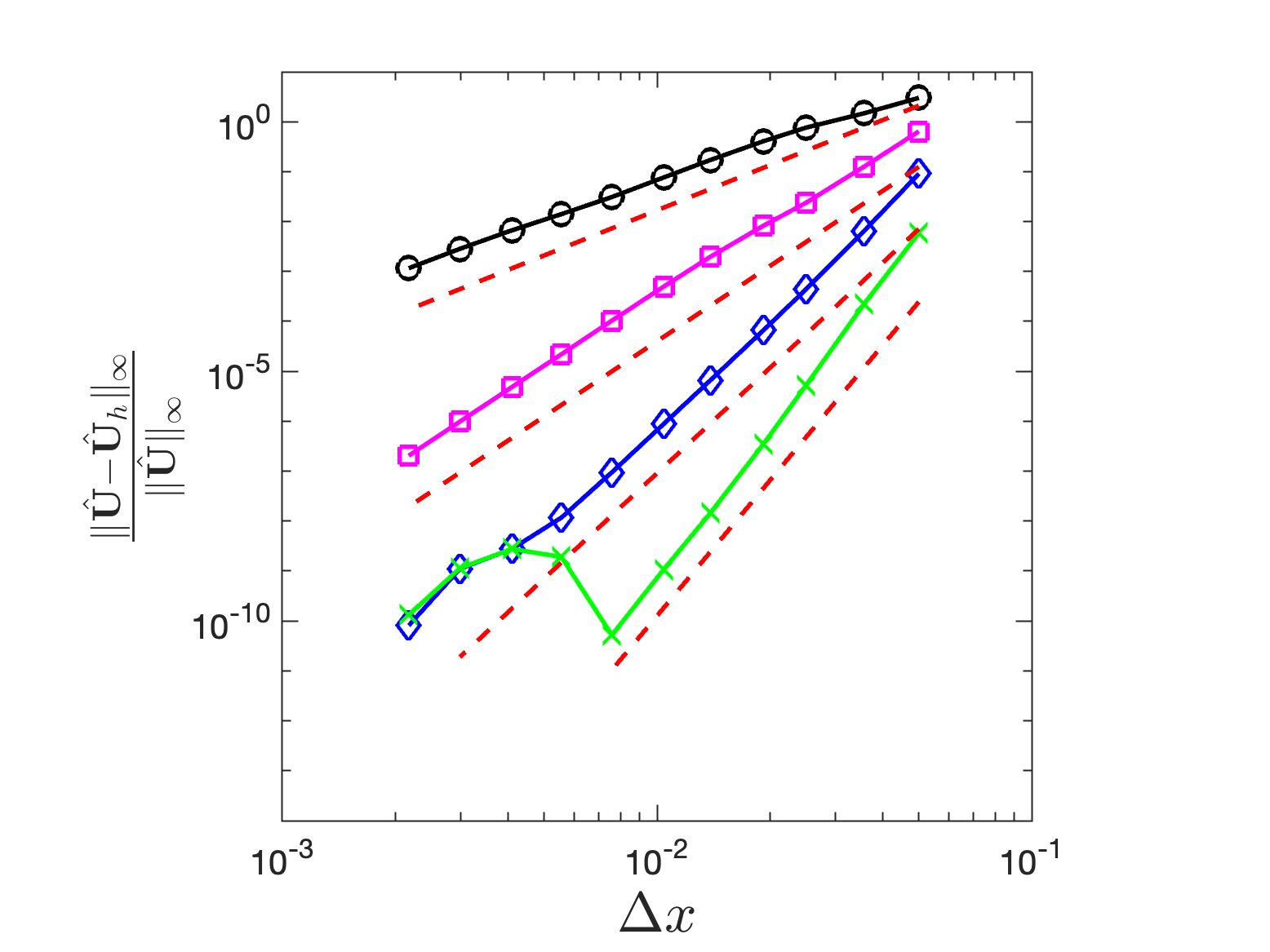} \hspace{-20pt}
		\includegraphics[width=2.5in,trim={0.0cm 0.0cm 1.75cm 0cm},clip]{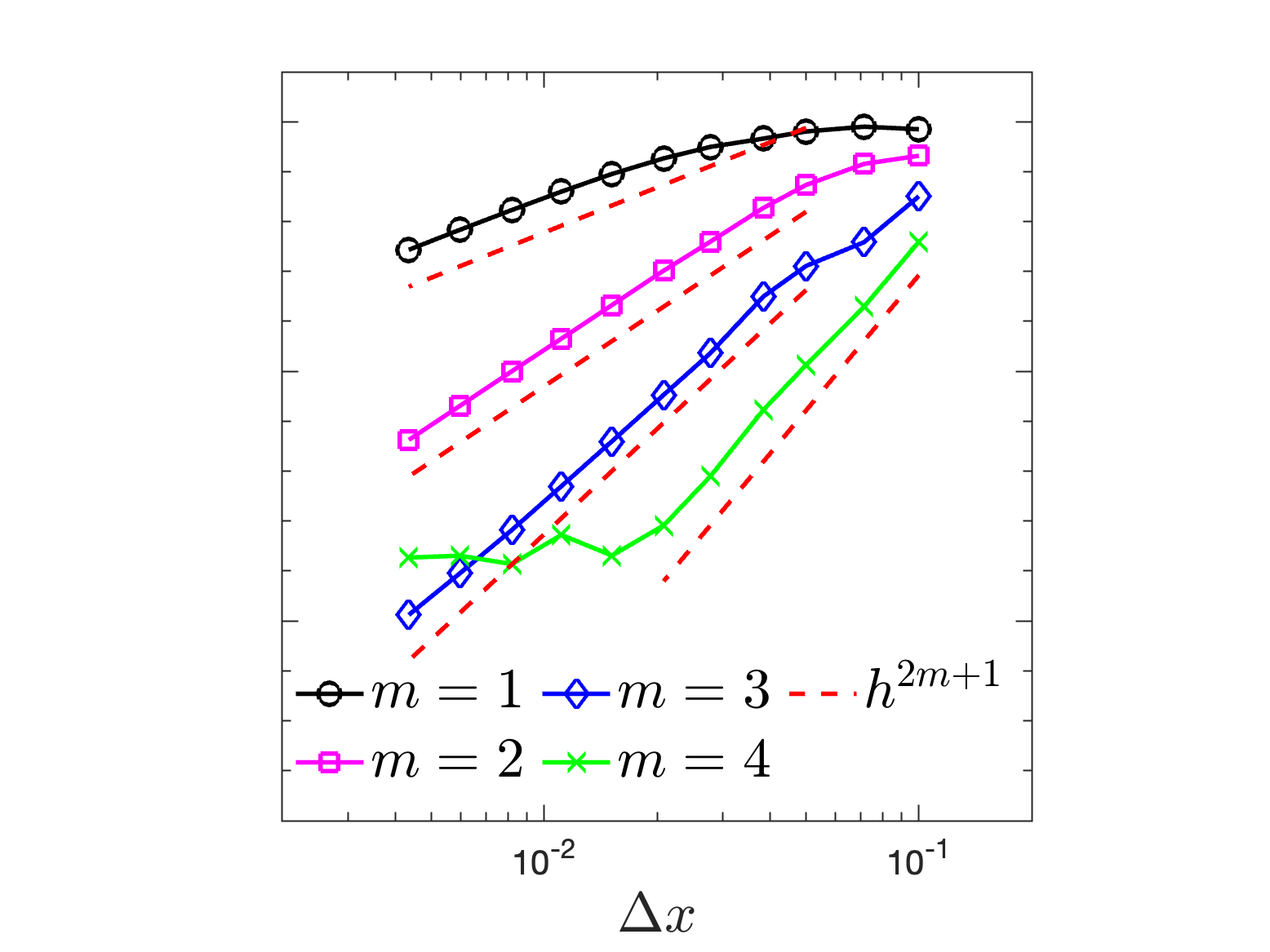}
	\end{adjustbox} 
       \caption{Convergence plots for manufactured solutions problems for $m=1-4$ in one space dimension. 
       		Here $\hat{\mathbf{U}}$ is a vector containing all variables.}
       \label{fig:conv_1d}
\end{figure}
%***************************************

As a second numerical example, 
	we consider the manufactured solution, 
\begin{equation}
	\begin{aligned}
		\hat{H} = \hat{E} = -\eta e^{-(\frac{\eta}{\sigma})^2}, \quad \hat{P} = -a \hat{E}^3, \quad 
		\hat{J} = -\frac{3 a \eta^2 e^{-3(\frac{\eta}{\sigma})^2}(2\eta^2-\sigma^2)}{\sigma^2}, & & \\ 
	 \hat{Q} = \hat{E}^2, \quad \hat{S} = -\frac{2 \eta e^{-2(\frac{\eta}{\sigma})^2}(2\eta^2-\sigma^2)}{\sigma^2}, \quad \eta = x + t, & & 
	\end{aligned}
\end{equation} 
	representing a localized pulse traveling to the left at constant speed of one.
Here $\sigma=0.1$.
All physical parameters are the same as before and source terms are also required in the ODEs \eqref{eq:J_ODE_1D} and \eqref{eq:S_ODE_1D}.

We set $\Omega = [-1,1]$, 
	$I =[0,10]$ and $\Delta t = \Delta x/2$.
The right plot of Figure~\ref{fig:conv_1d} illustrates the relative error in maximum norm as a function of the mesh size for $m=1-4$.
As expected, 
	we observe a $2m+1$ rate of convergence. 

Let us now consider the $p$-adaptive method for the second numerical example. 
We set the time interval $I=[0,100]$ and $0\leq m \leq 6$.
Note that the $\hat{\mathbf{U}}(x,0) = \hat{\mathbf{U}}(x,100)$.
The values of $N_\tau$ are taken as in equation \eqref{eq:nb_substeps_equation}. 
%***************************************
\begin{figure}   
	\centering
	\begin{adjustbox}{max width=1.0\textwidth,center}
		\includegraphics[width=2.5in,trim={0.0cm 0cm 1.75cm 0cm},clip]{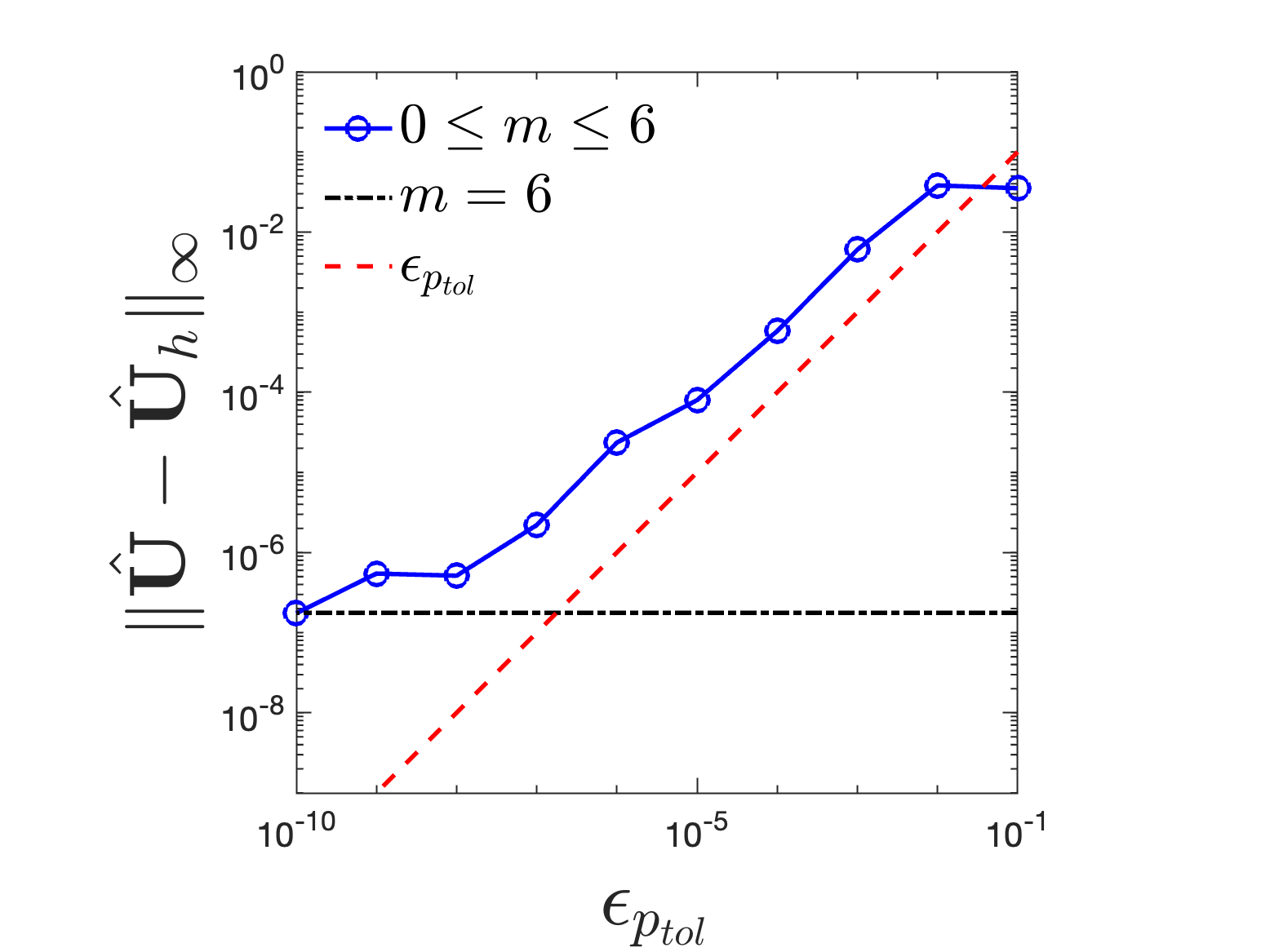} \hspace{-17pt}
		\includegraphics[width=2.5in,trim={0.0cm 0.0cm 1.75cm 0cm},clip]{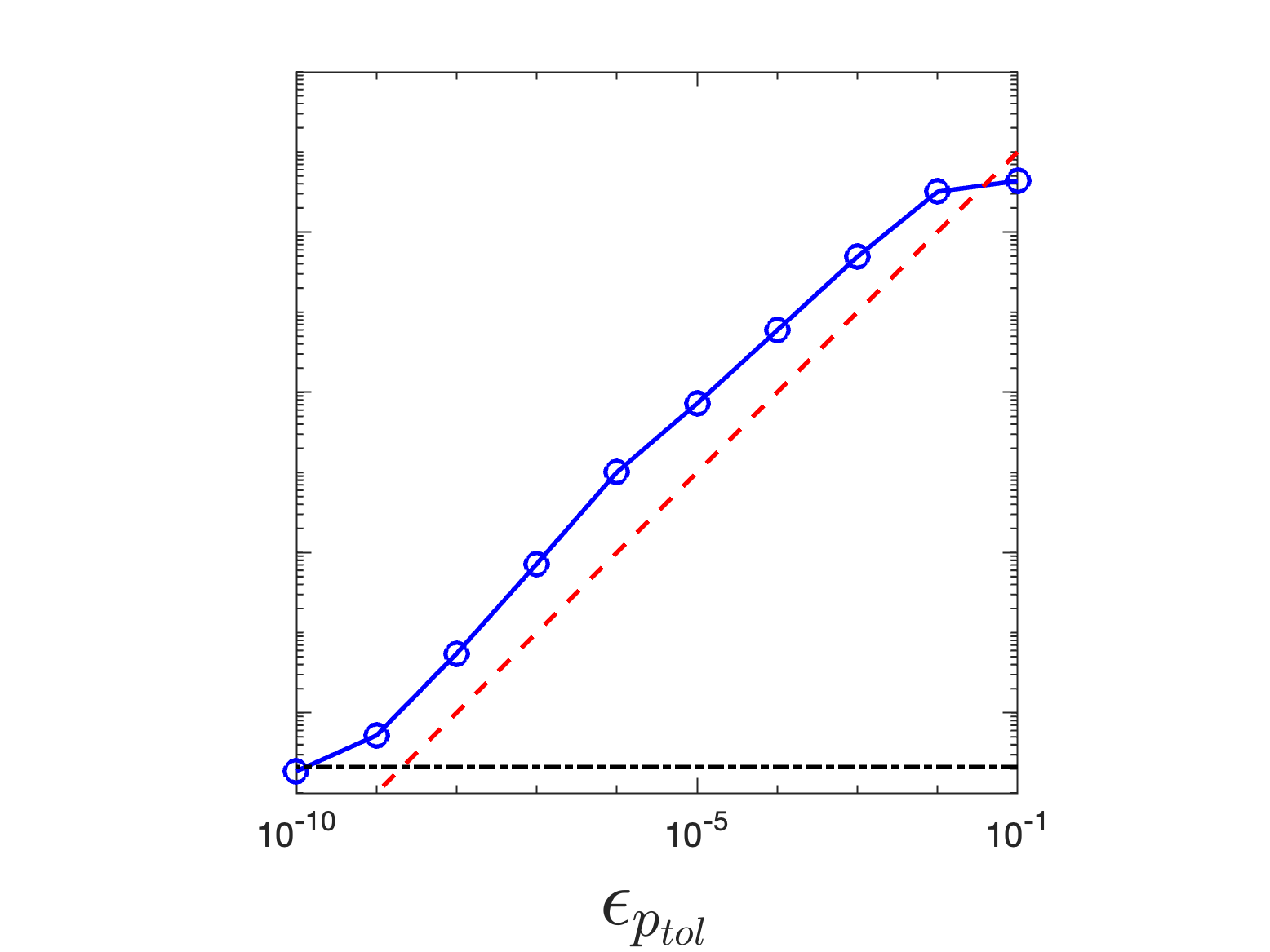}\hspace{-17pt}
		\includegraphics[width=2.5in,trim={0.0cm 0.0cm 1.75cm 0cm},clip]{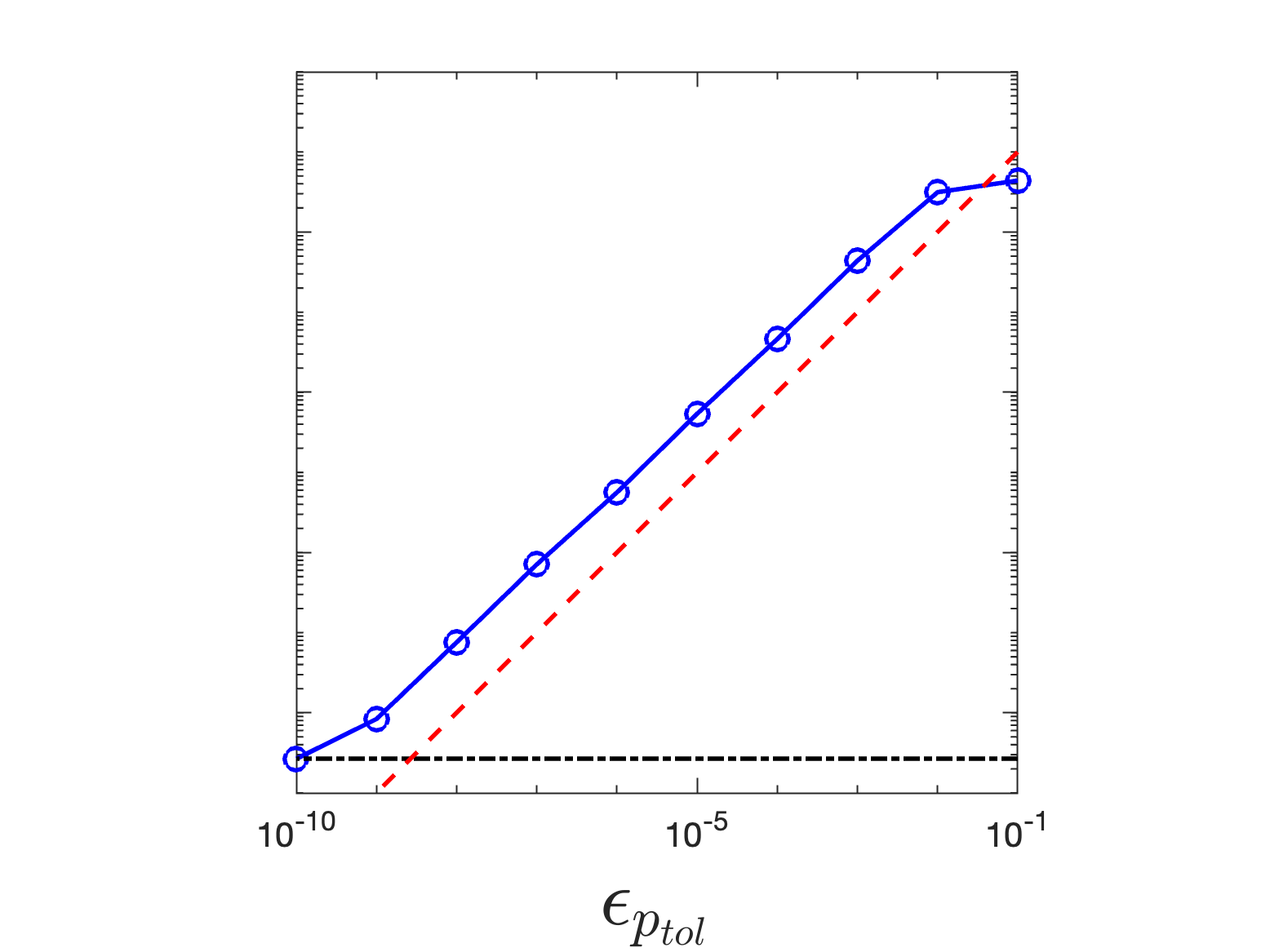}
	\end{adjustbox} 
       \caption{The error in maximum norm as a function of the tolerance $\epsilon_{p_{tol}}$ using the $p$-adaptive algorithm for the manufactured solution problem representing the propagation of a pulse. 
       		The left, middle and right plots represent respectively the mesh size $\Delta x = 1/12.5$, $\Delta x = 1/25$ and $\Delta x = 1/50$.
       		Here $\hat{\mathbf{U}}$ is a vector containing all variables.}
       \label{fig:p_adapt_1d_err}
\end{figure}
%***************************************
%***************************************
\begin{figure}   
	\centering
	\begin{adjustbox}{max width=1.0\textwidth,center}
		\includegraphics[width=2.5in,trim={0.0cm 0cm 1.75cm 0cm},clip]{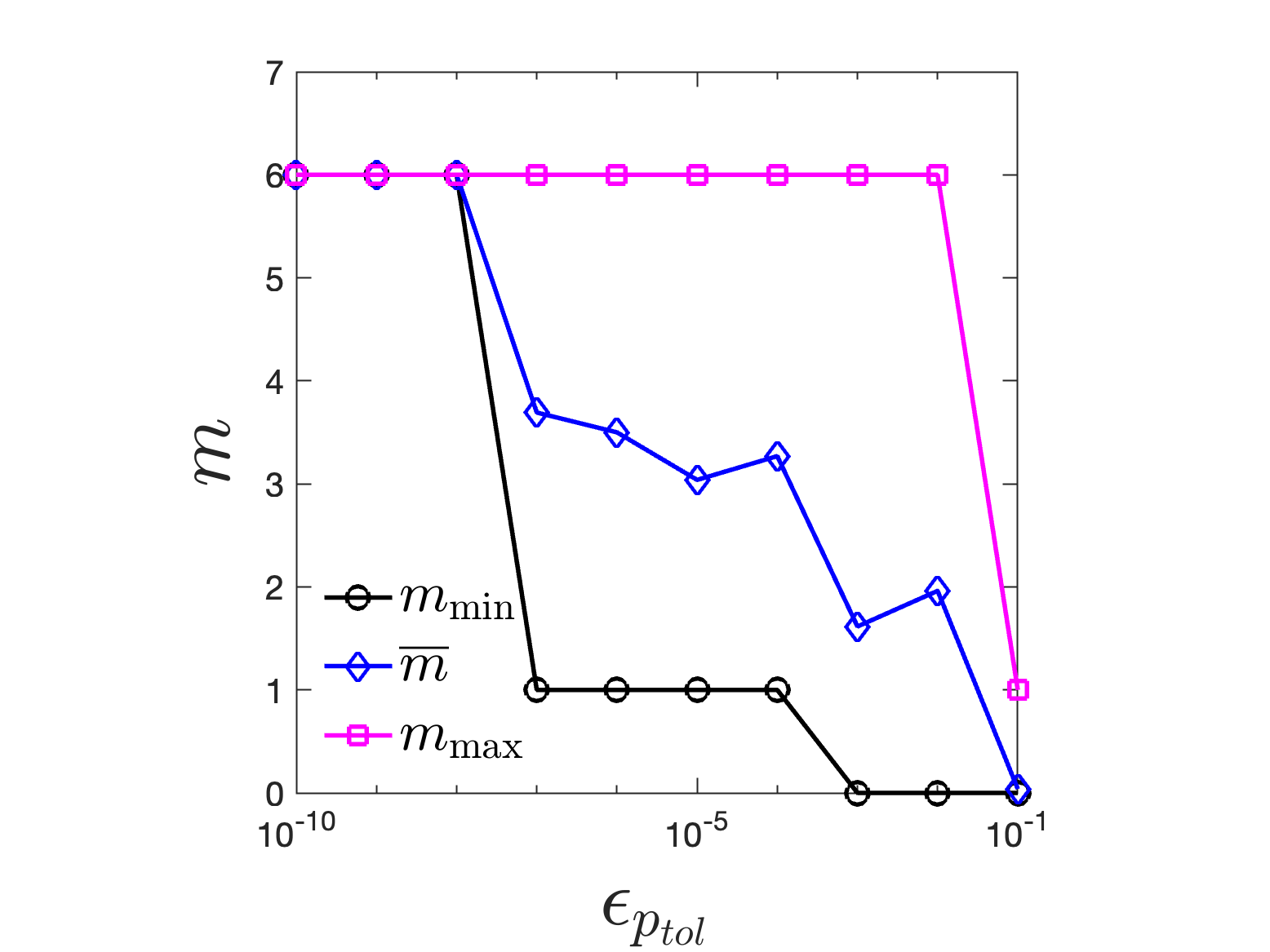} \hspace{-17pt}
		\includegraphics[width=2.5in,trim={0.0cm 0.0cm 1.75cm 0cm},clip]{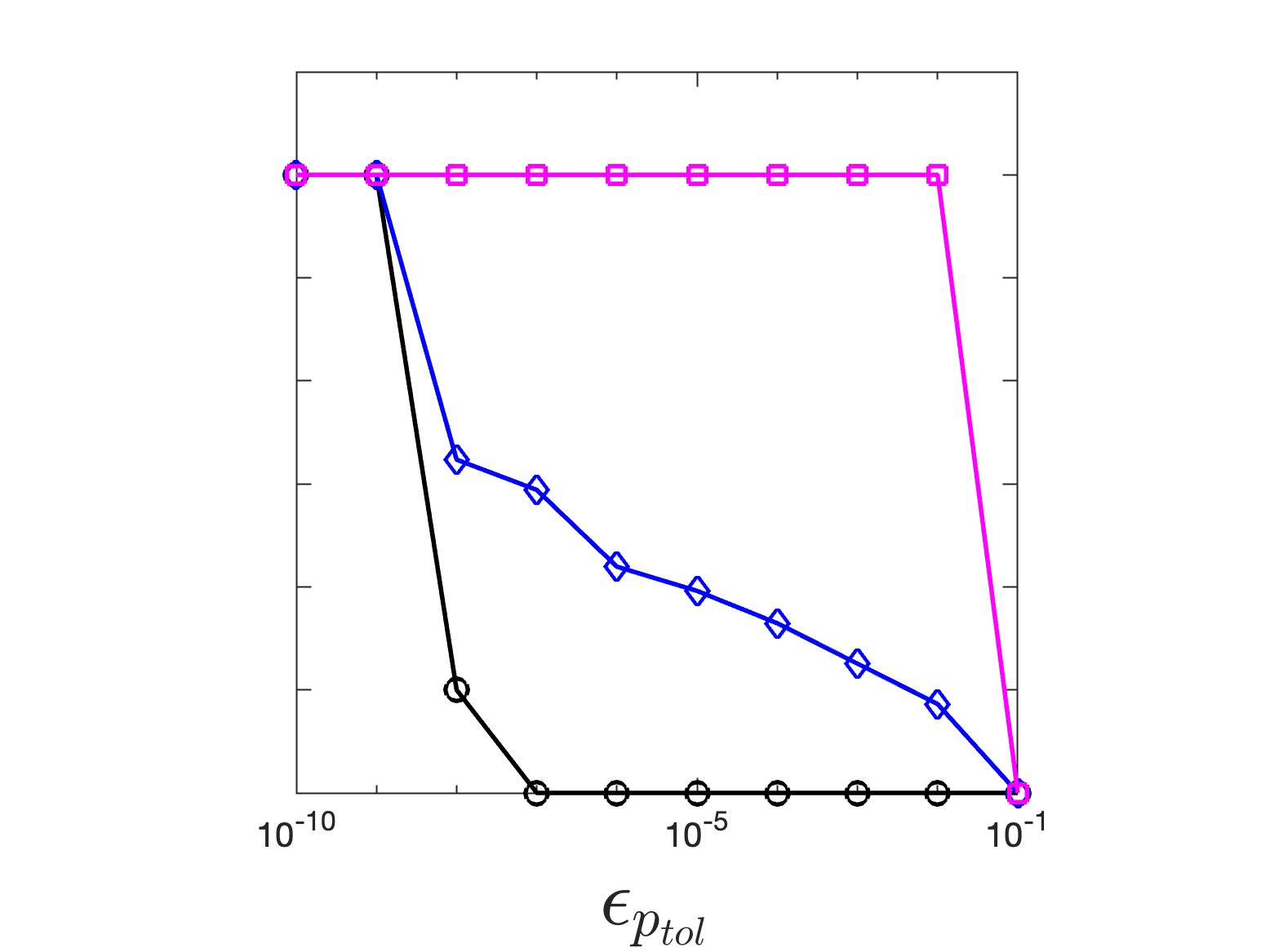}\hspace{-17pt}
		\includegraphics[width=2.5in,trim={0.0cm 0.0cm 1.75cm 0cm},clip]{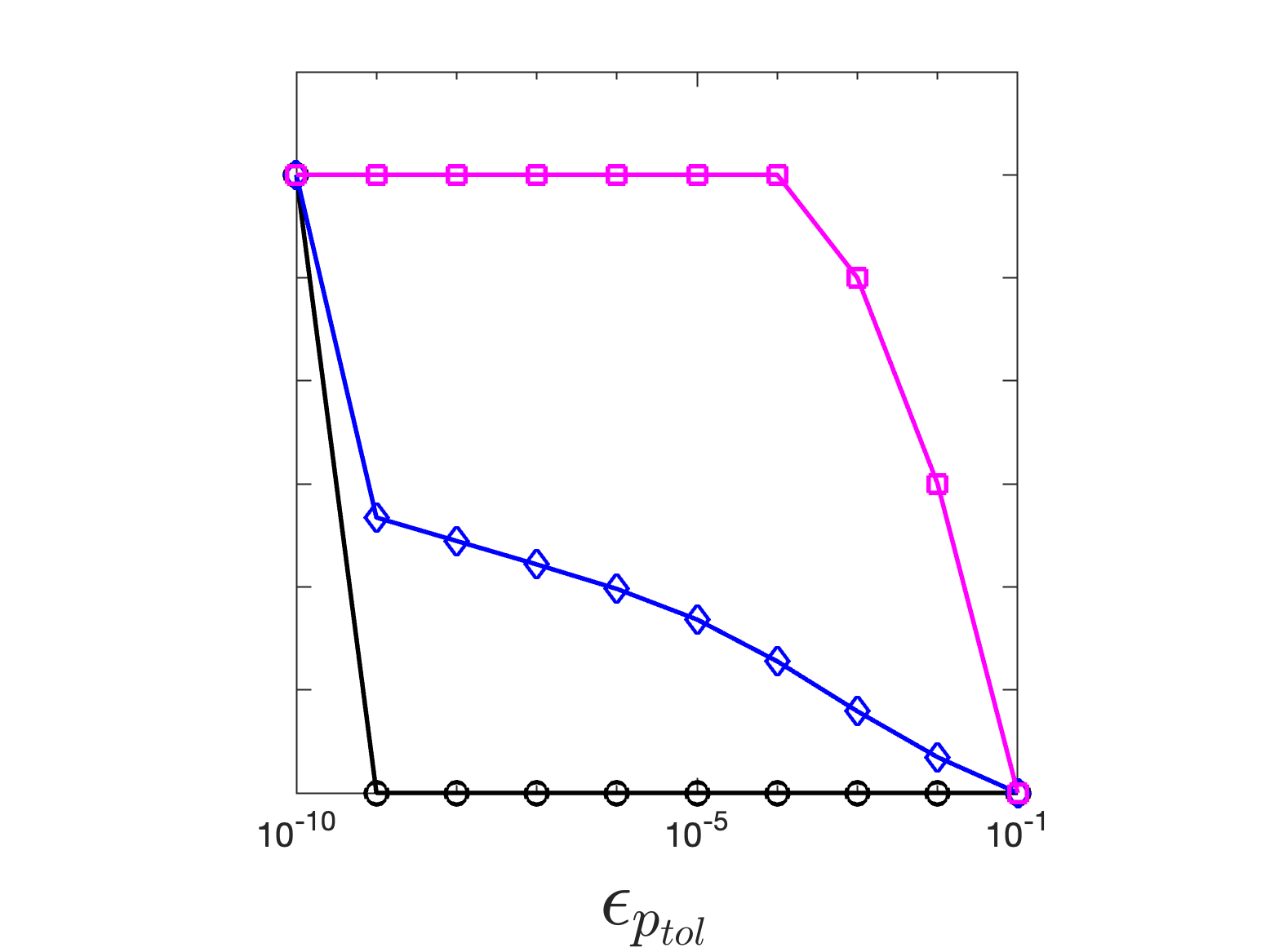}
	\end{adjustbox} 
       \caption{The values of $m$ as a function of the tolerance $\epsilon_{p_{tol}}$ using the $p$-adaptive algorithm for the manufactured solution problem representing the propagation of a pulse. 
       		The left, middle and right plots represent respectively the mesh size $\Delta x = 1/12.5$, $\Delta x = 1/25$ and $\Delta x = 1/50$.}
       \label{fig:p_adapt_1d_val_m}
\end{figure}
%***************************************

Figure~\ref{fig:p_adapt_1d_err} illustrates the error in maximum norm as a function of the tolerance $\epsilon_{p_{tol}}$ for different mesh sizes.
The error diminishes in proportion to the tolerance until it reaches the maximum possible accuracy for the considered mesh size using $m=6$.
In all cases, 
	the $p$-adaptive method produces approximations for which the error in maximum norm is larger than the considered tolerance. 
This indicates that the cutoff criteria \eqref{eq:cutoff_criteria} could be improved and this will be investigated in future work.  

Figure~\ref{fig:p_adapt_1d_val_m} illustrates the maximal, minimal and average values of $m$ as a function of the tolerance $\epsilon_{p_{tol}}$ at the final time.
The average values of $m$ increases as $\epsilon_{p_{tol}}$ diminishes while the minimal value of $m$ is between zero and one until the tolerance gets too close of 
	the maximum accuracy attainable with the considered mesh size.
Moreover, 
	the largest values of $m$ is localized at the pulse as illustrated in Figure~\ref{fig:values_m_1D}, 
	showing that the $p$-adaptive algorithm captures the localized features. 	
%**********************************************
 \begin{figure} 
 \centering
	\begin{adjustbox}{max width=1.0\textwidth,center}
	 \centering
		\includegraphics[width=2.5in,trim={1.5cm 0cm 1.75cm 0cm},clip]{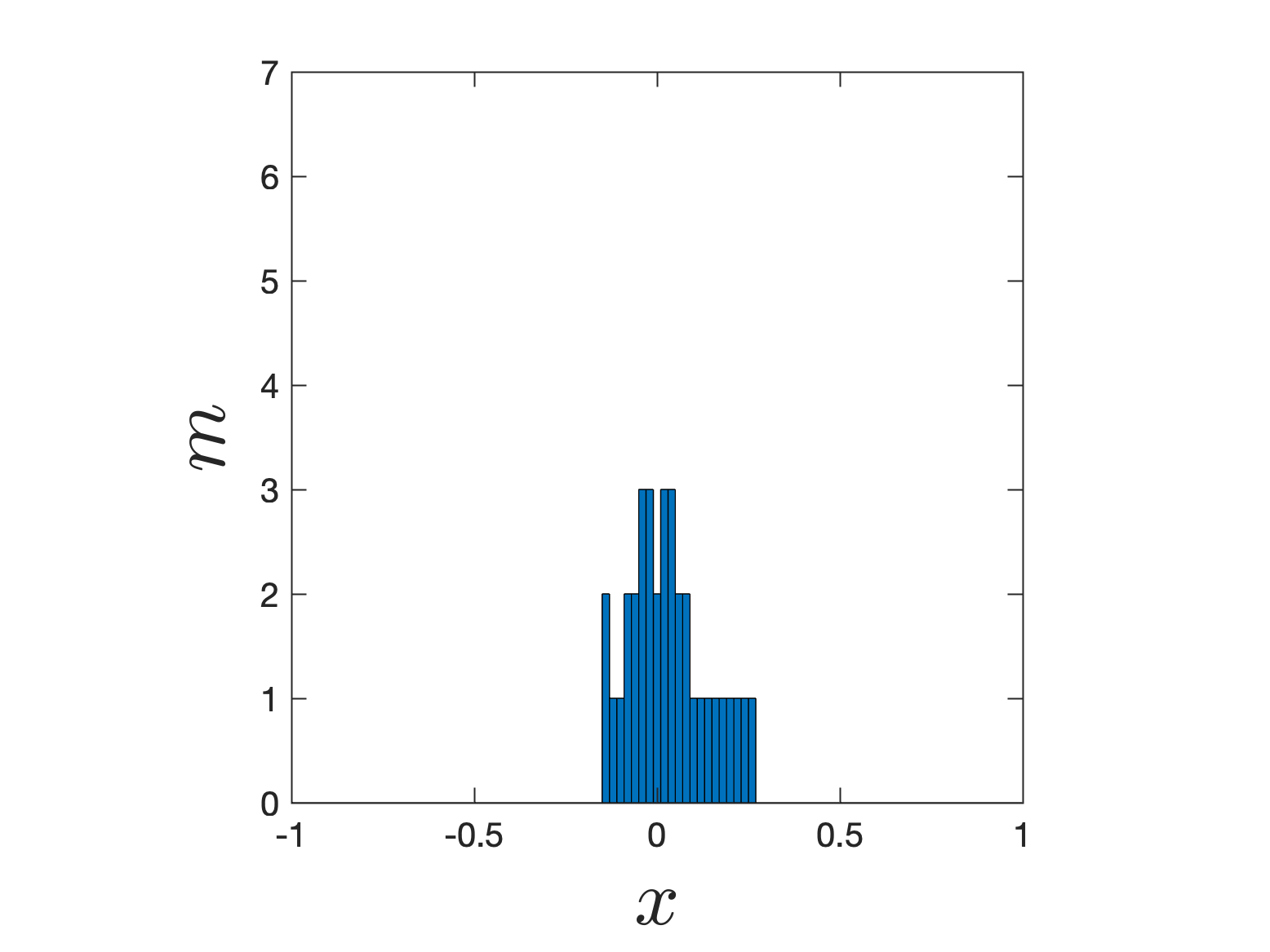}\hspace{-20pt}
	   	\includegraphics[width=2.5in,trim={1.5cm 0cm 1.75cm 0cm},clip]{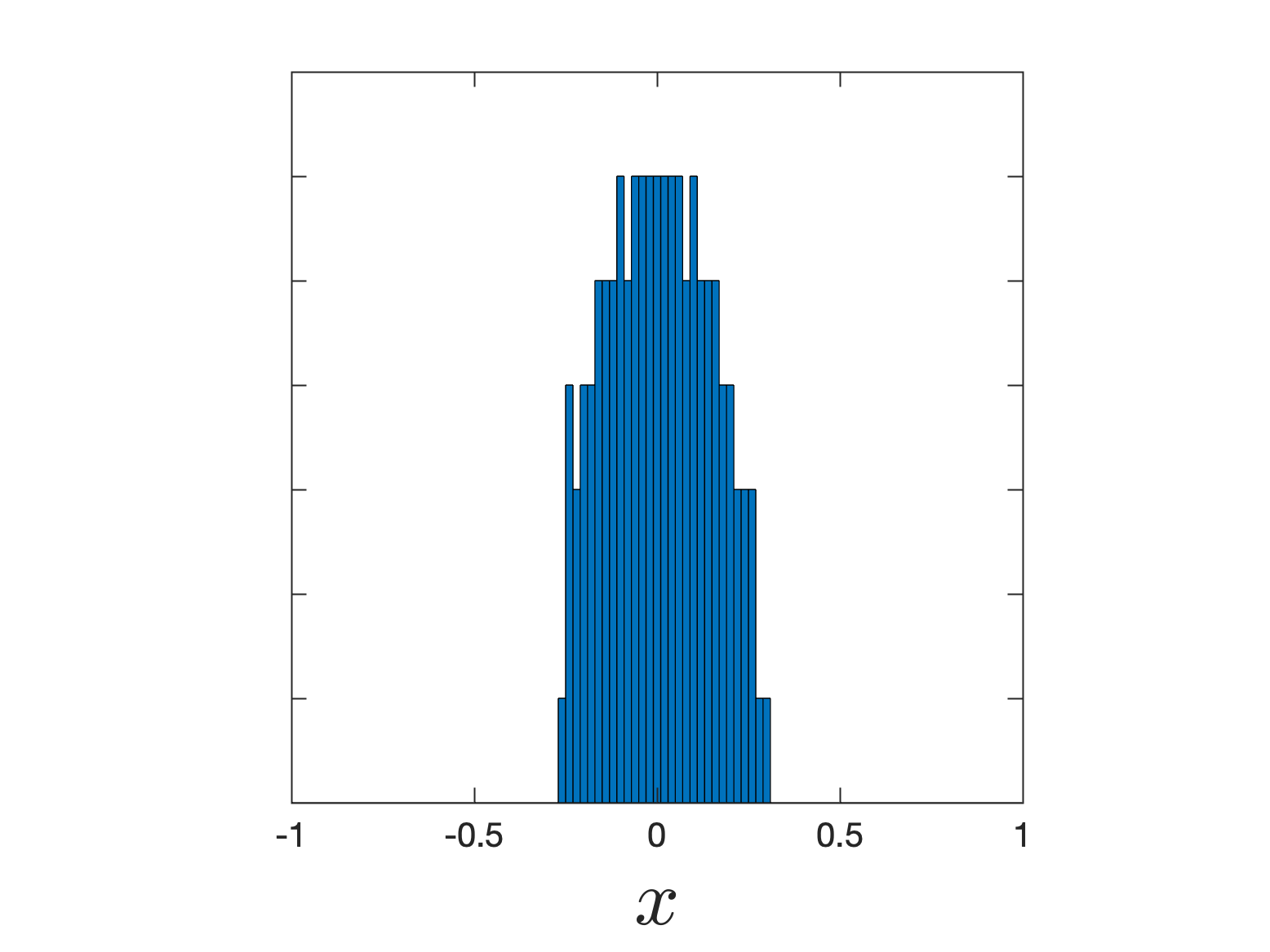}\hspace{-20pt}
	   	\includegraphics[width=2.5in,trim={1.5cm 0cm 1.75cm 0cm},clip]{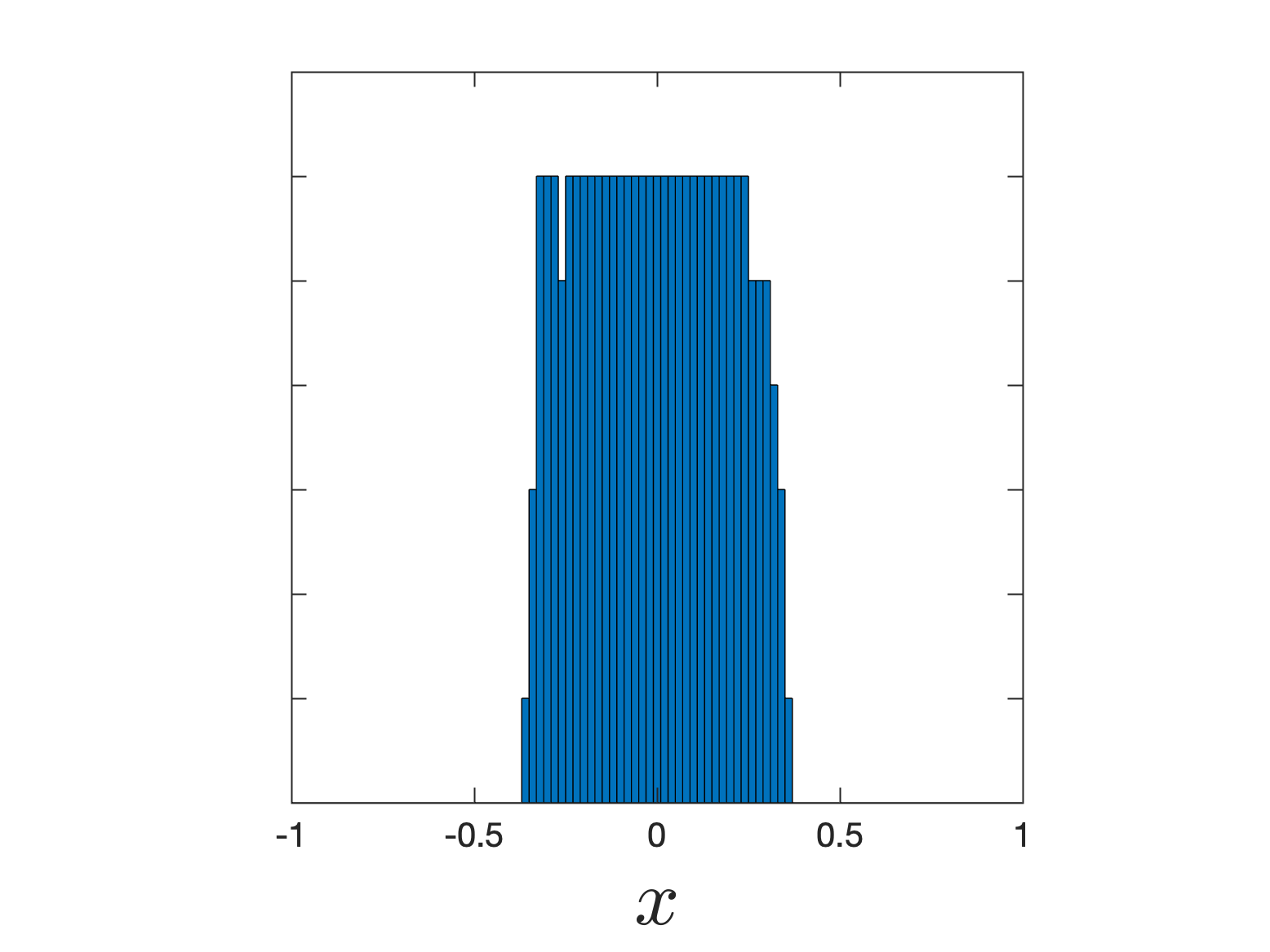} 
	\end{adjustbox}		
  \caption{Values of $m$ as a function of $x$ for different values of tolerance $\epsilon_{p_{tol}}$ using $\Delta x = 1/50$. 
  		The left, 
			middle and right plots are for the tolerances $10^{-2}$ , 
			$10^{-4}$ and $10^{-6}$.}
   \label{fig:values_m_1D}
\end{figure}
%*************************************** 

Figure~\ref{fig:err_vs_nb_dof_1D} illustrates the error in maximum norm as a function of the number of degrees of freedom. 
The $p$-adaptive method typically requires fewer degrees of freedom than a fixed value of $m$ for the same accuracy.
Since the Hermite method requires $(m+1)^d$ degrees of freedom in $\mathbb{R}^d$ for each node and for each variable, 
	we expect better performance of the $p$-adaptive method in higher dimensions, 
	as demonstrated with the two dimensional examples in the next subsection, 
	and for larger domains. 
%**********************************************
 \begin{figure} 
 \centering
	\begin{adjustbox}{max width=1.0\textwidth,center}
	 \centering
		\includegraphics[width=2.5in,trim={1.5cm 0cm 1.75cm 0cm},clip]{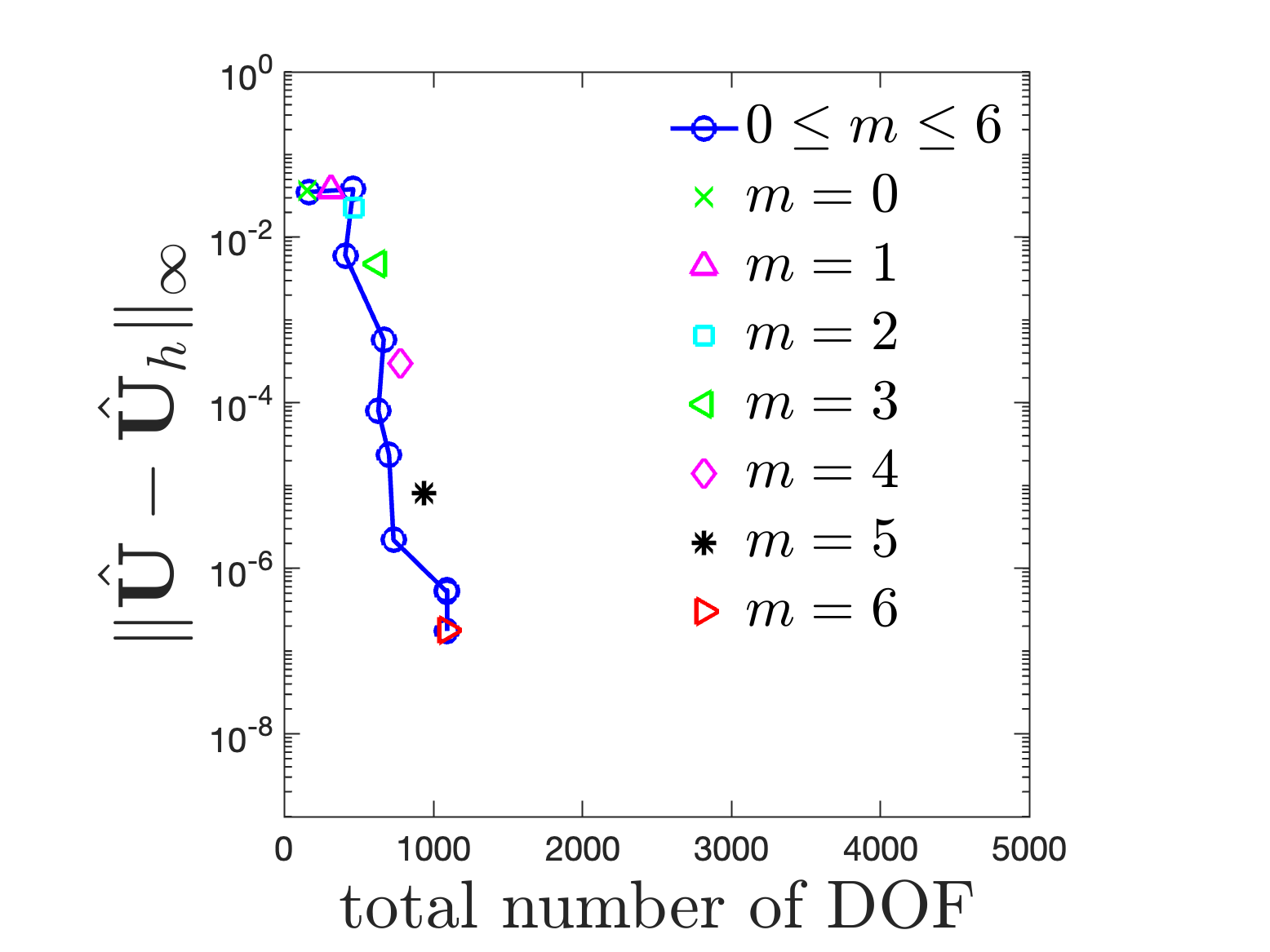}\hspace{-17pt}
	   	\includegraphics[width=2.5in,trim={1.5cm 0cm 1.75cm 0cm},clip]{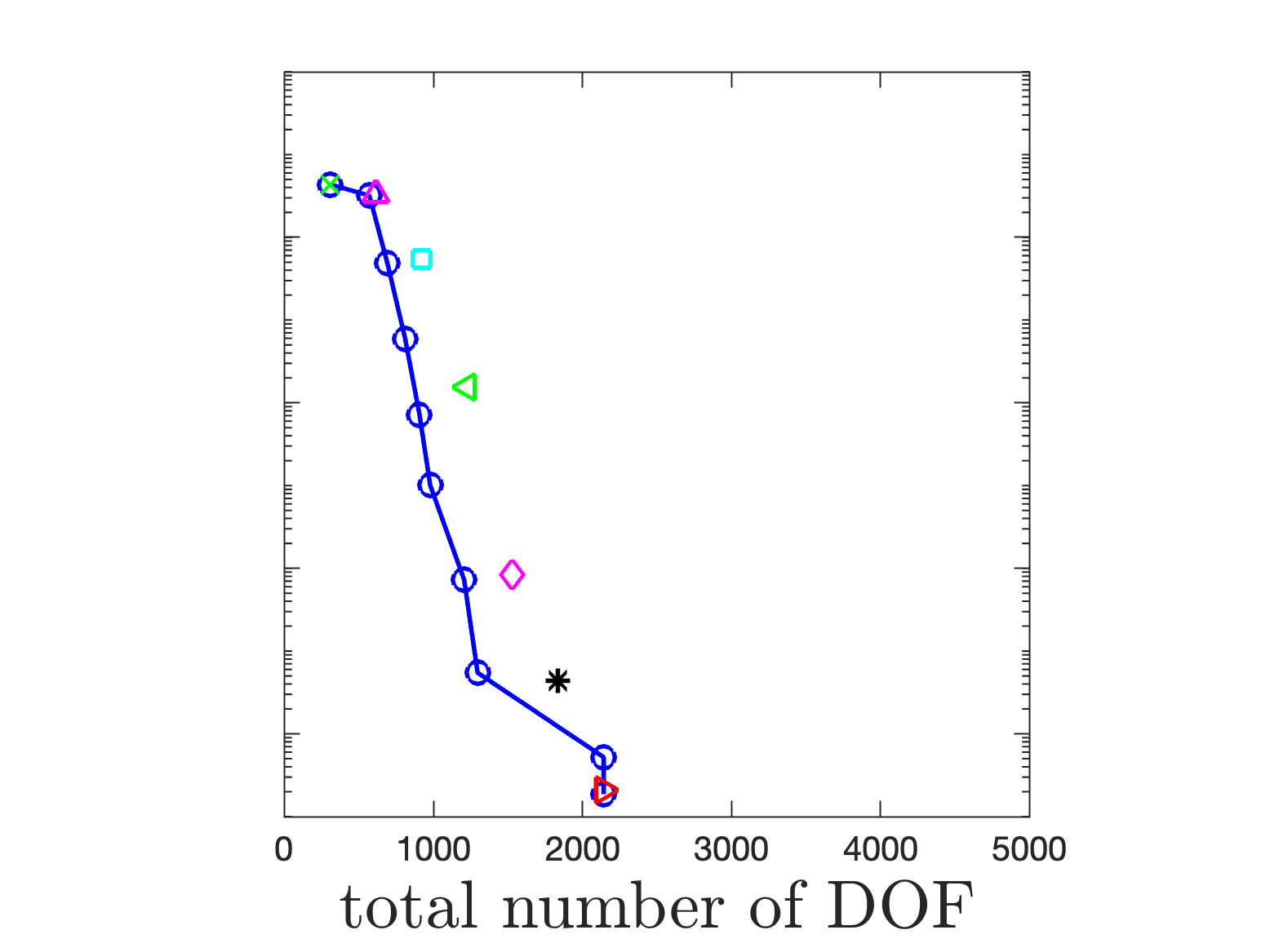}\hspace{-17pt}
	   	\includegraphics[width=2.5in,trim={1.5cm 0cm 1.75cm 0cm},clip]{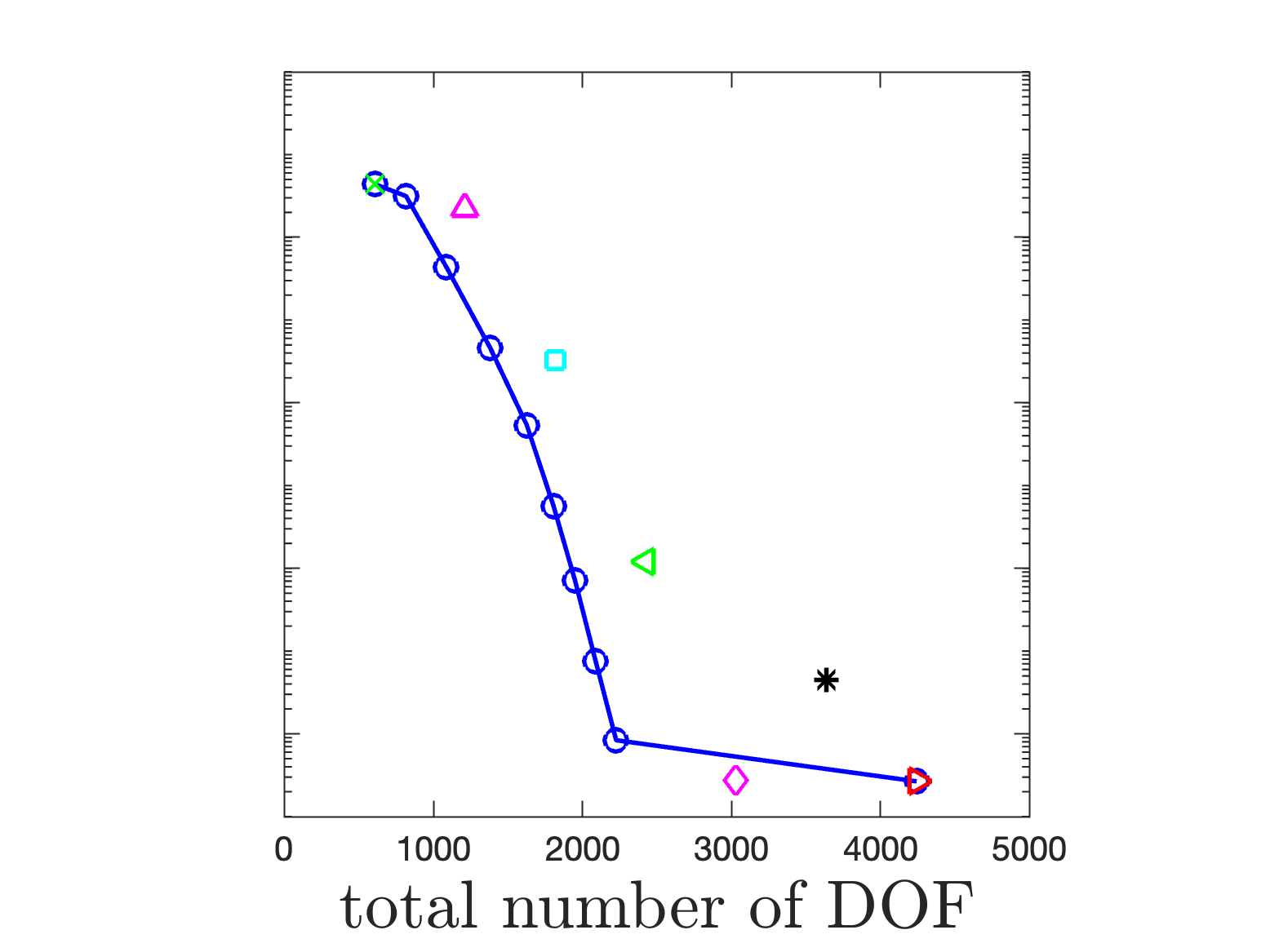} 
	\end{adjustbox}		
  \caption{The error in maximum norm as a function of the total number of degrees of freedom (DOF) using the $p$-adaptive algorithm and various fixed values of $m$ for the manufactured solution problem representing the propagation of a pulse. 
       		The left, middle and right plots represent respectively the mesh size $\Delta x = 1/12.5$, $\Delta x = 1/25$ and $\Delta x = 1/50$.
       		Here $\hat{\mathbf{U}}$ is a vector containing all variables.}
   \label{fig:err_vs_nb_dof_1D}
\end{figure}
%*************************************** 

The number of time sub-steps,
	given by \eqref{eq:nb_substeps_equation}, 
	is important as the value of $m$ increases. 
In this situation, 
	we investigate the impact of a smaller number of sub-steps on the accuracy. 
We therefore limit the maximum number of time sub-steps to $N_{\tau_{\max}} = 2-5$, 
	significantly reducing the number of time sub-steps for $m=5-6$.  
Note that $N_{\tau_{\max}} = 1$ leads to an unstable method.

Figure~\ref{fig:p_adapt_1d_N_tau_err} illustrates the error in maximum norm as a function of the tolerance for different values of $N_{\tau_{\max}}$. 
In all cases, 
	as the tolerance $\epsilon_{p_{tol}}$ decreases, 
	the error also decreases until it reaches a plateau due to the time error.
Moreover,
	as $N_{\tau_{\max}}$ increases, 
	the value corresponding to the plateau in the error diminishes and the error curve approaches the one obtained with $N_\tau^*$. 
Finally,
	as the mesh size diminishes, 
	we can reach greater accuracy for a given $N_{\tau_{\max}}$.
Note that the choice of $N_{\tau_{\max}}$ does not impact significantly the results for $\epsilon_{p_{tol}} \geq 10^{-5}$.
%***************************************
\begin{figure}   
	\centering
	\begin{adjustbox}{max width=1.0\textwidth,center}
		\includegraphics[width=2.5in,trim={0.0cm 0cm 1.75cm 0cm},clip]{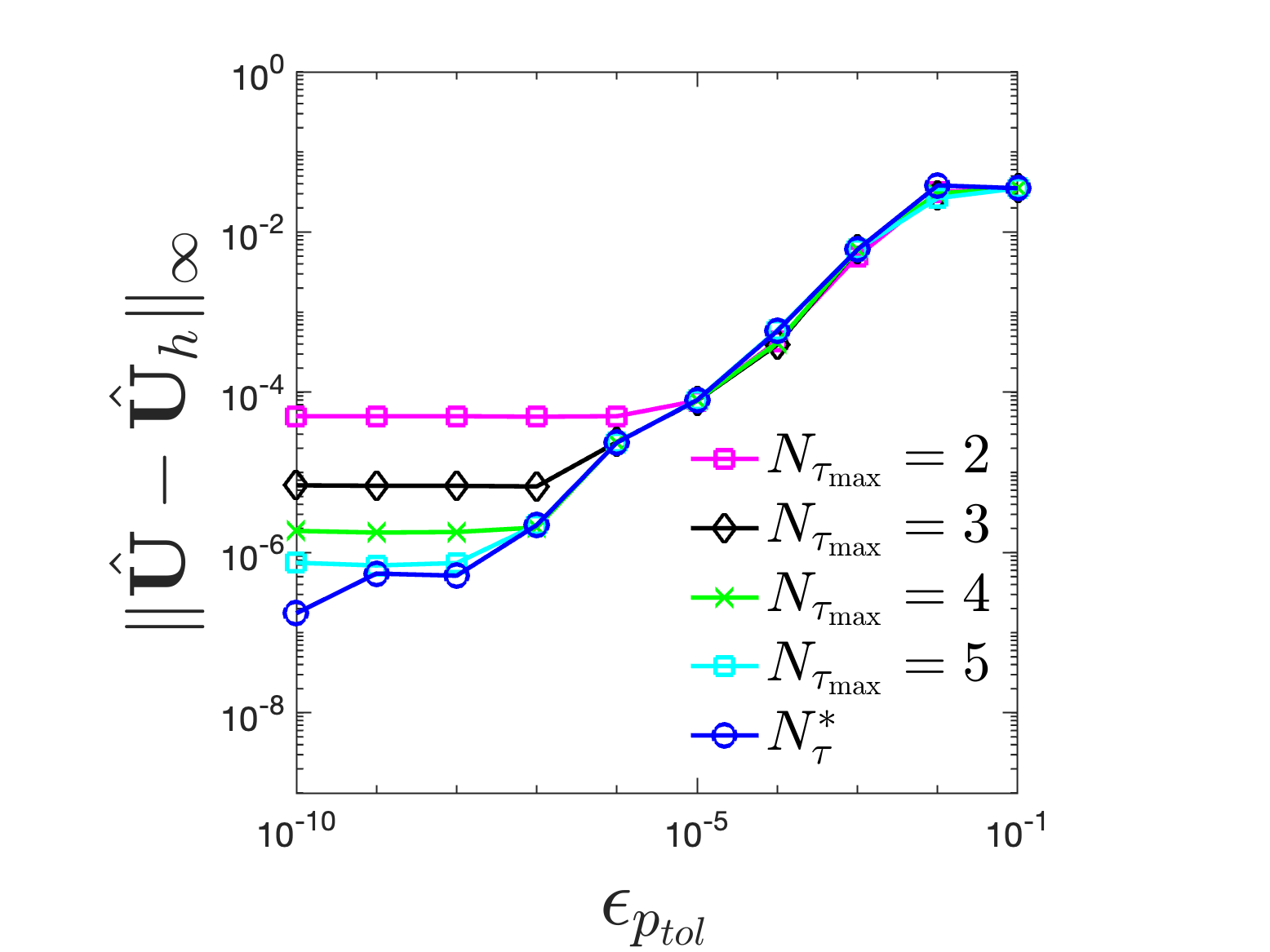} \hspace{-17pt}
		\includegraphics[width=2.5in,trim={0.0cm 0.0cm 1.75cm 0cm},clip]{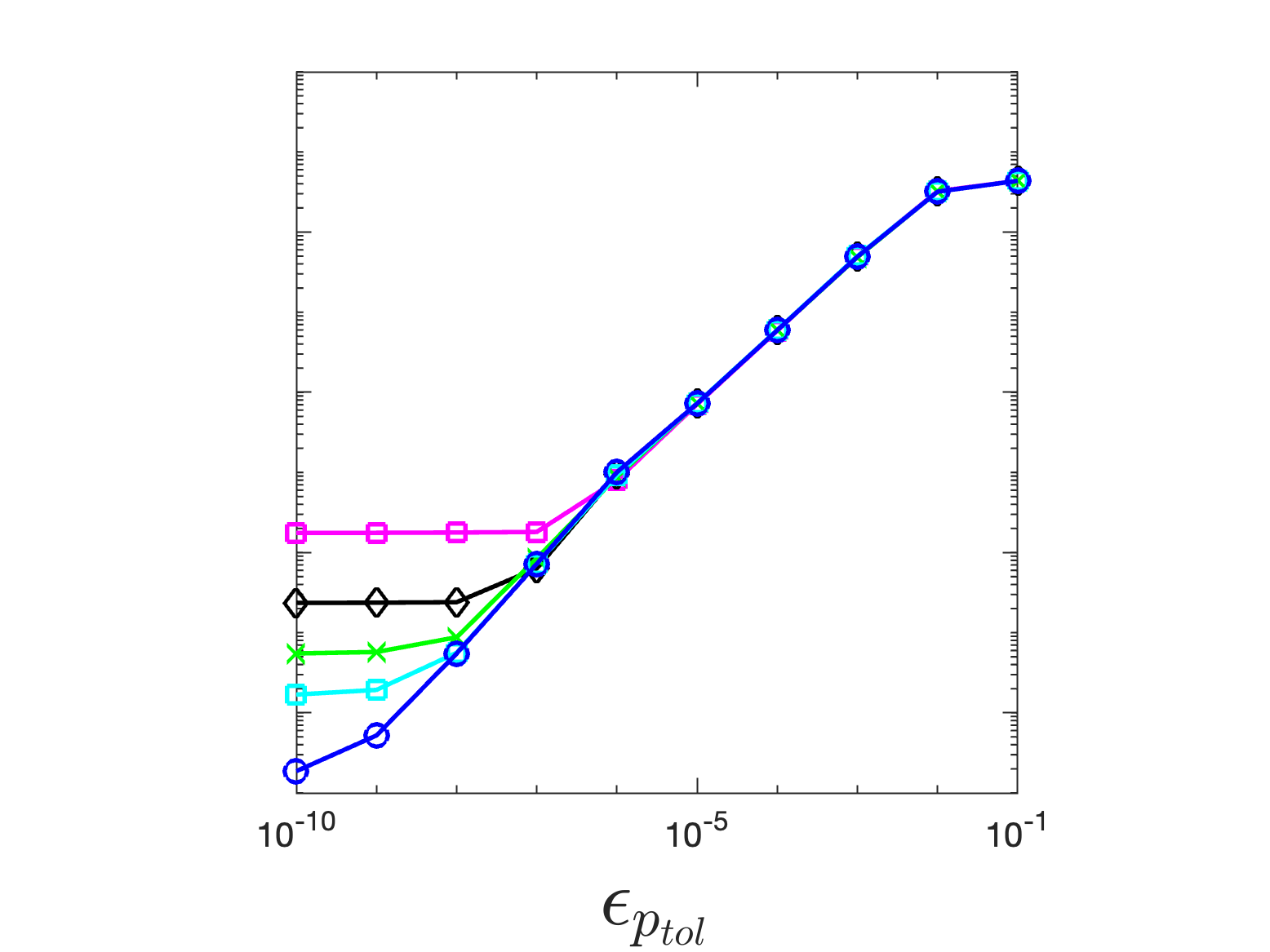}\hspace{-17pt}
		\includegraphics[width=2.5in,trim={0.0cm 0.0cm 1.75cm 0cm},clip]{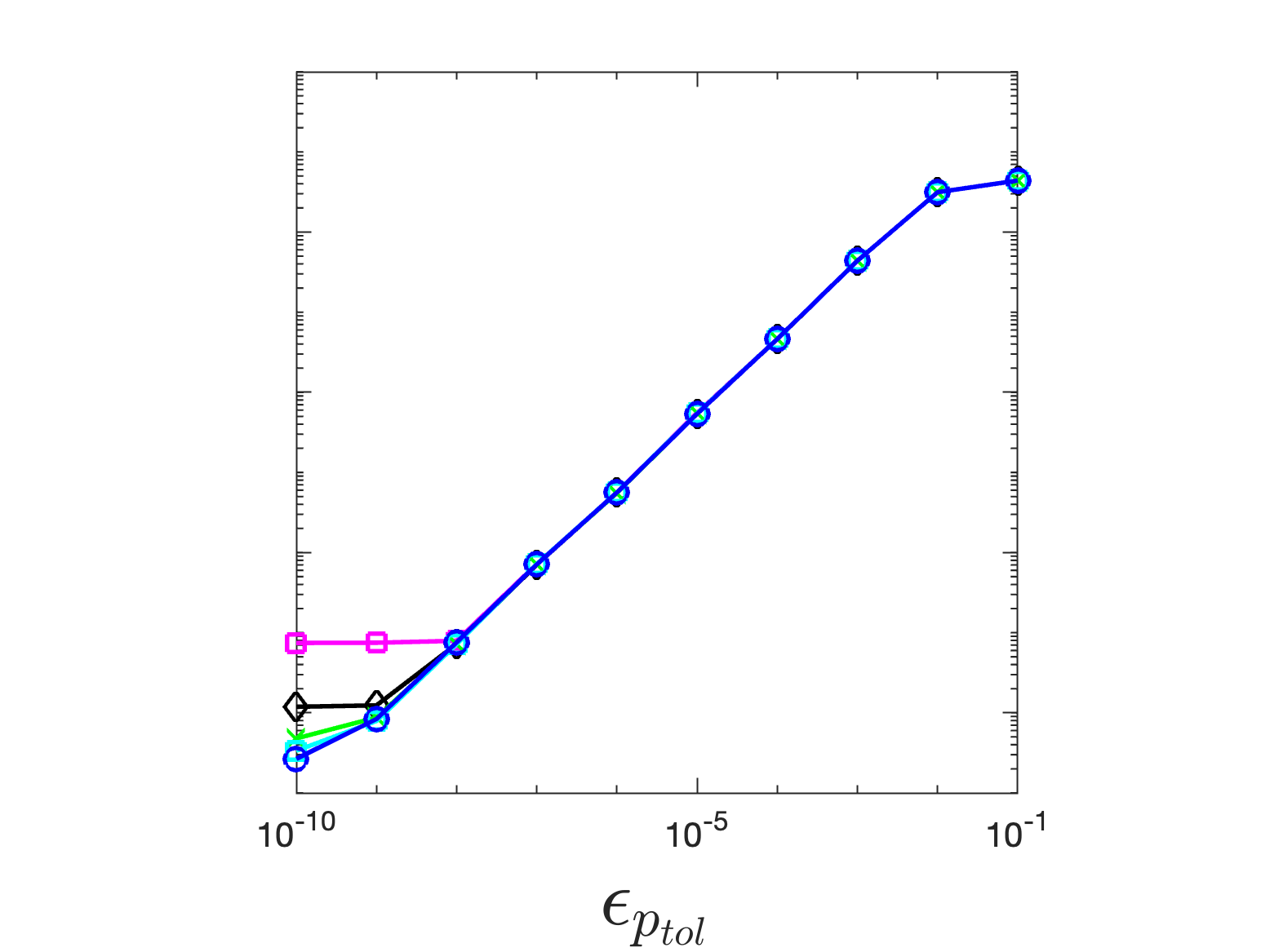}
	\end{adjustbox} 
       \caption{The error in maximum norm as a function of the tolerance $\epsilon_{p_{tol}}$ using different $N_{\tau_{\max}}$ and the $p$-adaptive algorithm for the manufactured solution problem representing the propagation of a pulse. 
       		The left, middle and right plots represent respectively the mesh size $\Delta x = 1/12.5$, $\Delta x = 1/25$ and $\Delta x = 1/50$.
       		Here $\hat{\mathbf{U}}$ is a vector containing all variables.}
       \label{fig:p_adapt_1d_N_tau_err}
\end{figure}
%***************************************	
%***************************************
\begin{figure}   
	\centering
	\begin{adjustbox}{max width=1.0\textwidth,center}
		\includegraphics[width=2.5in,trim={0.0cm 0cm 1.75cm 0cm},clip]{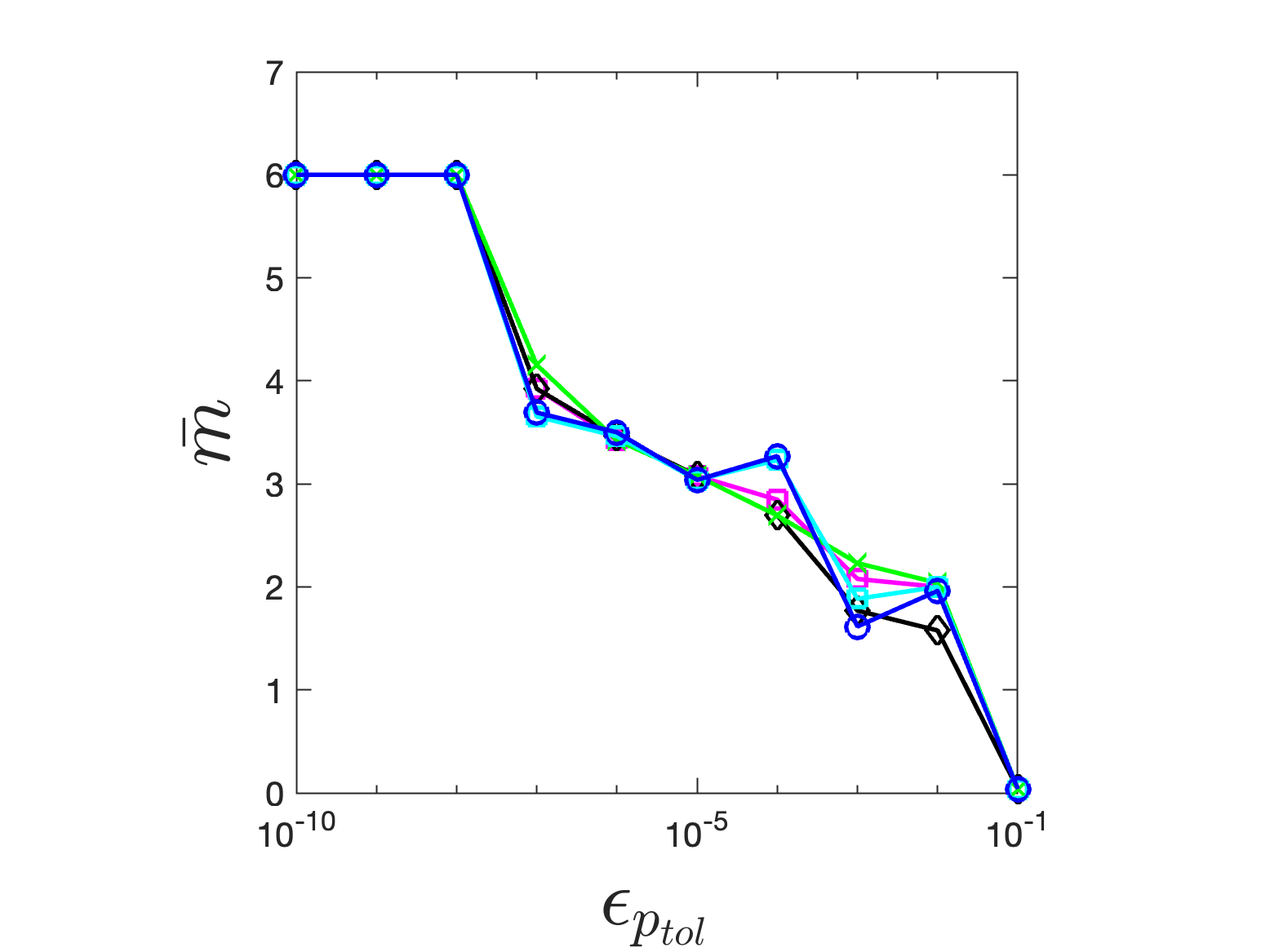} \hspace{-17pt}
		\includegraphics[width=2.5in,trim={0.0cm 0.0cm 1.75cm 0cm},clip]{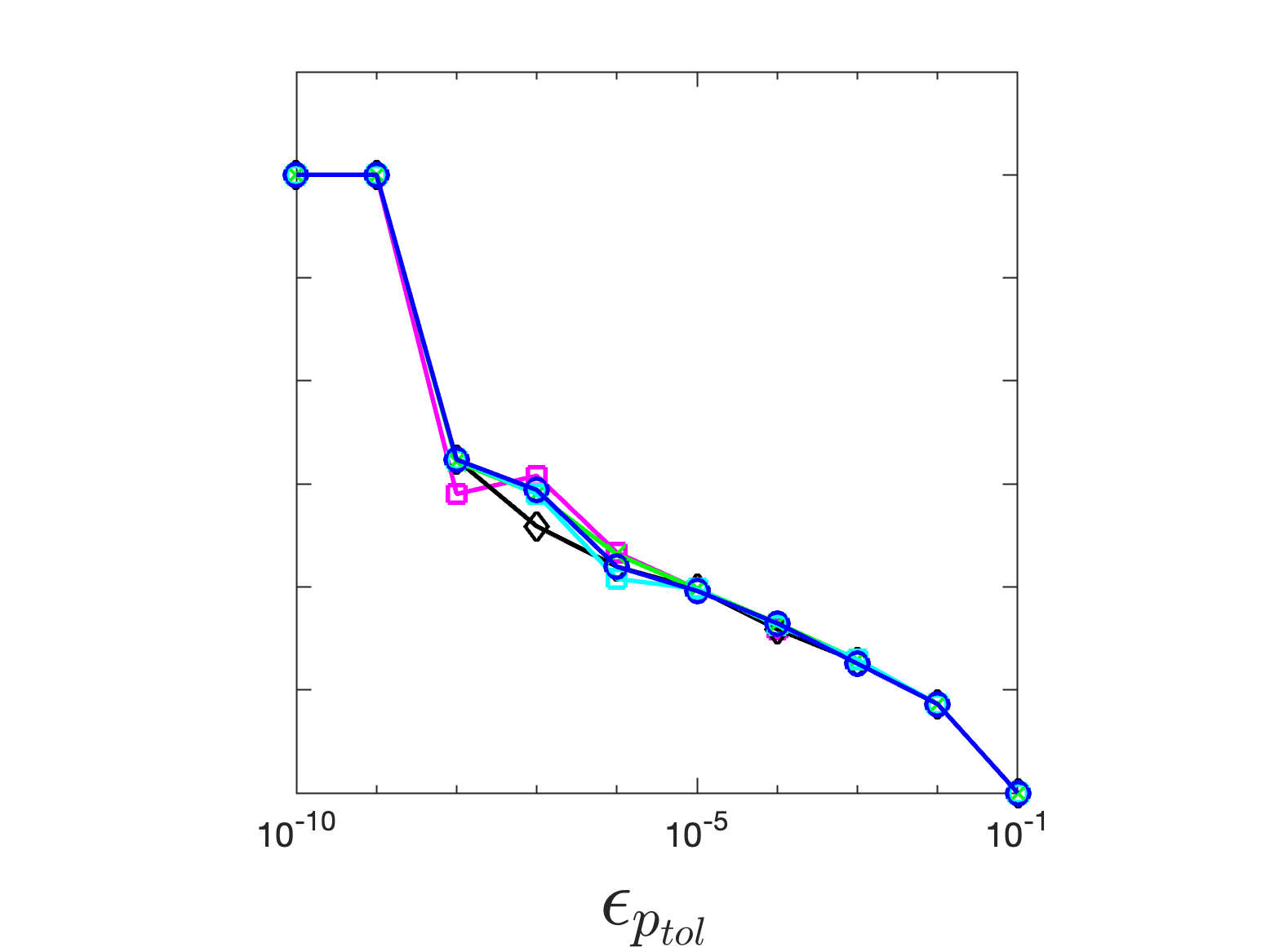}\hspace{-17pt}
		\includegraphics[width=2.5in,trim={0.0cm 0.0cm 1.75cm 0cm},clip]{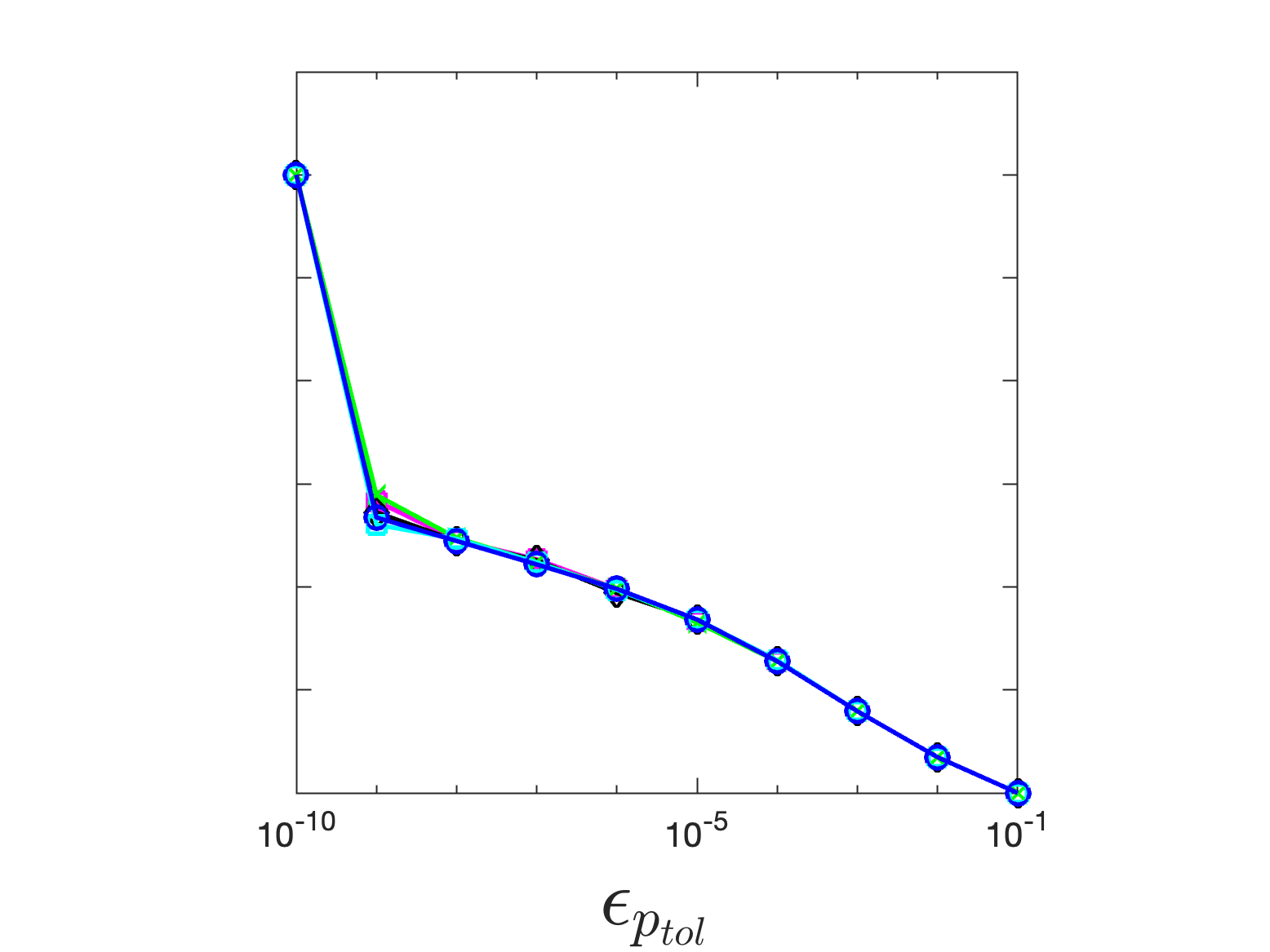}
	\end{adjustbox} 
       \caption{The average value of $m$ as a function of the tolerance $\epsilon_{p_{tol}}$ using different $N_{\tau_{\max}}$ and the $p$-adaptive algorithm for the manufactured solution problem representing the propagation of a pulse. 
       		The left, middle and right plots represent respectively the mesh size $\Delta x = 1/12.5$, $\Delta x = 1/25$ and $\Delta x = 1/50$.}
       \label{fig:p_adapt_1d_N_tau_average_m}
\end{figure}
%***************************************	
Figure~\ref{fig:p_adapt_1d_N_tau_average_m} illustrates the average value of $m$ as a function of the tolerance. 
The average value of $m$ remains roughly the same for all cases, 
	showing that the choice of $N_{\tau_{\max}}$ does not impact significantly the order adaptive algorithm. 
Note that the maximal value of $m$ remains the same for all considered $N_{\tau_{\max}}$, 
	while the minimal value varies slightly for $\Delta x = 1/12.5$ and $\Delta x =1/25$.

As a last example in one space dimension, 
	we consider a pulse as the initial condition for the electric field, 
	that is 
\begin{equation} \label{eq:cnd_initial_1d}
	\hat{E}(x,0) = A \,\mbox{sech}(x)\cos(\Omega_0 x),
\end{equation}
	and all the other variables are set to zero. 
Here $\Omega_0 = 12.57$ and the physical parameters are $\mu=1$, 
	$\epsilon = 1$,
	$\epsilon_\infty = 2.25$, 
	$a=0.07$,
	$\omega_0 = 5.84$, 
	$\omega_p = \omega_0\sqrt{5.25-\epsilon_\infty}$,
	$\omega_v = 1.28$, 
	$\gamma = 1.1685\times10^{-5}$,
	$\gamma_v = 29.2/32$ and $\theta = 0.3$. 
Note that these are the same parameters used in \cite{Gilles2000,Bokil2017} to generate solitons, 
	however they use the equation \eqref{eq:cnd_initial_1d} as a function in time for the boundary condition of the electric field. 
	
The domain $\Omega = [-100,100]$ and the time interval is $I=[0,200]$.
We set $\epsilon_{p_{tol}} = 10^{-4}$, 
	$N_{\tau_{\max}} = 3$, 
	$0\leq m \leq 6$, 
	$\Delta x = 1/25$ for a total of 5000 cells and $\Delta t = \Delta x/2$. 
Figure~\ref{fig:gaussian_pulse_A_1} and Figure~\ref{fig:gaussian_pulse_A_2} illustrate the electric field and the distribution of $m$ for $x\in[0,100]$ at different times, 
    with respectively $A=1$ and $A=2$. 
Although we do not expect to have the same results as in \cite{Gilles2000,Bokil2017} in the current setting, 
	we can still observe a daughter pulse that travels ahead of the main pulse.  
Note that the pulses are well captured by the $p$-adaptive method.
Finally,
	Figure~\ref{fig:energy_1d} illustrates the evolution of the energy $\mathcal{E}$ in time. 
As expected, 
	the energy dissipates with time, 
	suggesting that the method is stable. 

%***************************************
\begin{figure}   
	\centering
	\begin{adjustbox}{max width=1.0\textwidth,center}
		\includegraphics[width=2.5in,trim={0.0cm 0cm 1.75cm 0cm},clip]{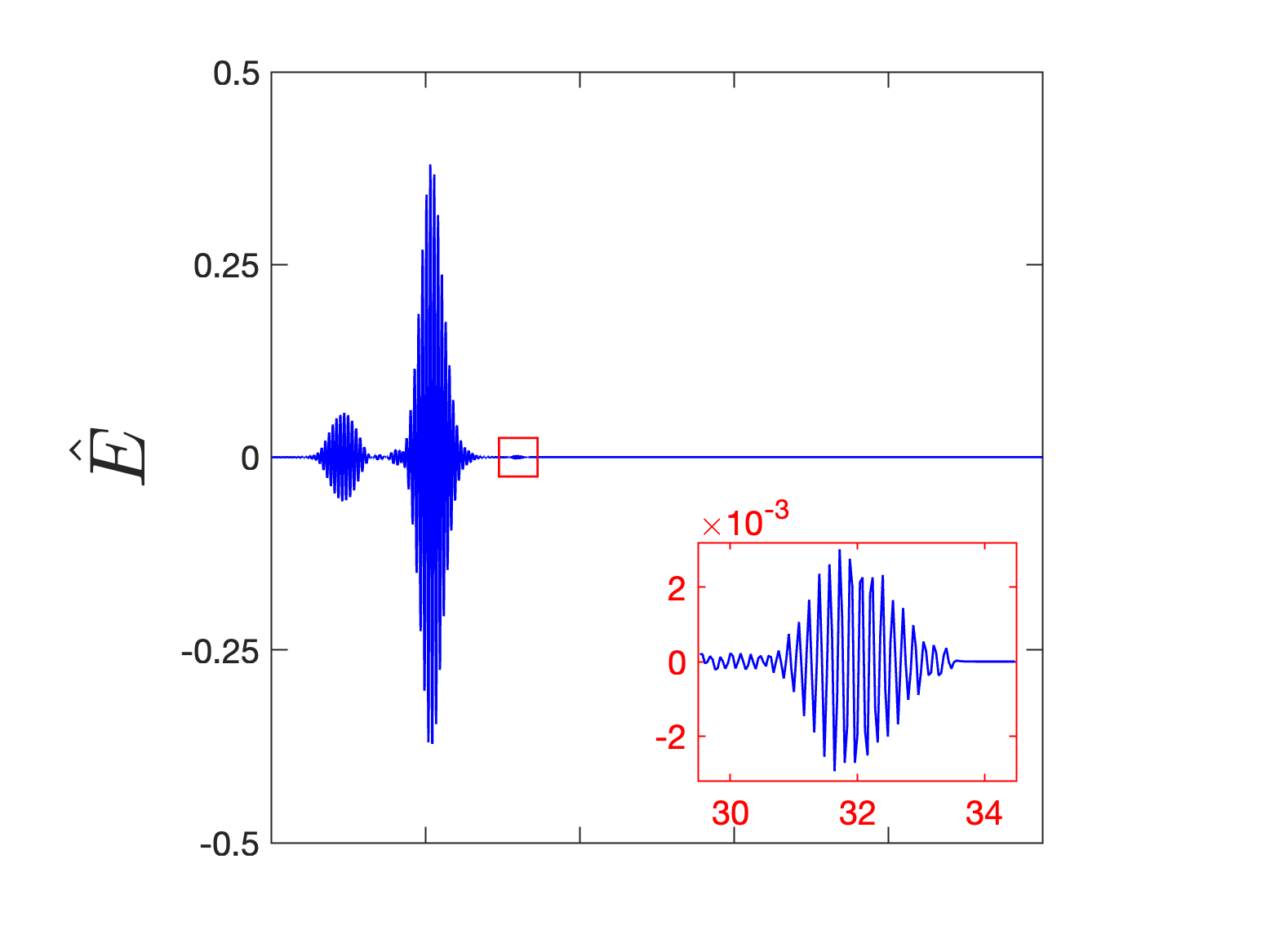} \hspace{-17pt}
		\includegraphics[width=2.5in,trim={0.0cm 0.0cm 1.75cm 0cm},clip]{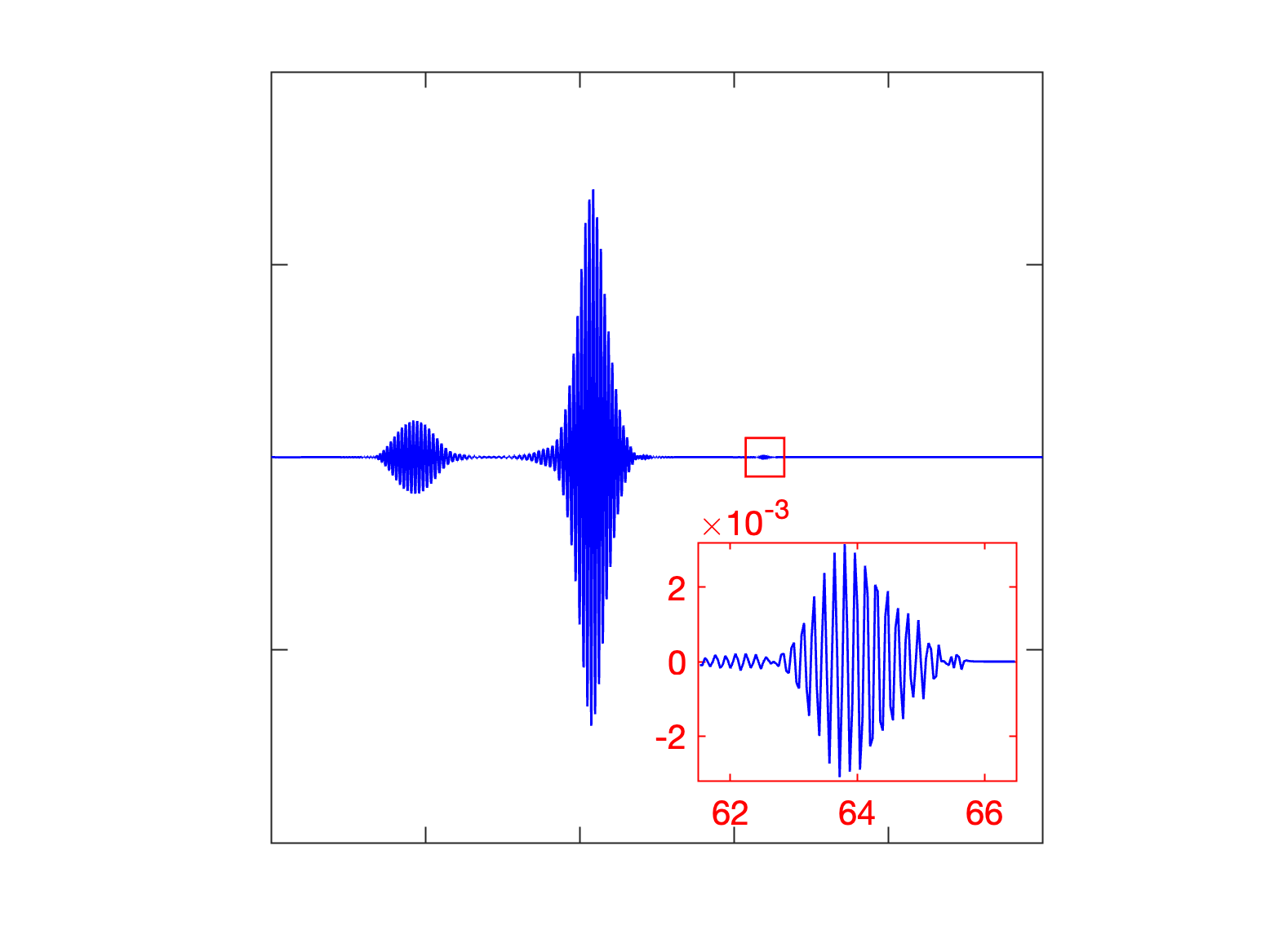}\hspace{-17pt}
		\includegraphics[width=2.5in,trim={0.0cm 0.0cm 1.75cm 0cm},clip]{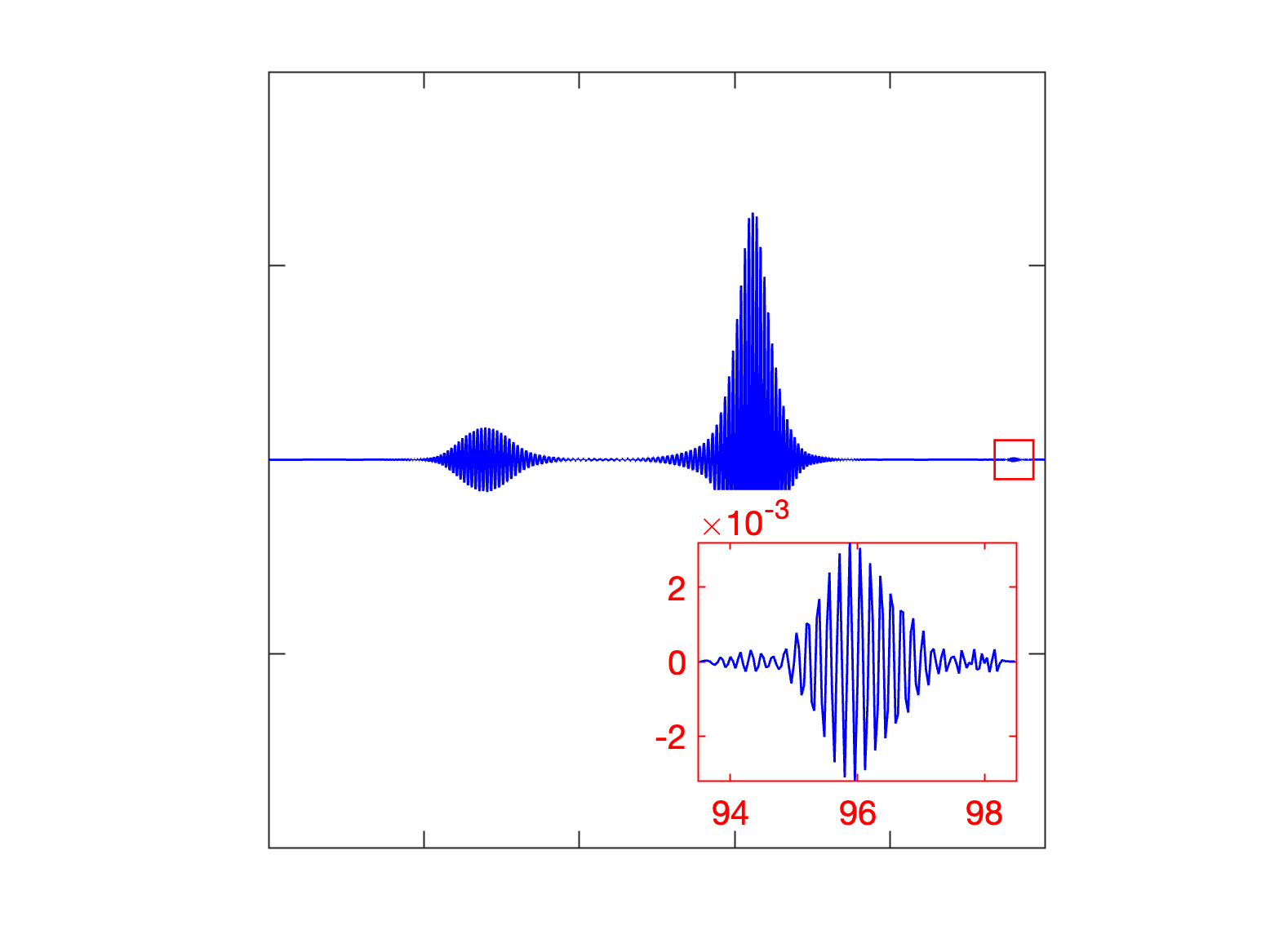}
	\end{adjustbox} 
	\begin{adjustbox}{max width=1.0\textwidth,center}
		\includegraphics[width=2.5in,trim={0.0cm 0cm 1.75cm 0cm},clip]{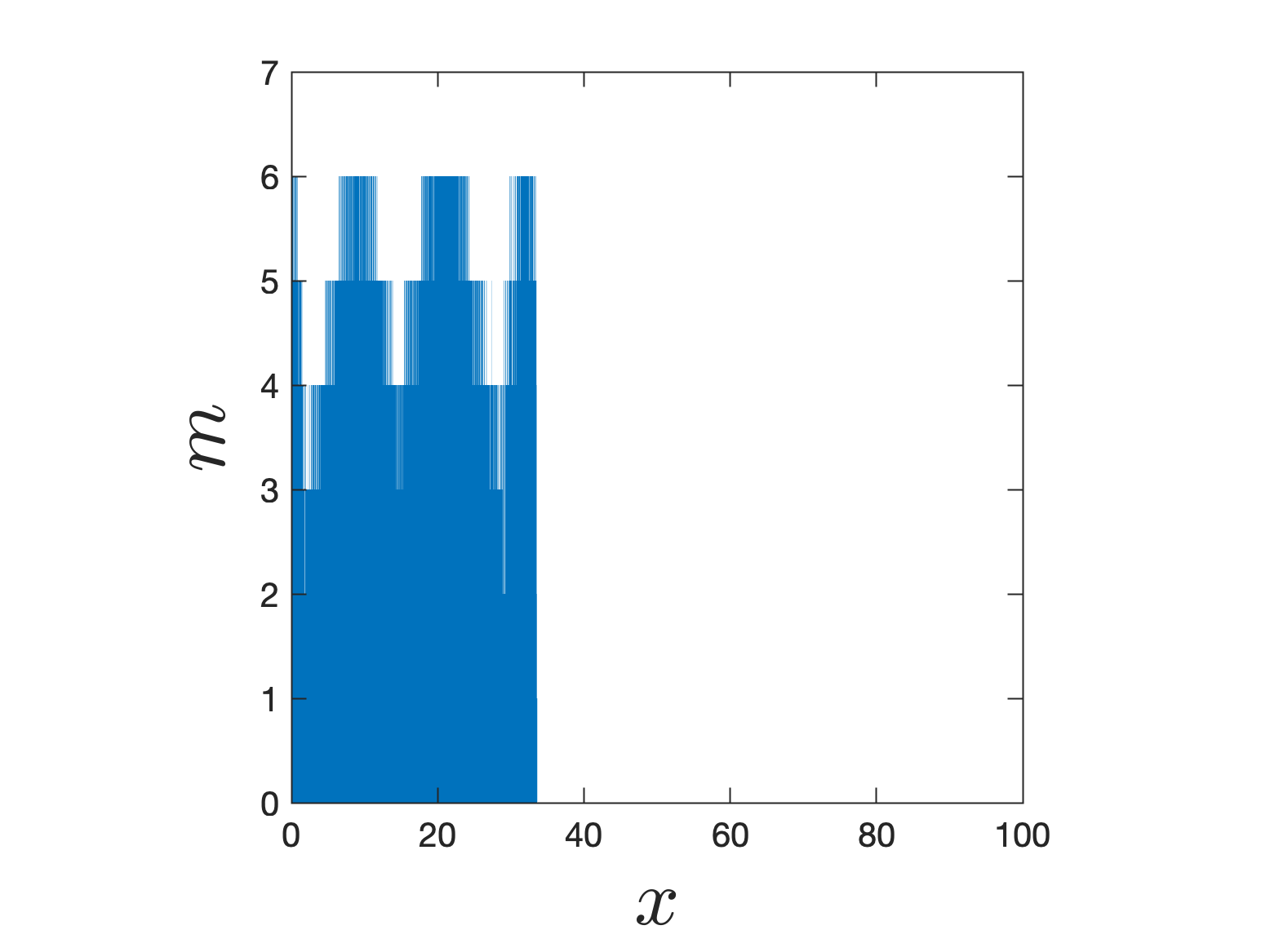} \hspace{-17pt}
		\includegraphics[width=2.5in,trim={0.0cm 0.0cm 1.75cm 0cm},clip]{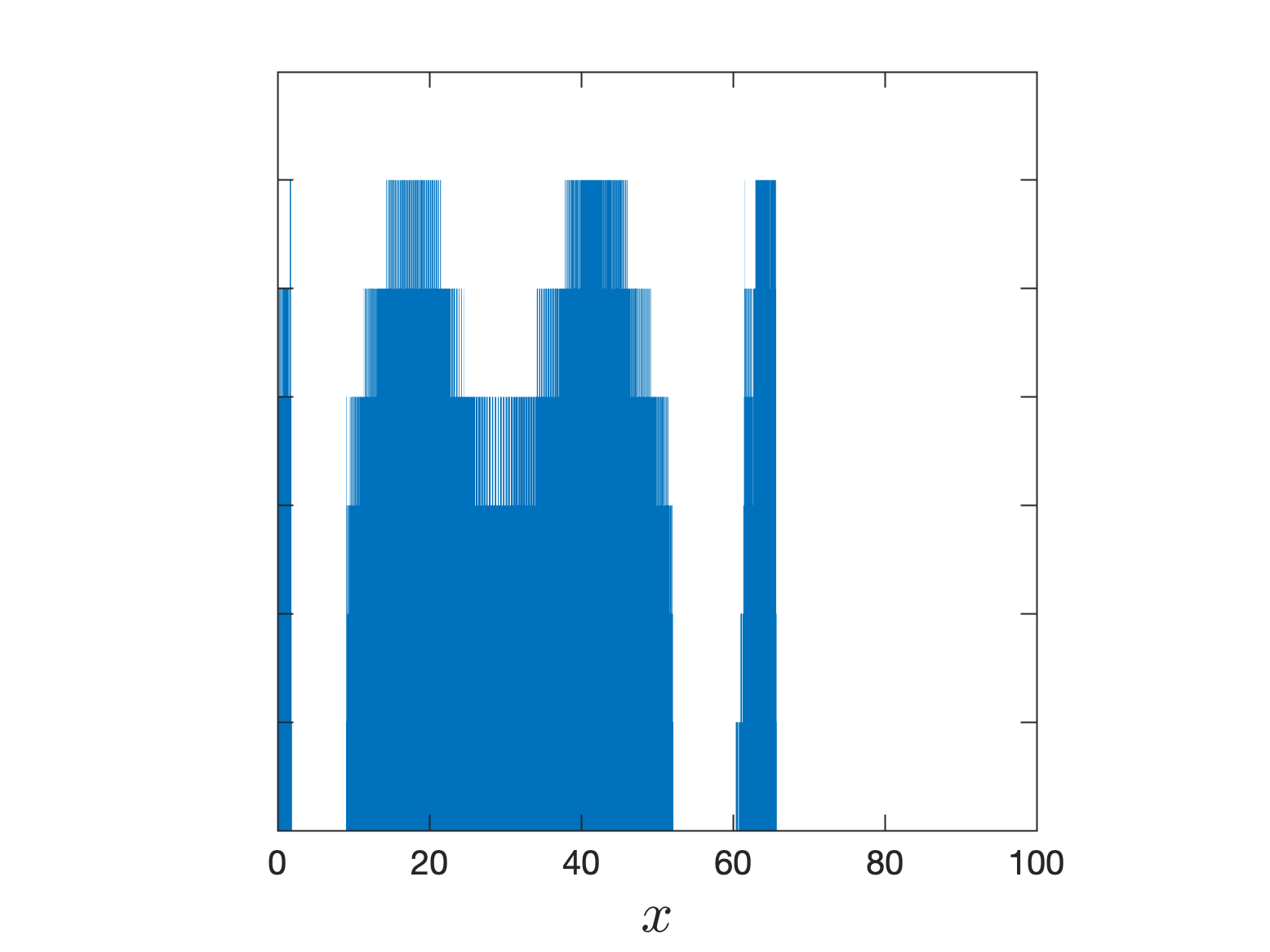}\hspace{-17pt}
		\includegraphics[width=2.5in,trim={0.0cm 0.0cm 1.75cm 0cm},clip]{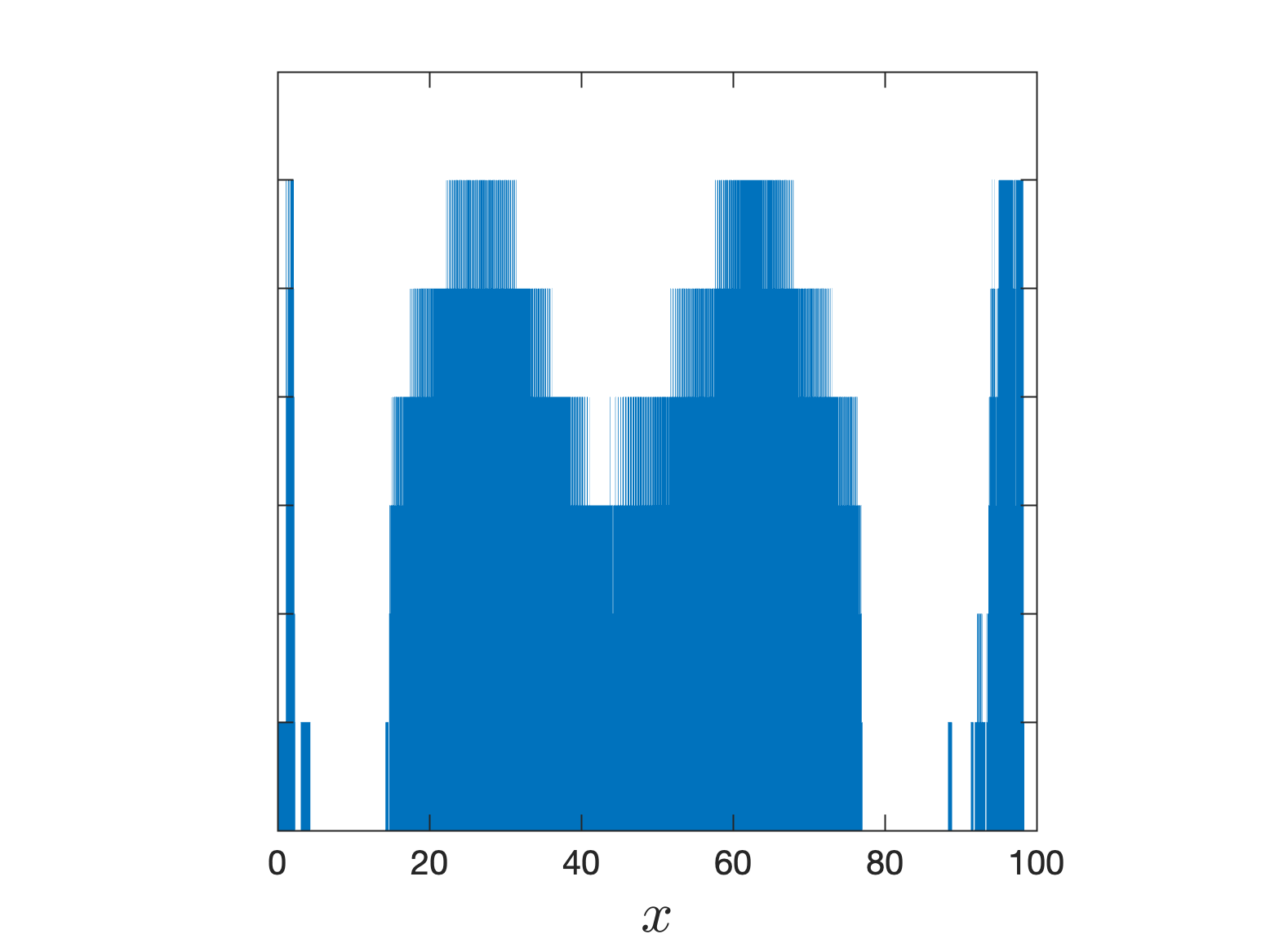}
	\end{adjustbox} 
       \caption{The electric field and the distribution of $m$ as a function of $x$ for different times using $A=1$. 
       		The first, second and third columns represent respectively the times 50, 100 and 150.
            For clarity, 
            we only show the solutions for $x\geq0$ since they are symmetric about the $y$-axis.}
       \label{fig:gaussian_pulse_A_1}
\end{figure}
%***************************************	
%***************************************
\begin{figure}   
	\centering
	\begin{adjustbox}{max width=1.0\textwidth,center}
		\includegraphics[width=2.5in,trim={0.0cm 0cm 1.75cm 0cm},clip]{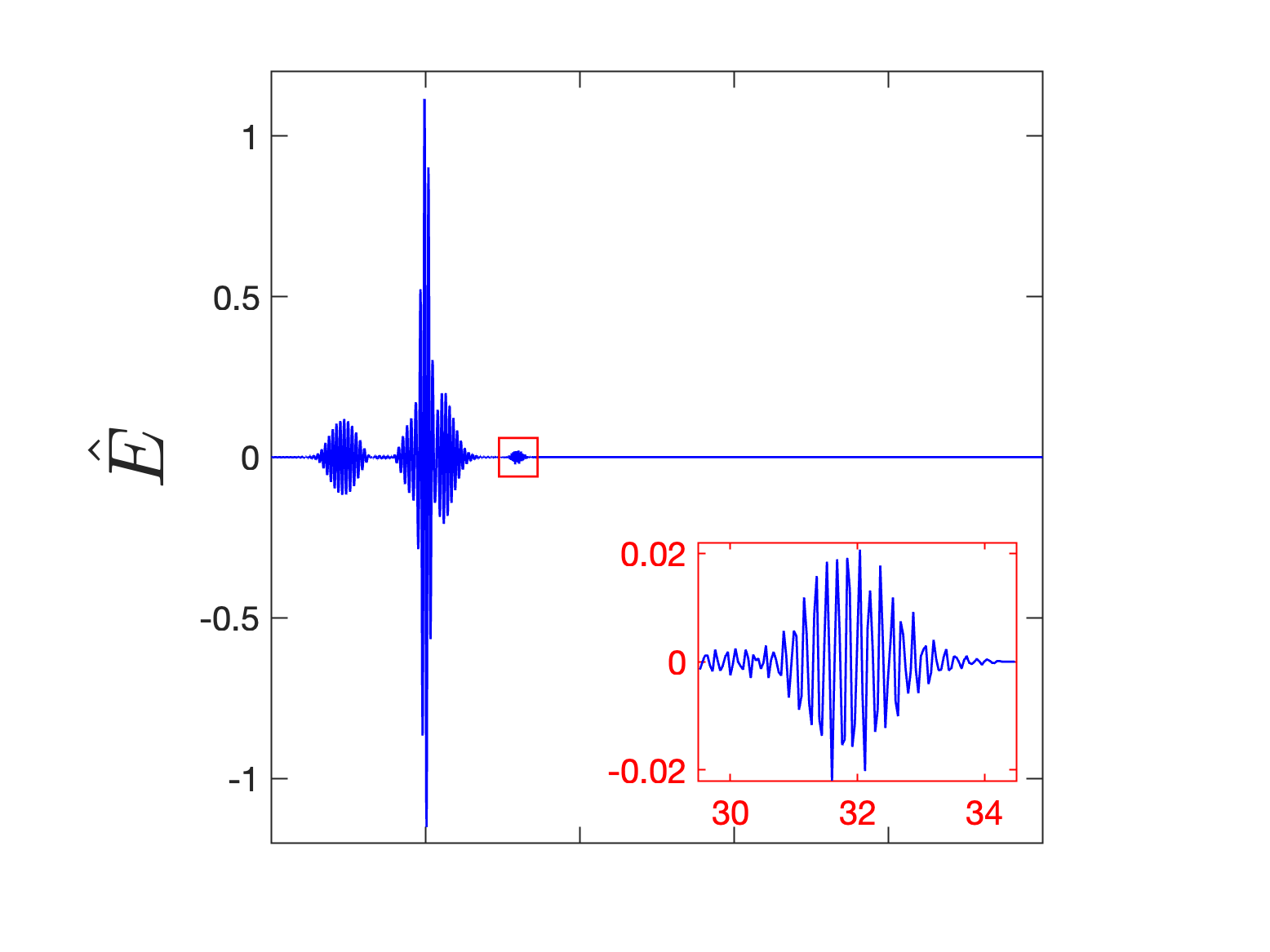} \hspace{-17pt}
		\includegraphics[width=2.5in,trim={0.0cm 0.0cm 1.75cm 0cm},clip]{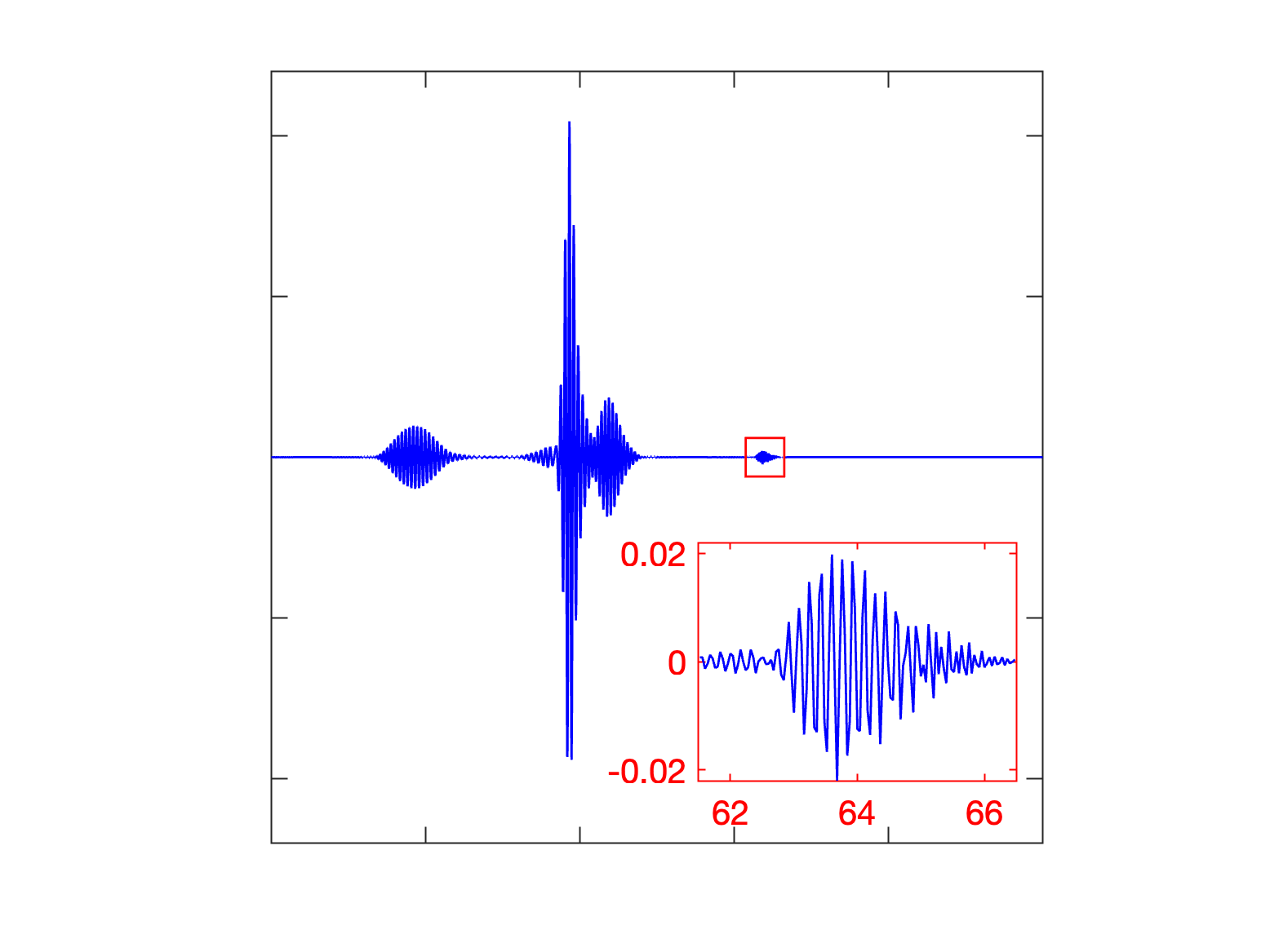}\hspace{-17pt}
		\includegraphics[width=2.5in,trim={0.0cm 0.0cm 1.75cm 0cm},clip]{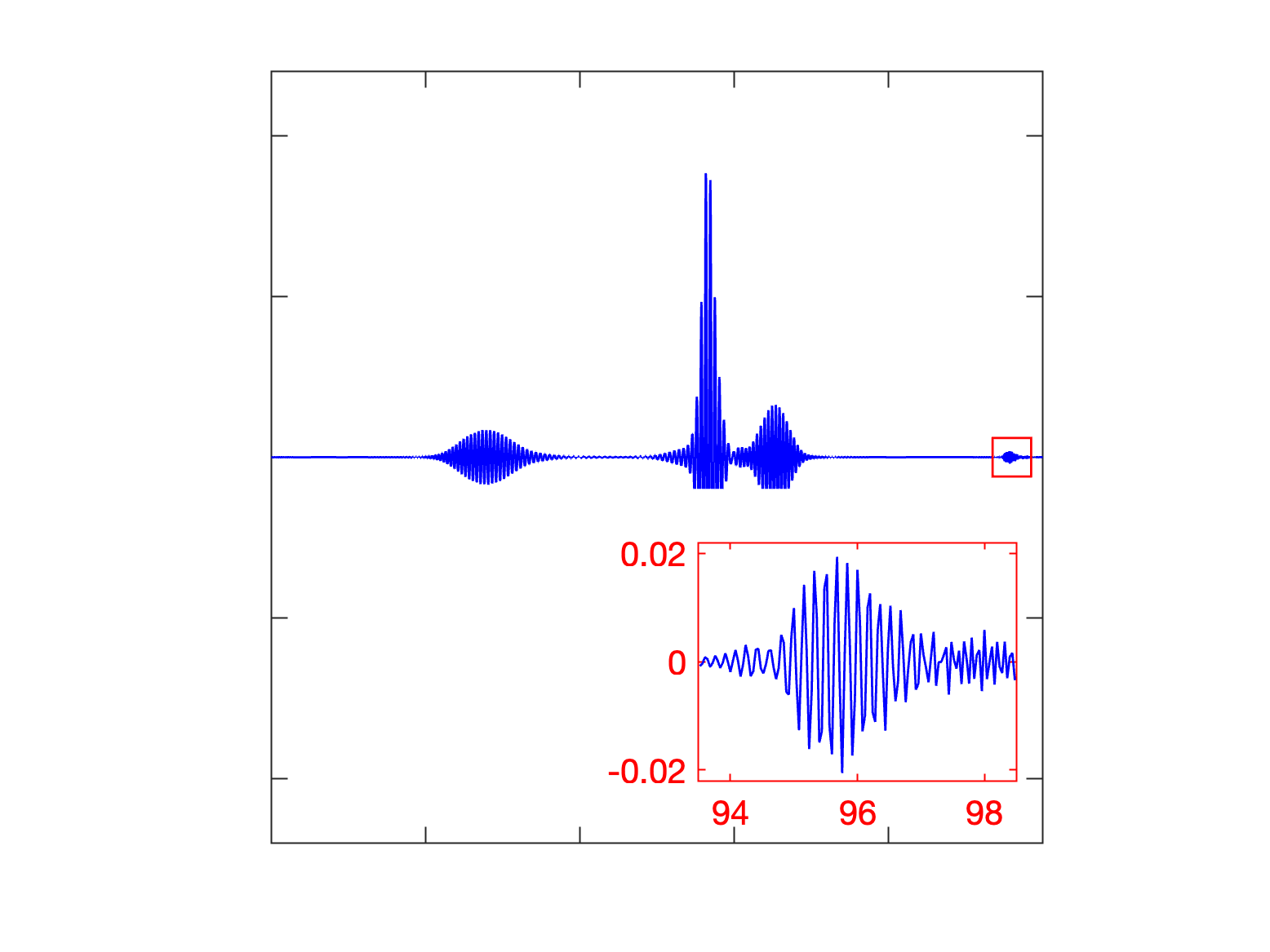}
	\end{adjustbox} 
	\begin{adjustbox}{max width=1.0\textwidth,center}
		\includegraphics[width=2.5in,trim={0.0cm 0cm 1.75cm 0cm},clip]{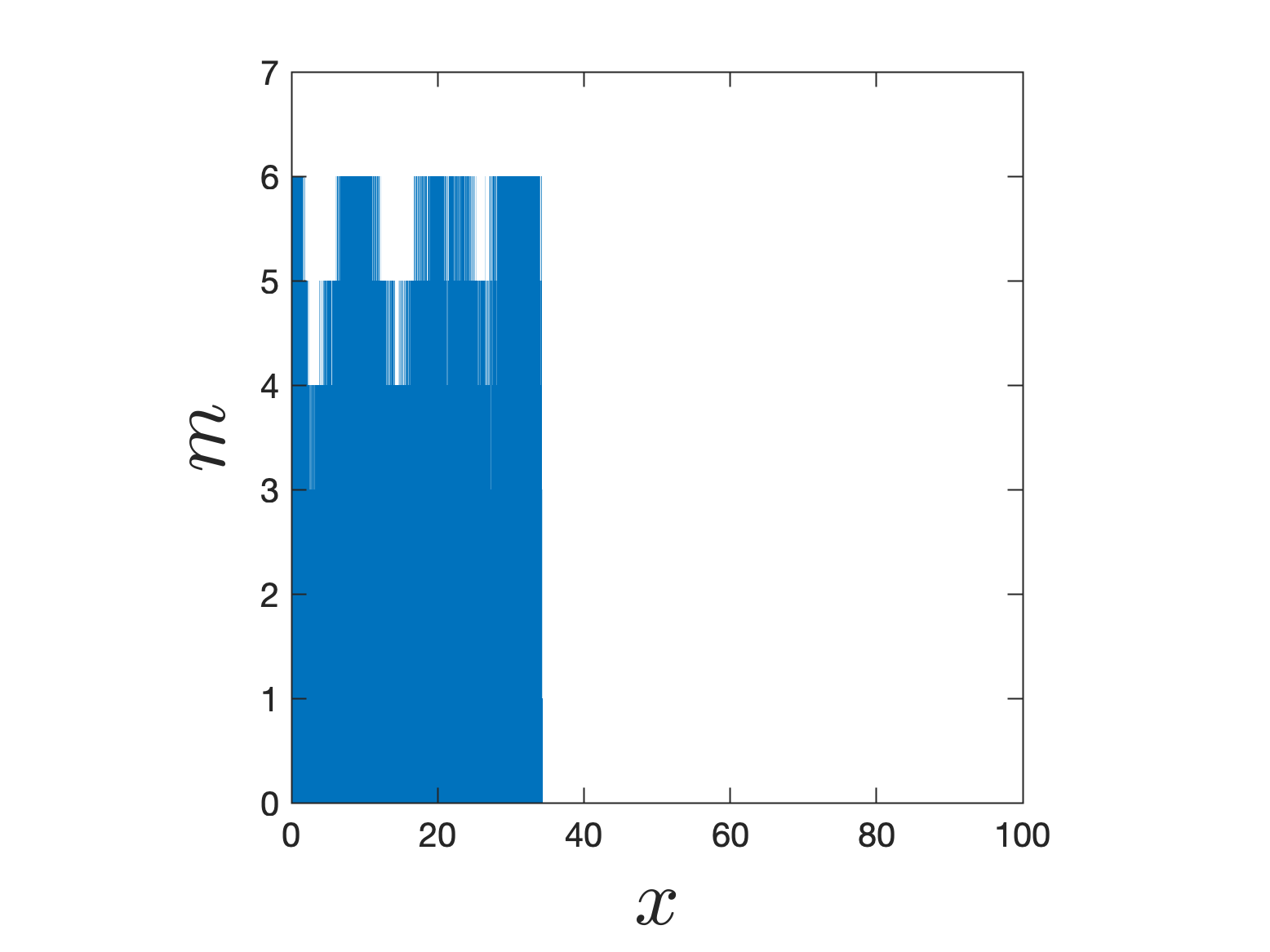} \hspace{-17pt}
		\includegraphics[width=2.5in,trim={0.0cm 0.0cm 1.75cm 0cm},clip]{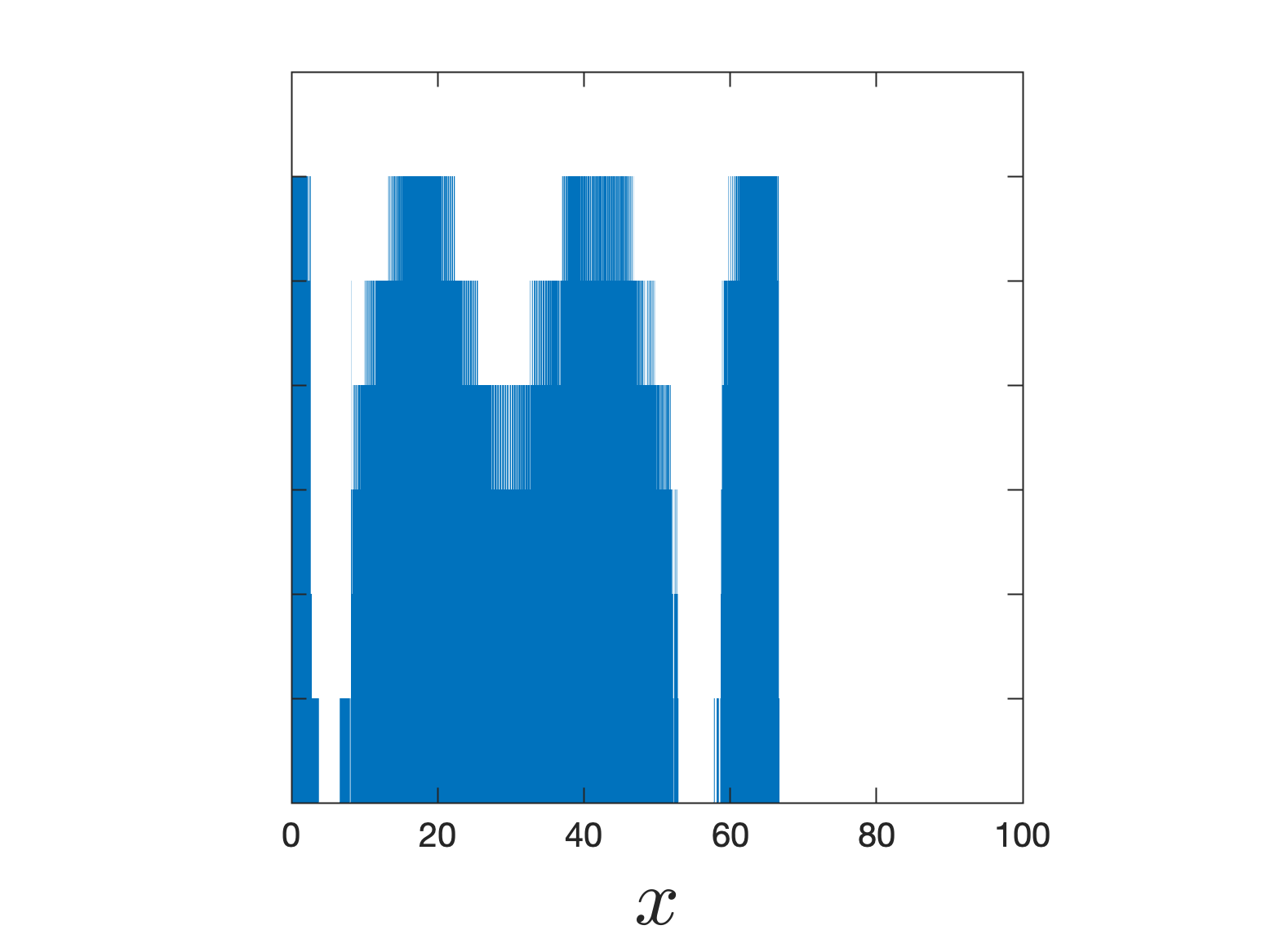}\hspace{-17pt}
		\includegraphics[width=2.5in,trim={0.0cm 0.0cm 1.75cm 0cm},clip]{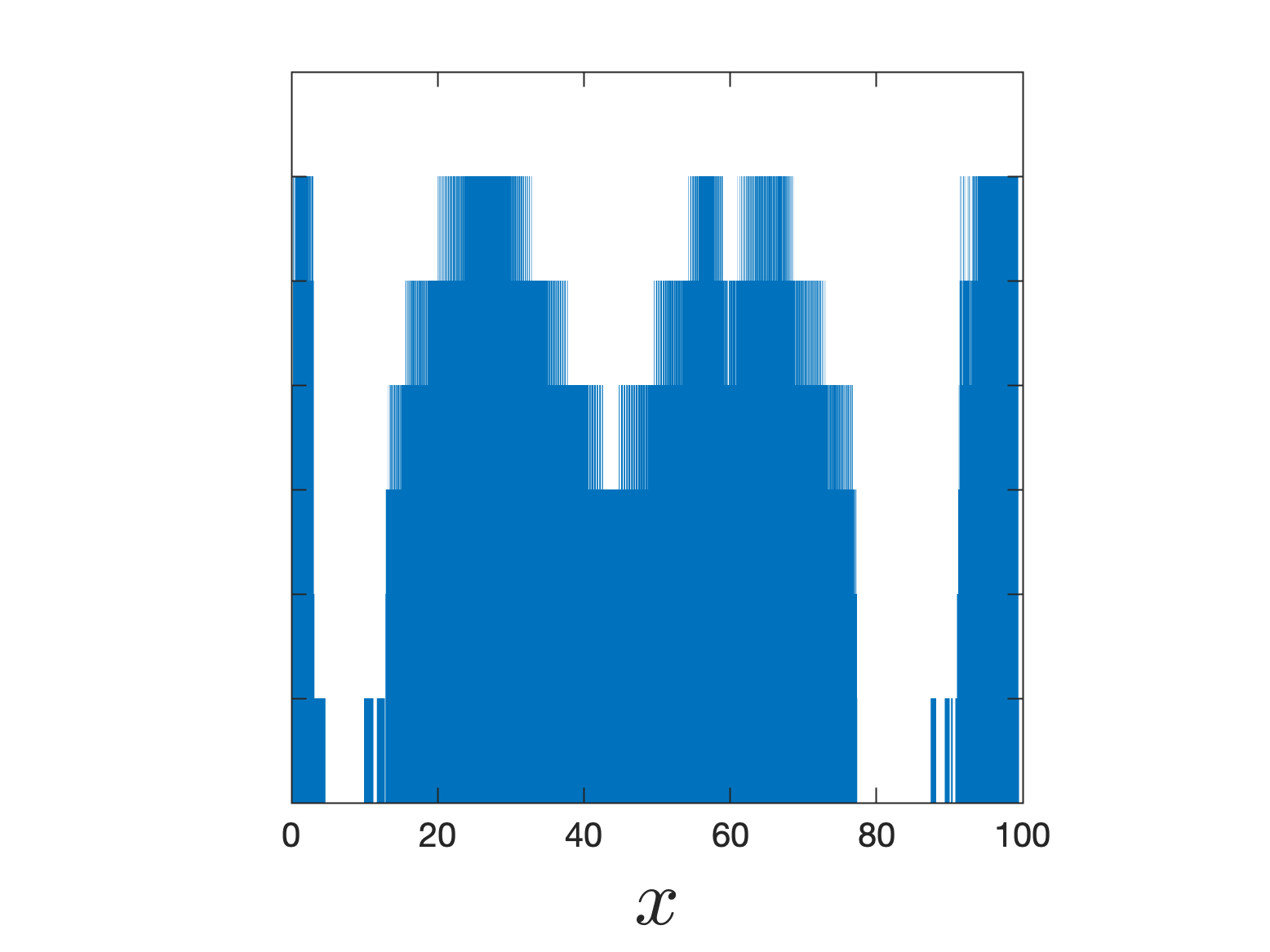}
	\end{adjustbox} 
       \caption{The electric field and the distribution of $m$ as a function of $x$ for different times using $A=2$. 
       		The first, second and third columns represent respectively the times 50, 100 and 150.
            For clarity, 
            we only show the solutions for $x\geq0$ since they are symmetric about the $y$-axis.}
       \label{fig:gaussian_pulse_A_2}
\end{figure}
%***************************************	
%***************************************
\begin{figure}   
	\centering
	\begin{adjustbox}{max width=1.0\textwidth,center}
		\includegraphics[width=2.5in,trim={0.0cm 0cm 1.75cm 0cm},clip]{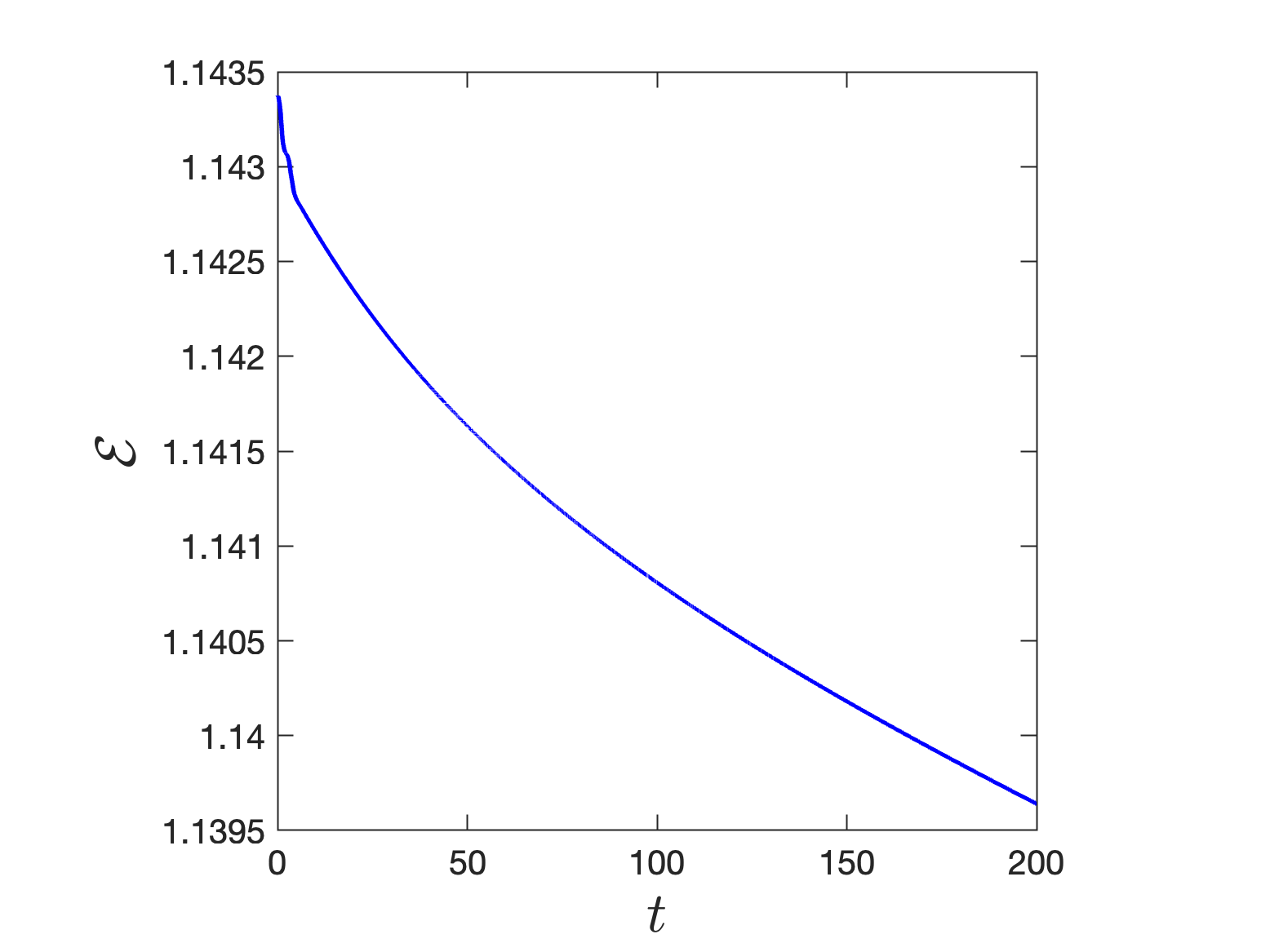} \hspace{-20pt}
		\includegraphics[width=2.5in,trim={0.0cm 0.0cm 1.75cm 0cm},clip]{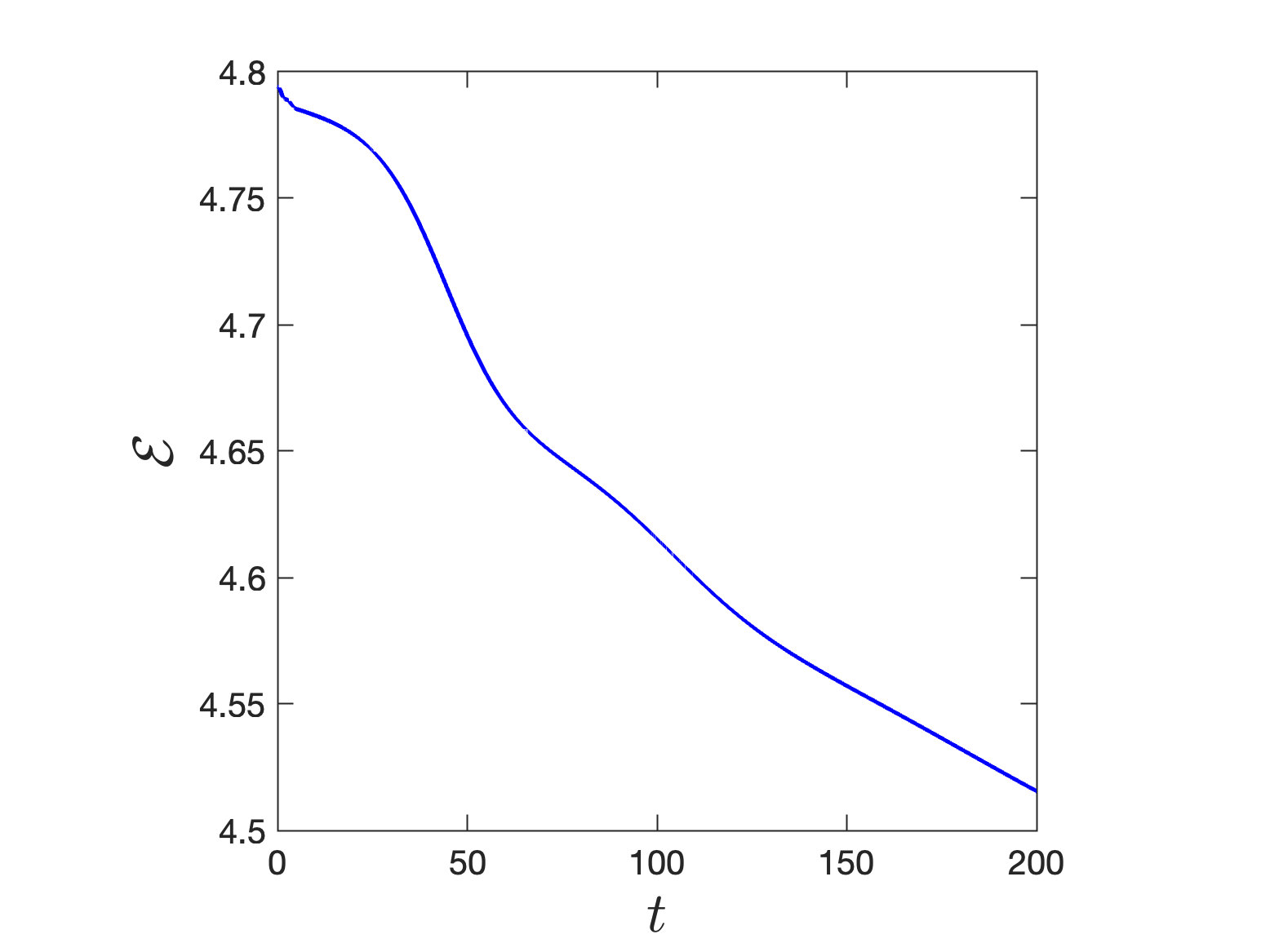}
	\end{adjustbox} 
       \caption{Evolution of the energy $\mathcal{E}$ in time. The left and right plots are for respectively $A=1$ and $A=2$.}
       \label{fig:energy_1d}
\end{figure}
%***************************************

As mentioned before, 
	the energy $\mathcal{E}$ is conserved when the damping terms are neglected. 
We therefore repeat the last numerical example but with $\gamma=\gamma_v=0$.
The maximal relative error on the energy $|\mathcal{E}(t)-\mathcal{E}(0)|/\mathcal{E}(0)$ over the time 
	interval $[0,200]$ is around $2\times10^{-8}$ for $M=1$ and $5\times10^{-8}$ for $M=2$.  
Note that the energy is underestimated as expected from a dissipative method. 
	
\subsection{Examples in Two Space Dimensions}

We use the manufactured solution 
\begin{equation}
	\begin{aligned}
		\hat{E}_x = \frac{\sin(wx)\cos(wy)\sin(\sqrt{2}wt)}{\sqrt{2}}, \quad \hat{E}_y = -\frac{\cos(wx)\sin(wy)\sin(\sqrt{2}wt)}{\sqrt{2}}, & \\
		\hat{H}_z = \sin(wx)\sin(wy)\cos(\sqrt{2}wt), \quad \hat{P}_x = -a(\hat{E}_x^2+\hat{E}_y^2)\hat{E}_x,  \quad \hat{J}_x = \partial_t \hat{P}_x,& \\
		\hat{P}_y = -a(\hat{E}_x^2+\hat{E}_y^2)\hat{E}_y, \quad \hat{J}_y = \partial_t \hat{P}_y, \quad \hat{Q} = \hat{E}_x^2+\hat{E}_y^2,  \quad \hat{S} = \partial_t \hat{Q}.& 
	\end{aligned}
\end{equation}
Here $w= 10\pi$ and there are source terms in the ODEs for $\hat{J}_x$, 
	$\hat{J}_y$ and $\hat{S}$ in \eqref{eq:ODE_system_2D}. 
All the physical parameters are one, 
	except for $a = 1/3$, 
	$\theta = 1/2$,
	$\gamma = 1/20$ and $\gamma_v=1/20$. 
The domain is $\Omega = [0,1]\times[0,1]$ and the time interval is $I=[0,10]$.
We set $h=\Delta x = \Delta y$,
	$\Delta t = h/2$ and $N_{\tau} = N_\tau^*$, 
	and consider $1\leq m \leq 4$.
The left plot of Figure~\ref{fig:conv_2d} illustrates that we obtain a $2m+1$ rates of convergence,
    as expected.

As a second numerical example, 
	we consider the manufactured solution 
\begin{equation}
	\begin{aligned}
		\hat{E}_x =&\,\,  \hat{H}_z = (\eta_x+\eta_y) e^{-\frac{\eta_x^2+\eta_y^2}{\sigma^2}}=-\hat{E}_y, \\
		\hat{P}_x =&\,\, - a (\hat{E}_x^2 + \hat{E}_y^2) E_x, \quad \hat{P}_y = - a (\hat{E}_x^2 + \hat{E}_y^2) \hat{E}_y, \quad \hat{Q} = \hat{E}_x^2 + \hat{E}_y^2, \\
		\hat{J}_x =&\,\, -\hat{J}_y = 12 a (\eta_x+\eta_y)^2 e^{-\frac{3(\eta_x^2 + \eta_y^2)}{\sigma^2}} \bigg(\frac{(\eta_x+\eta_y)^2}{\sigma^2} - 1\bigg), \\
		\hat{S} =&\,\, 8 (\eta_x+\eta_y)  e^{-\frac{2(\eta_x^2 + \eta_y^2)}{\sigma^2}} \bigg(1 - \frac{(\eta_x+\eta_y)^2}{\sigma^2}\bigg).
	\end{aligned}
\end{equation}
Here $\sigma = 0.1$, 
	$\eta_x = x+t$ and $\eta_y = y+t$.
The wave packets travel along the direction $(-1,-1)$ at a constant speed.
The domain is $\Omega = [-2,2]\times [-2,2]$ and $I = [0,10]$. 
The physical parameters remain the same as before. 
Note that source terms are required in the ODEs to evolve $\hat{J}_x$, 
	$\hat{J}_y$ and $\hat{S}$ in \eqref{eq:ODE_system_2D}, 
	and the two first equations of Maxwell's equations \eqref{eq:PDE_system_2D}.
We set $\Delta x = \Delta y = h$, 
	$\Delta t = h/2$ and $N_\tau = N_\tau^*$.
The right plot of Figure~\ref{fig:conv_2d} shows a $2m+1$ rates of convergence for $1\leq m \leq 4$.
	
%***************************************
\begin{figure}   
	\centering
	\begin{adjustbox}{max width=1.0\textwidth,center}
		\includegraphics[width=2.5in,trim={0.0cm 0cm 1.75cm 0cm},clip]{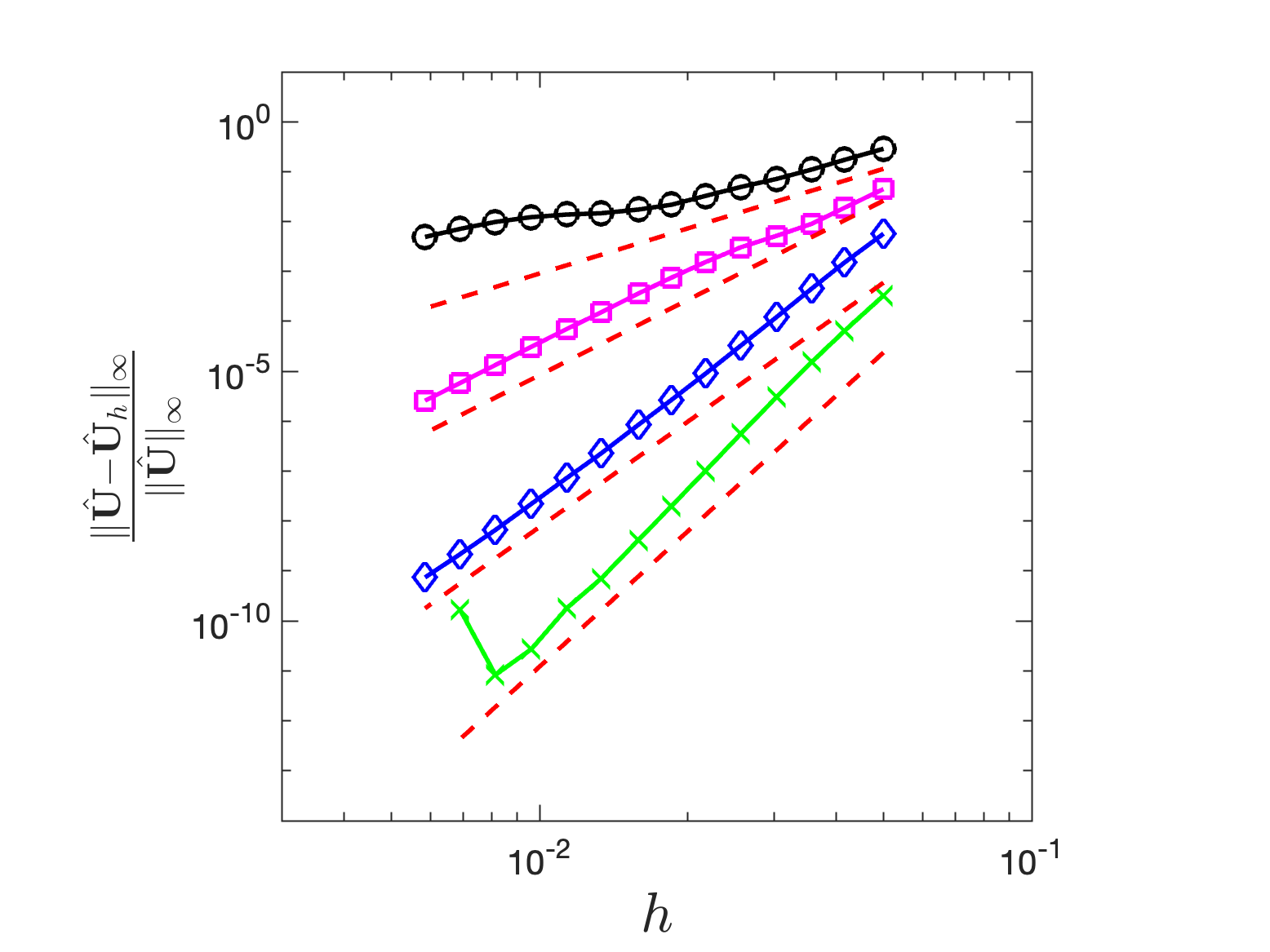} \hspace{-20pt}
		\includegraphics[width=2.5in,trim={0.0cm 0.0cm 1.75cm 0cm},clip]{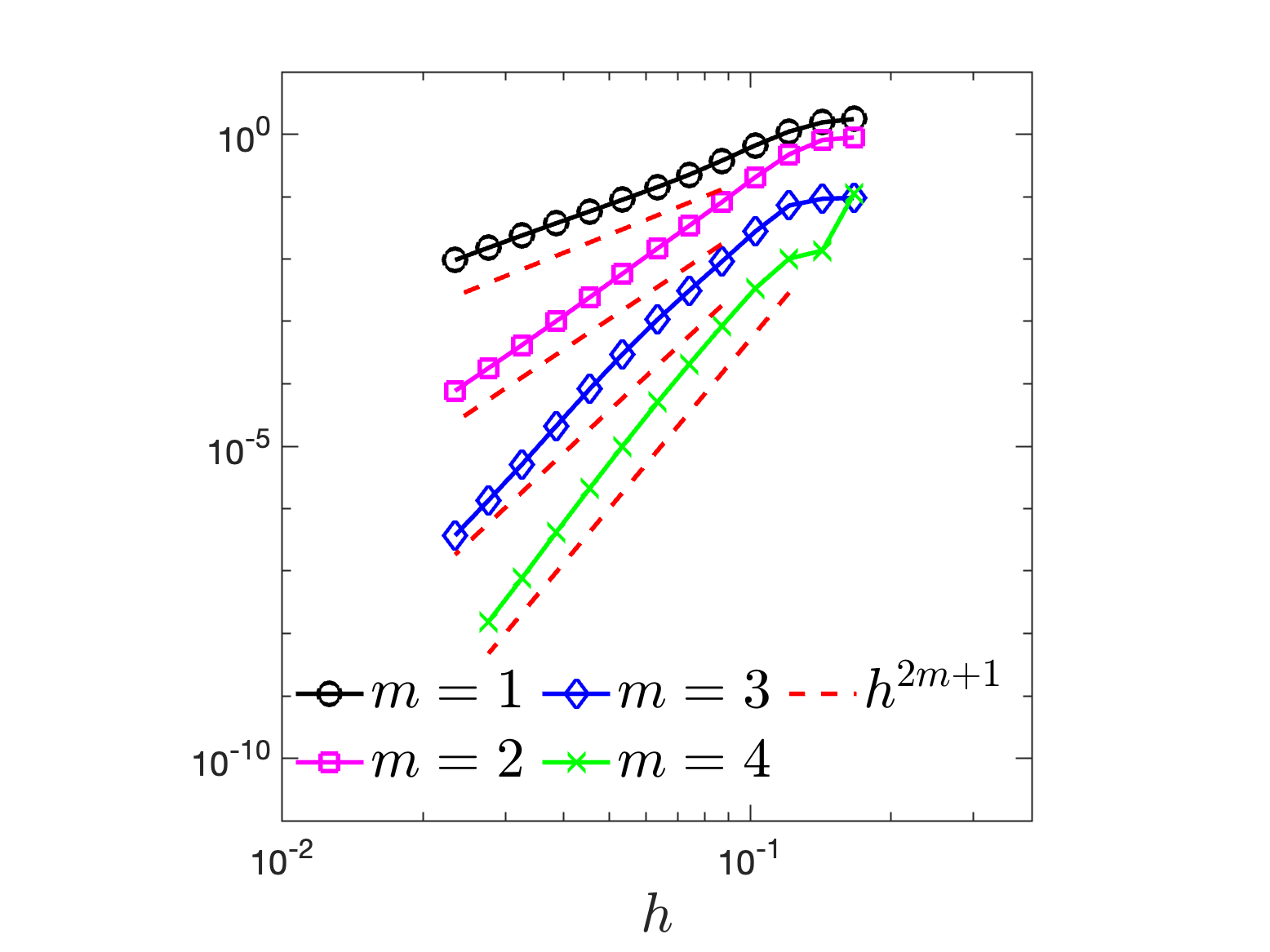}
	\end{adjustbox} 
       \caption{Convergence plots for manufactured solutions problems for $m=1-4$ in two space dimensions. 
       		Here $\hat{\mathbf{U}}$ is a vector containing all variables.}
       \label{fig:conv_2d}
\end{figure}
%***************************************
	
We now repeat the second numerical example but with the $p$-adaptive method using $h=1/25$, 
	$0\leq m \leq 6$ and $N_{\tau_{\max}} = 3$.
The left plot of Figure~\ref{fig:err_vs_nb_dof_2D} illustrates the error in maximum norm as a function of the tolerance $\epsilon_{p_{tol}}$.
The error diminishes as the tolerance decreases until it reaches the maximum accuracy attainable with this mesh when $m=6$ everywhere. 
However, 
	the error remains larger than the considered tolerance, 
	as in the 1-D case. 
The relevance of the $p$-adaptive method is illustrated in the right plot of Figure~\ref{fig:err_vs_nb_dof_2D}, 
	where this method requires a fewer number of DOFs for a given accuracy than the Hermite method with $m$ fixed.
Figure~\ref{fig:values_m_2D} illustrates the values of $m$ as a function of $(x,y)$ at the final time, 
	confirming that the $p$-adaptive method increases the value of $m$ where the pulse is localized. 

%***************************************
\begin{figure}   
	\centering
	\begin{adjustbox}{max width=1.0\textwidth,center}
		\includegraphics[width=2.5in,trim={0.0cm 0cm 1.75cm 0cm},clip]{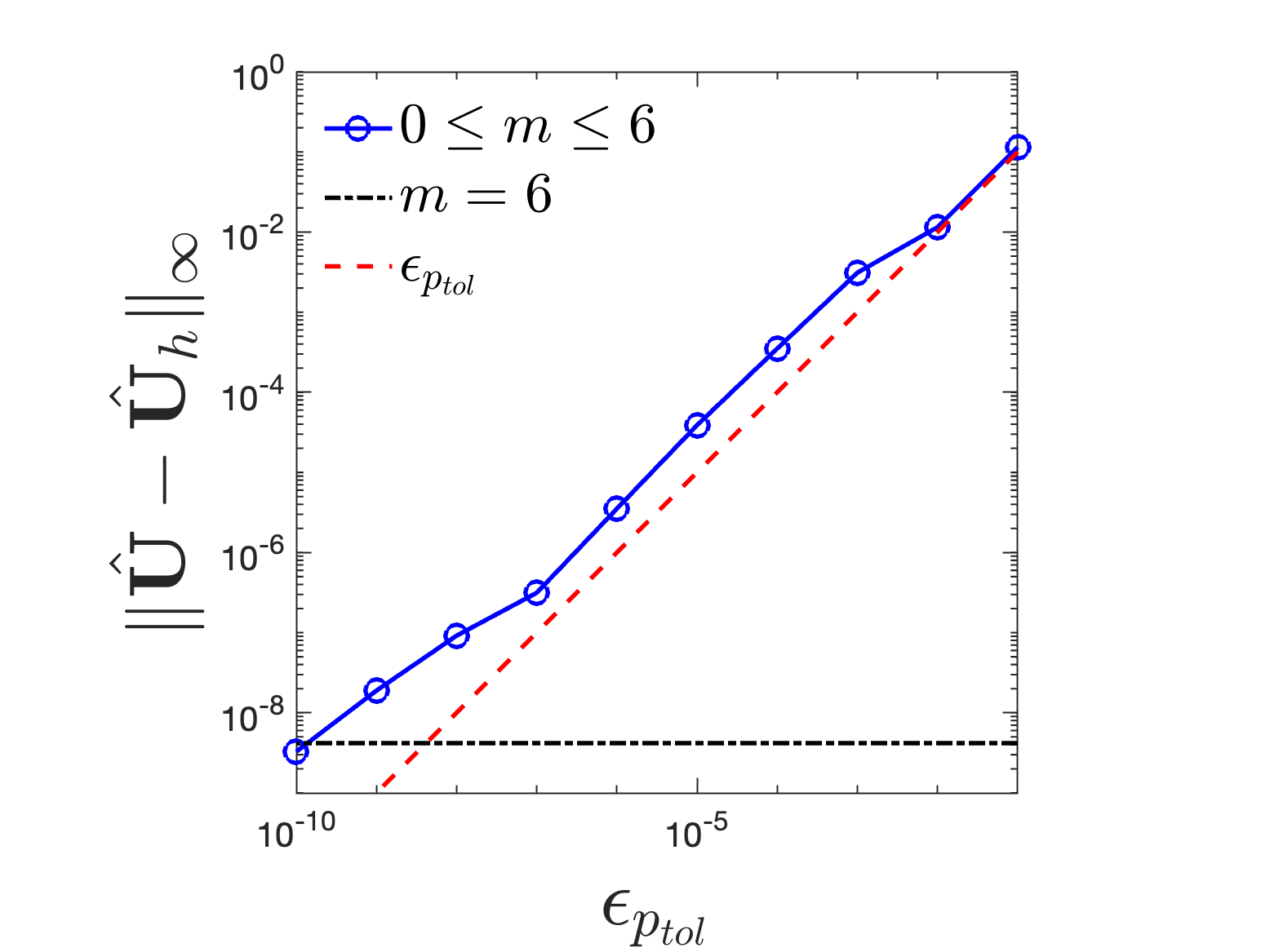} \hspace{-20pt}
		\includegraphics[width=2.5in,trim={0.0cm 0.0cm 1.75cm 0cm},clip]{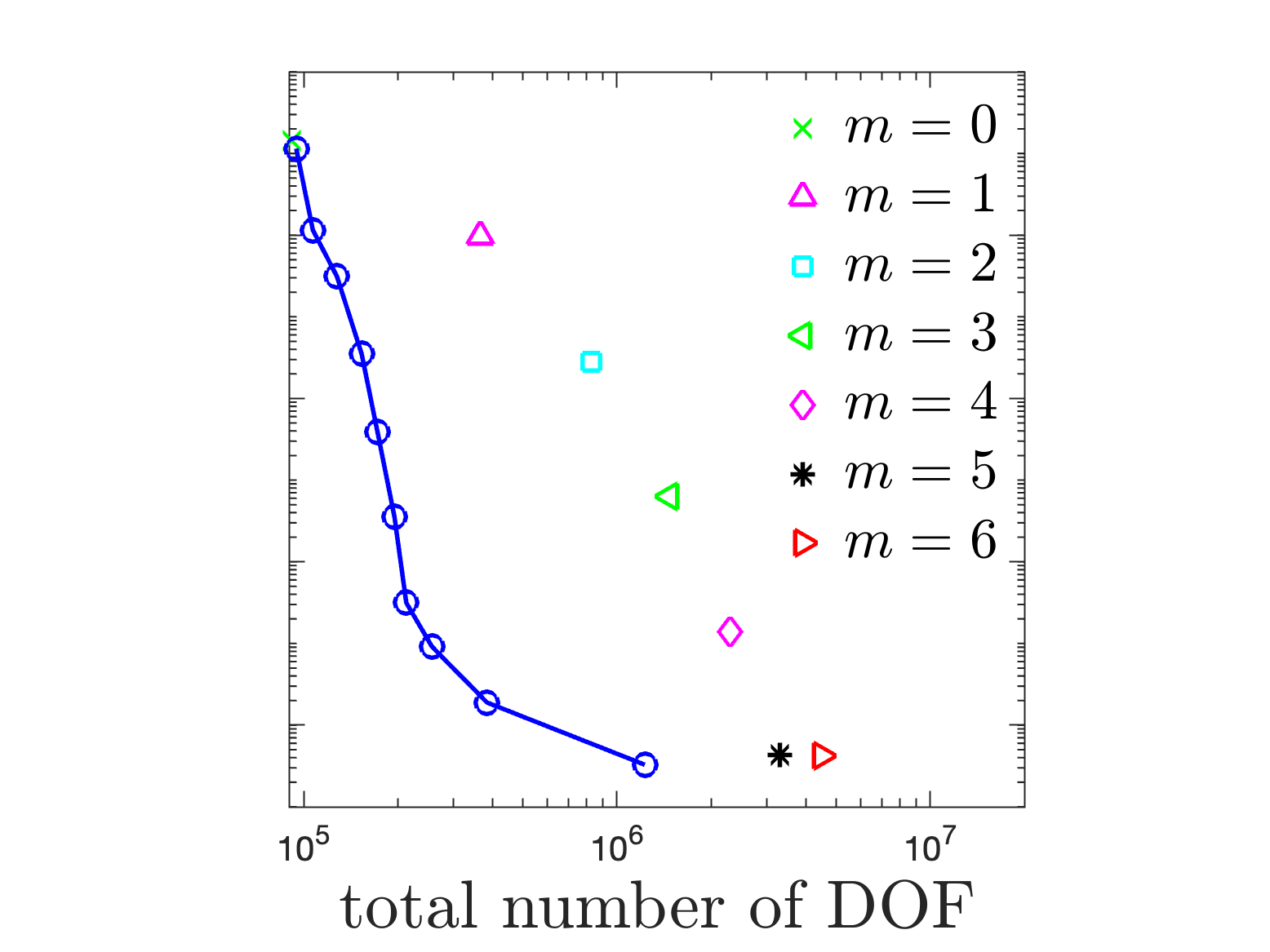}
	\end{adjustbox} 
       \caption{The error in maximum norm as a function of the tolerance (left plot) and the number of degrees of freedom (right plot).}
       \label{fig:err_vs_nb_dof_2D}
\end{figure}
%***************************************

%**********************************************
 \begin{figure} 
 \centering
	\begin{adjustbox}{max width=1.0\textwidth,center}
	 \centering
		\includegraphics[width=2.5in,trim={1.5cm 0cm 1.75cm 0cm},clip]{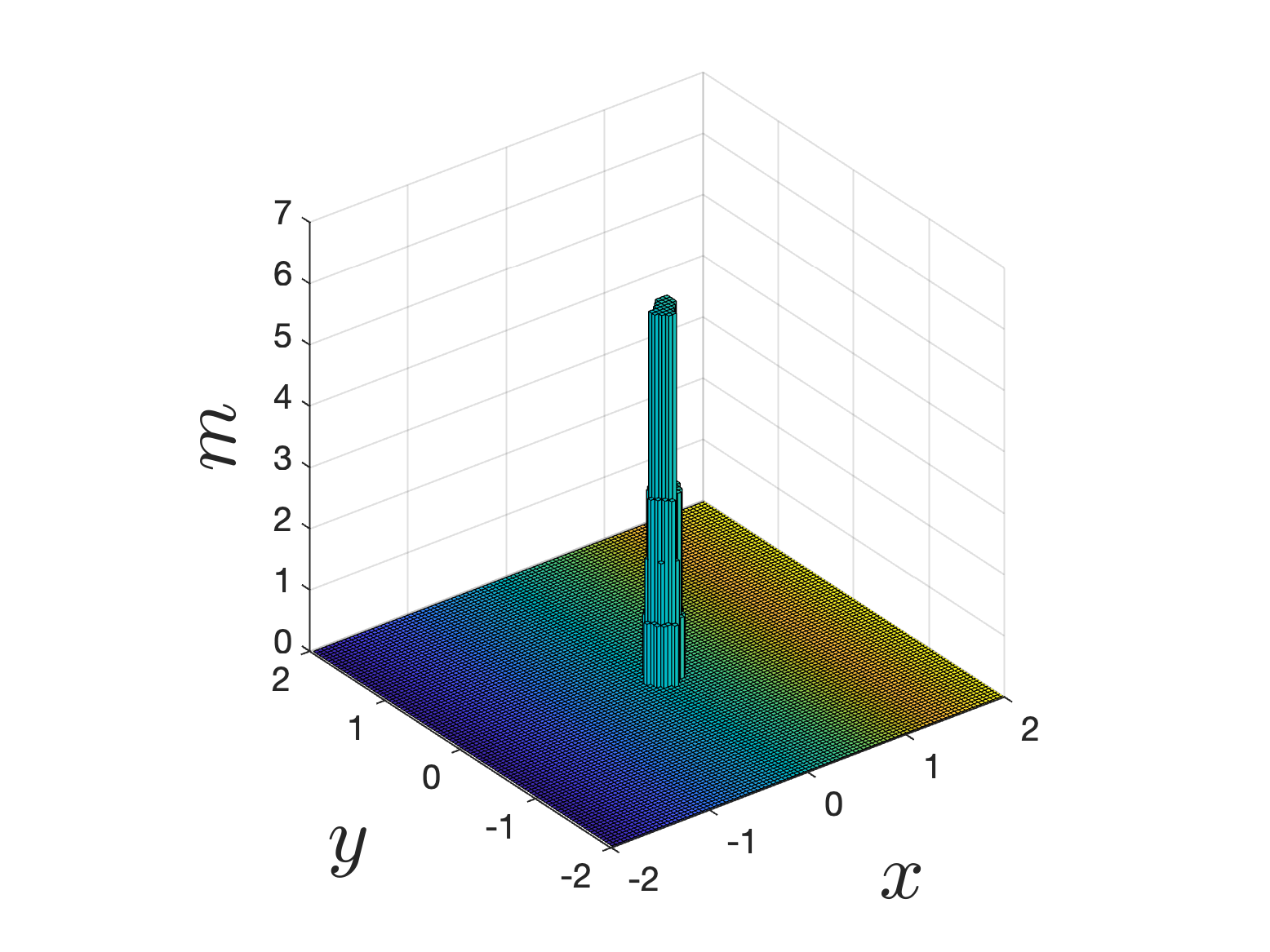}\hspace{-20pt}
	   	\includegraphics[width=2.5in,trim={1.5cm 0cm 1.75cm 0cm},clip]{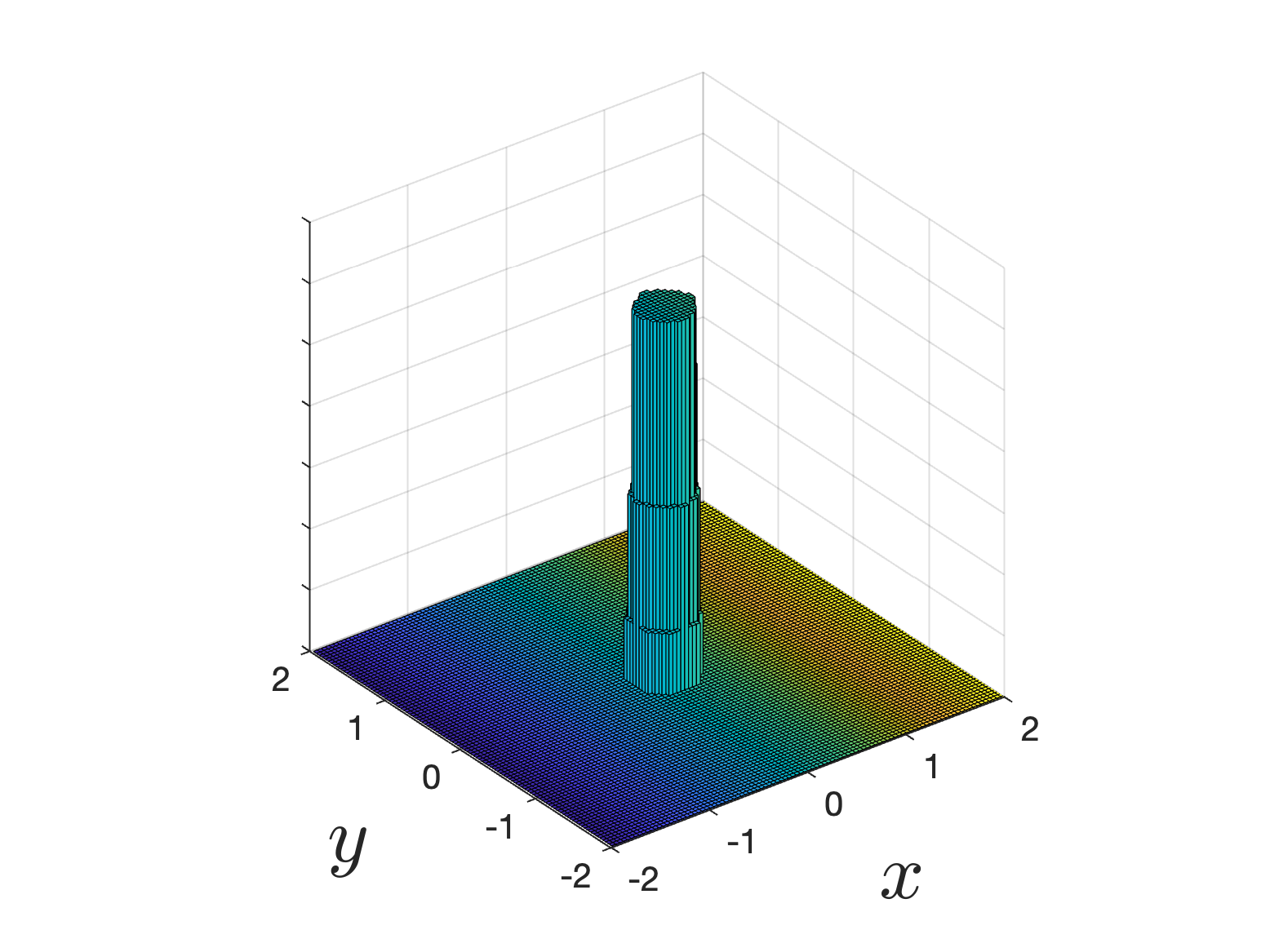}\hspace{-20pt}
	   	\includegraphics[width=2.5in,trim={1.5cm 0cm 1.75cm 0cm},clip]{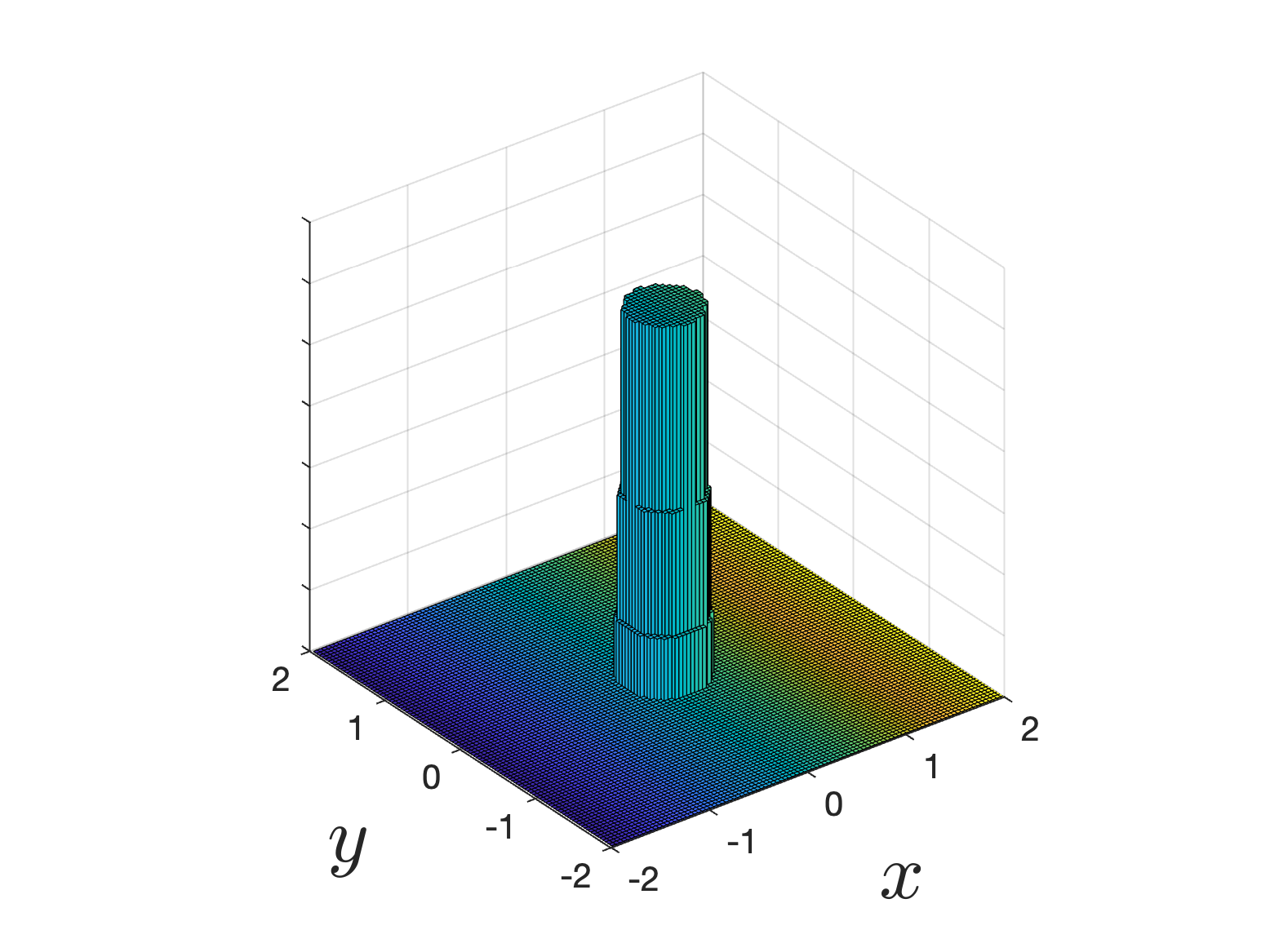} 
	\end{adjustbox}		
  \caption{Values of $m$ as a function of $x$ and $y$ for different values of tolerance $\epsilon_{p_{tol}}$ using $h = 1/25$. 
  		The left, 
			middle and right plots are for the tolerances $10^{-2}$ , 
			$10^{-4}$ and $10^{-6}$.}
   \label{fig:values_m_2D}
\end{figure}
%*************************************** 

As a last example, 
	we consider an airhole in a nonlinear glass.
The propagation of electromagnetic waves in air is modeled by Maxwell's equations \eqref{eq:Maxwell_3D} with the constitutive law $\hat{\mathbold{D}} = \epsilon \hat{\mathbold{E}}$.
At the interface between the air and the nonlinear dispersive medium, 
	we use a diffusive interface model.
For a circular airhole centered at $(x_c,y_c)$ of radius $r_\Gamma$, 
	we define 
\begin{equation} \label{eq:interface_model}
	\phi_\Gamma(r) = 
    \left\{ 
          \begin{aligned}
              0, &\quad \mbox{if} \quad r\leq r_\Gamma -\delta/2, \\[1pt] 
              1, &\quad \mbox{if} \quad r\geq r_\Gamma + \delta/2, \\[1pt] 
              1-3\tilde{r}^2 + 2\tilde{r}^3, &\quad \mbox{otherwise}. \\[1pt] 
          \end{aligned} \right. 
\end{equation}
Here $r=\sqrt{(x-x_c)^2+(y-y_c)^2}$, 
	$\tilde{r} = 1/2-(r-r_\Gamma)/\delta$
	and the transition function between the air ($\phi_\Gamma =0$) and the nonlinear medium ($\phi_\Gamma =1$) corresponds to a third degree Hermite interpolant 
	satisfying $\phi_\Gamma(r_\Gamma-\delta/2) = 0$,
	$\phi_\Gamma(r_\Gamma+\delta/2) = 1$ and $\phi_\Gamma'(r_\Gamma-\delta/2) = \phi_\Gamma'(r_\Gamma+\delta/2) = 0$. %$\tilde{r} =1-(r-(r_c-\delta/2))/\delta$.
The diffusive interface model leads to order reduction issues, 
	so we expect that the $p$-adaptive method uses the maximal value of $m$ in the vicinity of the interface. 
Future work will focus on the enforcement of interface conditions to circumvent this issue.

The domain is $\Omega = [-20,20]^2$ and the time interval is $I = [0,30]$.
Periodic boundary conditions are considered.
The geometry of the airhole is a circle centered at $(5\cos(\pi/4), 5\sin(\pi/4))$ with a radius $r_\Gamma =  0.5$.
The physical parameters in air are set to $\mu = 1$ and $\epsilon =1$ 
	while the parameters in the nonlinear medium are 
	$\mu=1$, 
	$\epsilon = 1$,
	$\epsilon_\infty = 2.25$, 
	$a=0.07$,
	$\omega_0 = 5.84$, 
	$\omega_p = \omega_0\sqrt{5.25-\epsilon_\infty}$,
	$\omega_v = 1.28$, 
	$\gamma = \gamma_v = 0$ and $\theta = 0.3$. 
The incorporation of the diffusive interface model in the method is done by weighting the parameters $a$, 
	$\omega_0$, 
	$\omega_p$ and $\omega_v$ by the function $\phi_\Gamma$, 
	and setting $\epsilon_\infty = 1 + (2.25-1)\phi_\Gamma$. 	

We consider a pulse as the initial condition for the magnetic field, 
	that is 
\begin{equation} \label{eq:cnd_initial_2d}
	\hat{H}_z(x,y,0) = \mbox{sech}(\alpha x)\mbox{sech}(\alpha y)\cos(\alpha\Omega_0 x)\cos(\alpha\Omega_0 y),
\end{equation}
	where $\alpha = 2.5$ and $\Omega_0 = 12.57$, 
	and all the other variables are set to zero. 

We set $h = \Delta x = \Delta y = 1/20$, 
	$\Delta t = h/2$ and $\delta  = 2h$.
We use the $p$-adaptive method with $0\leq m \leq 6$,
	$\epsilon_{p_{tol}} = 10^{-4}$ and $N_{\tau_{\max}} = 3$.
Figure~\ref{fig:pblm_interface} illustrates the evolution of $\|\hat{E}\|_2$, 
	$|\hat{Q}|$ and the distribution of $m$.
%***************************************
\begin{figure}   
	\centering
	\begin{adjustbox}{max width=1.0\textwidth,center}
		\includegraphics[width=2.5in,trim={2.5cm 0cm 1.75cm 0cm},clip]{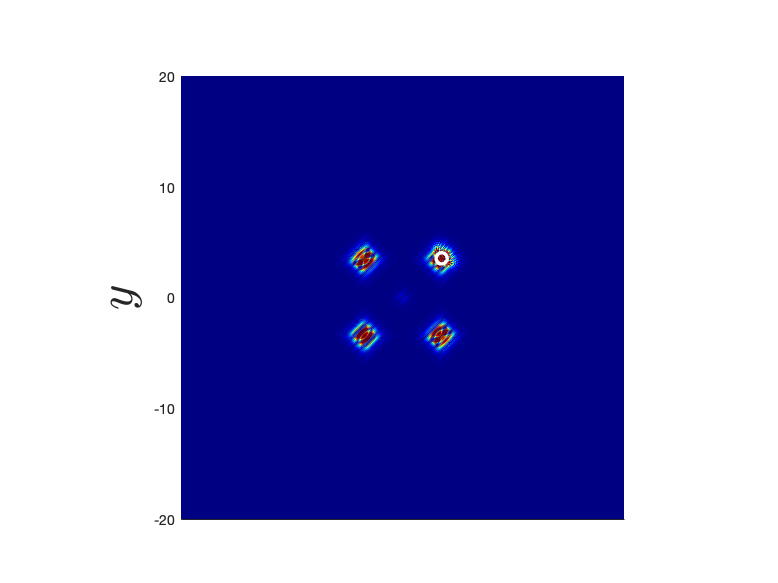}
		\includegraphics[width=2.5in,trim={2.5cm 0cm 1.75cm 0cm},clip]{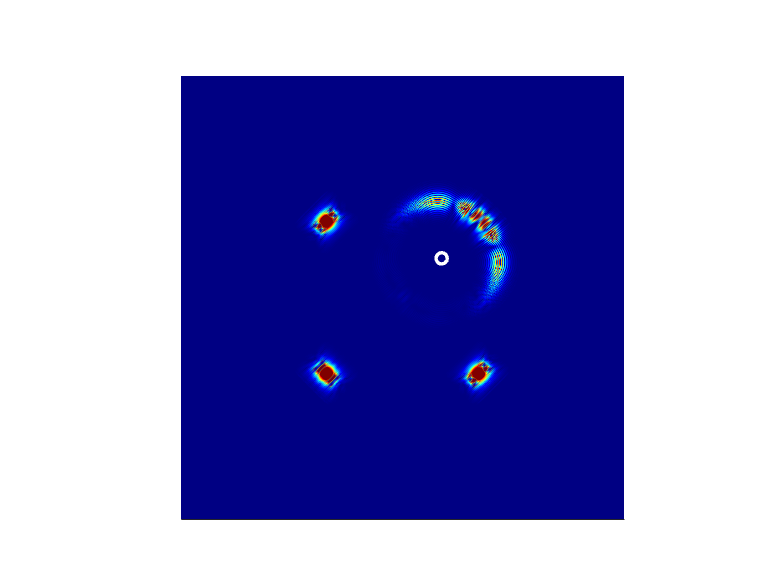}
		\includegraphics[width=2.5in,trim={2.5cm 0cm 1.75cm 0cm},clip]{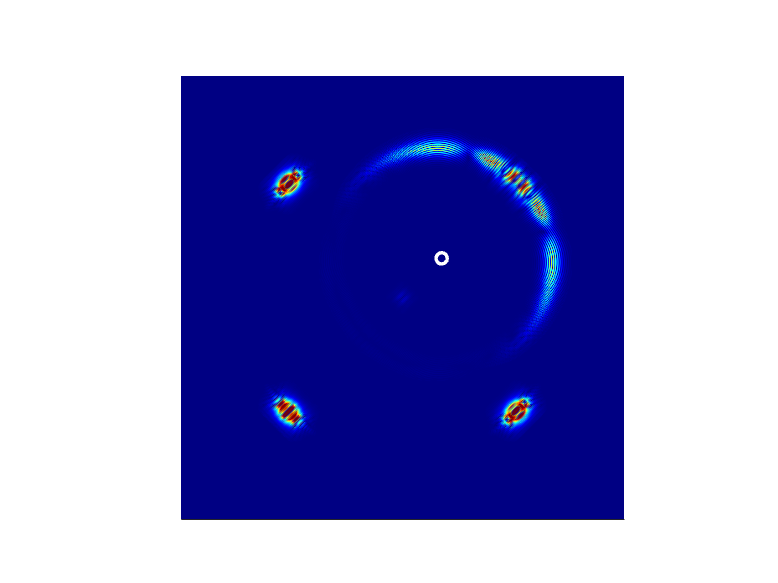}
	   	\includegraphics[width=2.5in,trim={2.5cm 0cm 1.75cm 0cm},clip]{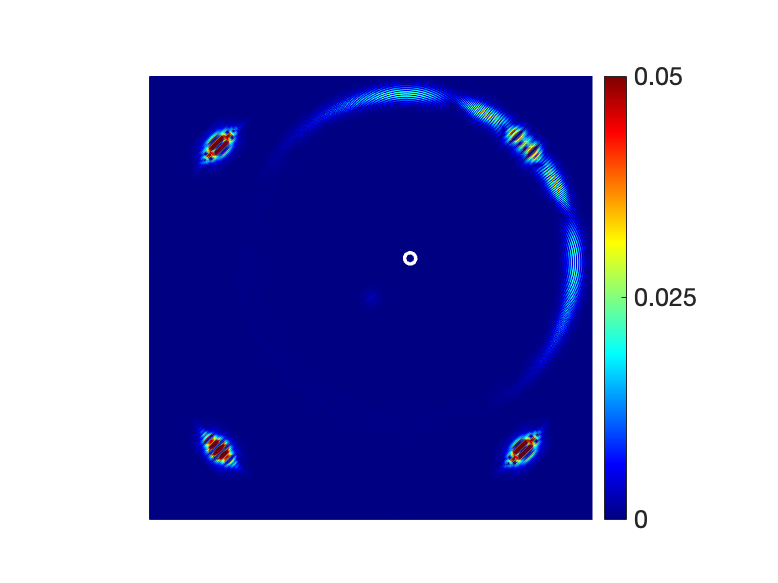}
	\end{adjustbox} 
	\begin{adjustbox}{max width=1.0\textwidth,center}
		\includegraphics[width=2.5in,trim={2.5cm 0cm 1.75cm 0cm},clip]{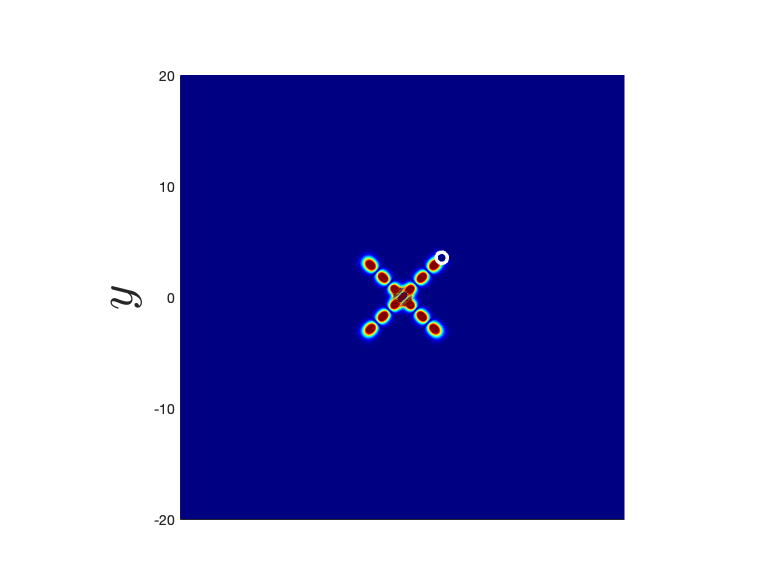}
		\includegraphics[width=2.5in,trim={2.5cm 0cm 1.75cm 0cm},clip]{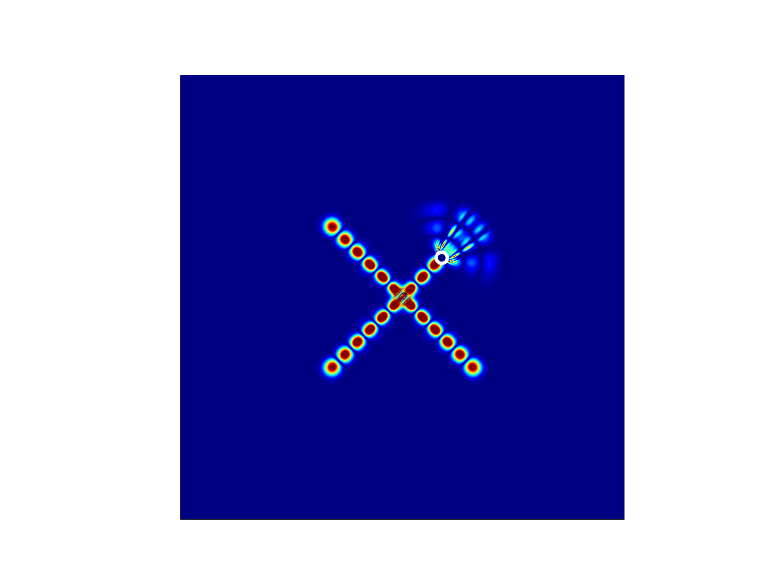}
		\includegraphics[width=2.5in,trim={2.5cm 0cm 1.75cm 0cm},clip]{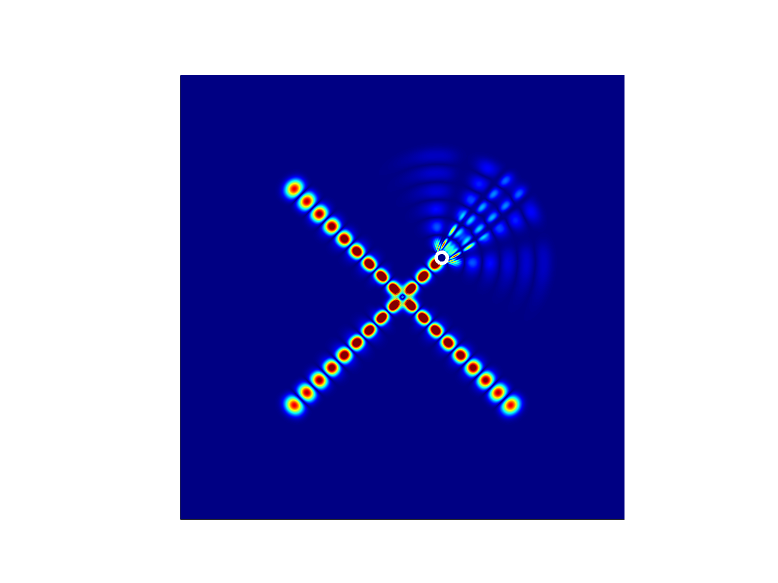}
	   	\includegraphics[width=2.5in,trim={2.5cm 0cm 1.75cm 0cm},clip]{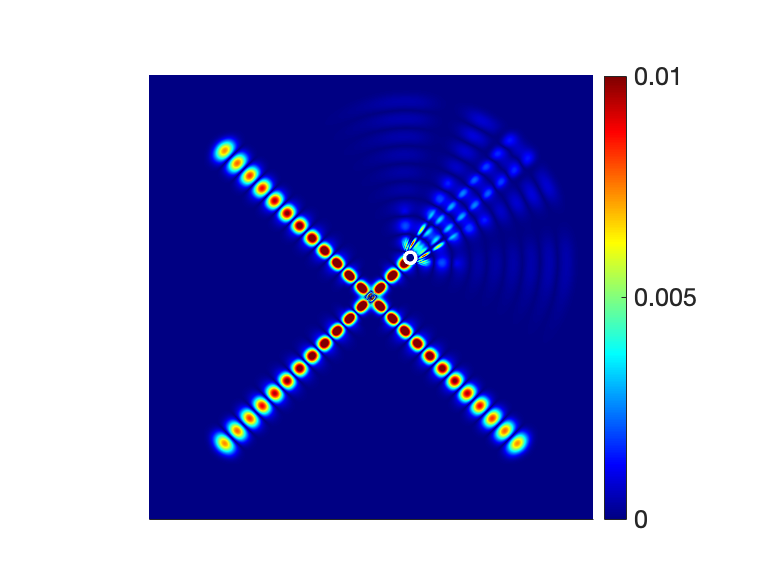}
	\end{adjustbox} 
	\begin{adjustbox}{max width=1.0\textwidth,center}
		\includegraphics[width=2.5in,trim={2.5cm 0cm 1.75cm 0cm},clip]{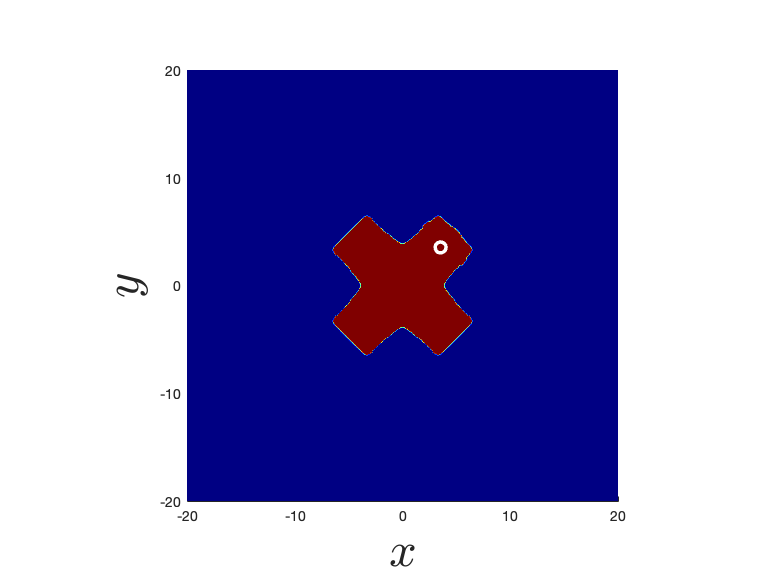}
		\includegraphics[width=2.5in,trim={2.5cm 0cm 1.75cm 0cm},clip]{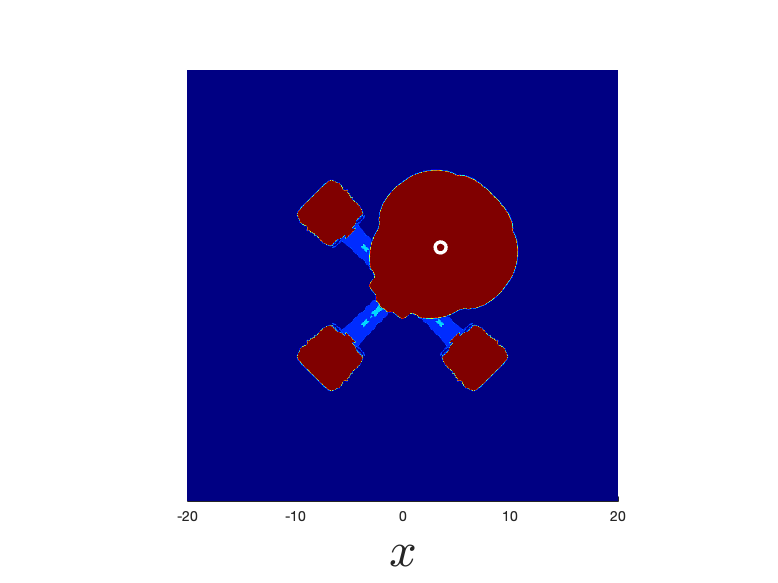}
		\includegraphics[width=2.5in,trim={2.5cm 0cm 1.75cm 0cm},clip]{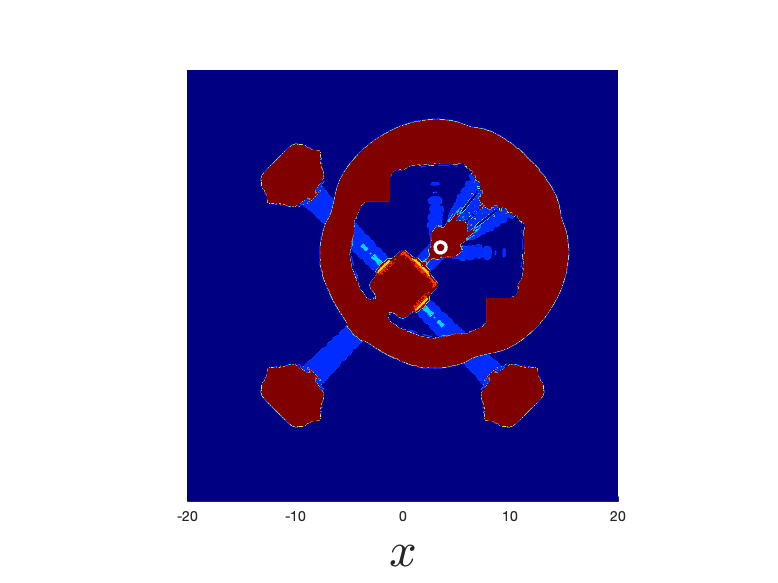}
	   	\includegraphics[width=2.5in,trim={2.5cm 0cm 1.75cm 0cm},clip]{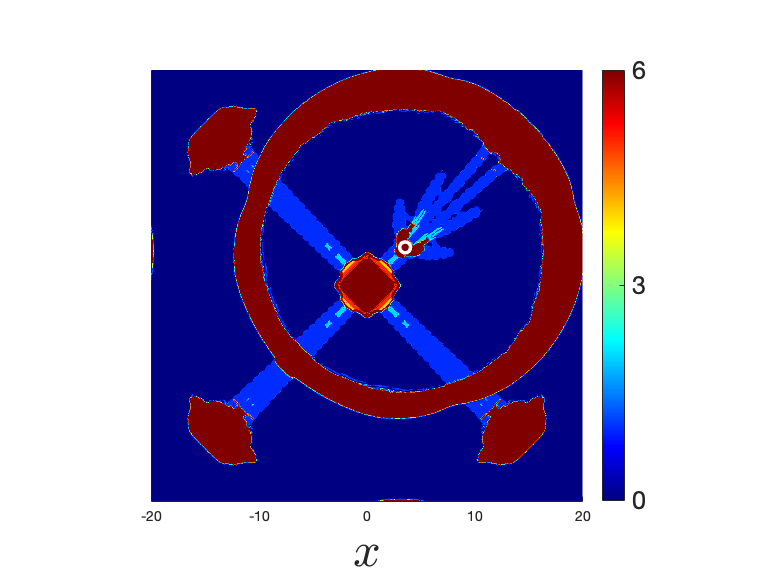}
	\end{adjustbox} 
       \caption{The evolution of $\|\hat{E}\|_2$ (top row), 
       	$|\hat{Q}|$ (middle row) and the distribution of $m$ (bottom row) for the airhole scattering problem using the 
	$p$-adaptive method. 
	From the left to the right, 
		the captions are take at times 7.5, 
		15, 
		22.5 and 30.
	The interface is illustrated by the white circle.}
       \label{fig:pblm_interface}
\end{figure}
%***************************************

The initial pulse for $\hat{H}_z$ leads to four pulses for the electric field,	
	one of which is interacting with the airhole.
The first row of Figure~\ref{fig:pblm_interface} illustrates the evolution of the electric field's norm and 
	we can clearly observe the scattering of the pulse when it reaches the airhole. 
The reaction of the nonlinear medium to the electromagnetic waves is observed in 
	the captions of $|\hat{Q}|$ in the middle row of Figure~\ref{fig:pblm_interface}.
The evolution of the values of $m$ by the $p$-adaptive method is illustrated in the bottom row of Figure~\ref{fig:pblm_interface}.
The method increases the value of $m$ where the pulses are located, 
	as expected.
The maximal value of $m$ is assigned in the vicinity of the interface by the $p$-adaptive algorithm due 
	to the reduction of order introduced by the diffusive interface model.
Nevertheless,
	the proposed $p$-adaptive Hermite method is able to accurately capture local features within the solution.

\section{Conclusion}

In this work, 
	we have proposed a $p$-adaptive Hermite method for simulating the response of a nonlinear dispersive medium
	to electromagnetic fields. 
A novel dissipative Hermite method is presented to handle Maxwell's equations for nonlinear dispersive media.
The Kerr and Raman types of nonlinearity are treated locally within each cell, 
	ultimately,
	requiring only the solution of small linear systems of equations,  
	whose dimension is independent of the order, 
	in each time step. 
The proposed method is able to achieve an arbitrary odd order accuracy in space and can be combined with any one-step method for the time 
	evolution. 
However, 
	the computational cost of the method scales as $\mathcal{O}(m^{2d})$ in $\mathbb{R}^d$. 
Therefore,	
	an order-adaptive algorithm was incorporated in the Hermite method to reduce considerably its computational cost and 
	efficiently capture local structures in the wave propagation. 
Numerical examples confirm the accuracy of the Hermite method (with and without order-adaptivity) and 
	suggest that the method is stable.
The $p$-adaptive Hermite method significantly reduces the computational cost in two space dimensions for a given mesh 
	and targeted accuracy.
The benefit of the proposed method should be 
	even greater in three space dimensions since the number of 
	degree of freedom is $(m+1)^d$ in $\mathbb{R}^d$ for each node and for each variable.

\section*{Declarations}
\section*{Funding} 
This work was supported in part by Grant NSF DMS-2208164, 
	DMS-2210286 and DMS-2309687.  
Any opinions, findings, and conclusions or recommendations expressed in this material are those of the authors and do not necessarily reflect the views of the NSF.
\section*{Conflicts of interest/Competing interests}
On behalf of all authors, 
	the corresponding author states that there is no conflict of interest.
\section*{Availability of data and material}
All data generated or analyzed during this study are included in this article.
%\section*{Code availability}
%The custom code used for this work is not yet publicly available.  	


\bibliographystyle{spmpsci}      % mathematics and physical sciences
\bibliography{references}   % name your BibTeX data base

%% Non-BibTeX users please use
%\begin{thebibliography}{}
%%
%% and use \bibitem to create references. Consult the Instructions
%% for authors for reference list style.
%\end{thebibliography}

\end{document}